\documentclass[11pt,reqno]{amsart}
\usepackage{amsmath,amsfonts,amssymb,amsthm,amscd,comment,euscript}
\usepackage[all]{xy}
\usepackage{graphicx}
\usepackage{mathptmx}
\usepackage{enumerate} 
\usepackage[colorlinks=true]{hyperref} 
\usepackage[usenames,dvipsnames]{xcolor}
\usepackage{enumitem}

 \calclayout

\numberwithin{equation}{section}

\DeclareMathOperator{\Fr}{Fr}

\swapnumbers
\theoremstyle{plain}
\newtheorem{lem}[equation]{Lemma}
\newtheorem{cor}[equation]{Corollary}
\newtheorem{prop}[equation]{Proposition}
\newtheorem{thm}[equation]{Theorem}
\newtheorem{thmm}[equation]{Theorem}
\newtheorem*{sublem}{Sublemma}

\theoremstyle{definition}
\newtheorem{ex}[equation]{Example}
\newtheorem{rem}[equation]{Remark}
\newtheorem{dfn}[equation]{Definition}
\newtheorem{conv}[equation]{Convention}

\newcommand{\Emb}{\operatorname{Emb}}
\newcommand{\w}{\wedge}
\newcommand{\ZF}{\mathbb{Z}F}
\newcommand{\op}{{\textrm{\rm op}}}
\newcommand{\Sm}{\mathbf{Sm}}
\newcommand{\A}{\mathbb{A}}
\newcommand{\PP}{\mathbb{P}}
\newcommand{\PPo}{\mathbb{P}^1}
\newcommand{\Ss}{\mathbb{S}}
\newcommand{\Os}{\mathcal{O}}
\newcommand{\Z}{\mathbb{Z}}
\newcommand{\Shv}{\mathbf{Shv}}
\newcommand{\Gm}{\mathbb{G}_m}
\newcommand{\G}{\mathbb{G}}

\newcommand{\GL}{\mathbf{GL}}
\newcommand{\Spec}{\operatorname{Spec}}

\newcommand{\gp}{\operatorname{gp}}
\newcommand{\colim}{\operatorname{colim}}
\newcommand{\uhom}{\underline{\operatorname{Hom}}}
\newcommand{\Hom}{\operatorname{Hom}}

\begin{document}

\title{Fibrant resolutions for motivic Thom spectra}

\author{Grigory Garkusha}
\address{Department of Mathematics, Swansea University, Fabian Way, Swansea SA1 8EN, UK}
\email{g.garkusha@swansea.ac.uk}

\author{Alexander Neshitov}
\address{Department of Mathematics, University of Western Ontario,
London, Ontario N6A 5B7, Canada}
\email{alexander.neshitov@gmail.com}




\begin{abstract}
Using the theory of framed correspondences developed by
Voevodsky~\cite{Voe2} and the machinery of framed motives introduced and developed in~\cite{GPMain},
various explicit fibrant resolutions for a motivic Thom spectrum $E$ are constructed in this paper.
It is shown that the bispectrum
   $$M_E^{\mathbb G}(X)=(M_{E}(X),M_{E}(X)(1),M_{E}(X)(2),\ldots),$$
each term of which is a twisted $E$-framed motive of $X$, introduced in the paper, represents $X_+\w E$
in the category of bispectra.
As a topological application, it is proved
that the $E$-framed motive with finite coefficients $M_E(pt)(pt)/N$, $N>0$, of the point $pt=\Spec k$
evaluated at $pt$ is a quasi-fibrant model of the topological $S^2$-spectrum $Re^\epsilon(E)/N$
whenever the base field $k$ is algebraically closed of
characteristic zero with an embedding $\epsilon:k\hookrightarrow\mathbb C$.
Furthermore, the algebraic cobordism spectrum $MGL$ is computed in terms of $\Omega$-correspondences
in the sense of~\cite{Omega}. It is also proved that $MGL$ is represented by a bispectrum each term of
which is a sequential colimit of simplicial smooth quasi-projective varieties.
\end{abstract}

\maketitle

\footskip30pt

\thispagestyle{empty} \pagestyle{plain}

\tableofcontents

\section{Introduction}

Voevodsky~\cite{Voe2} introduced framed correspondences in order to
suggest a new approach to stable motivic homotopy theory which will be more
amenable to explicit computations.
In~\cite{GPMain} the machinery of (big) framed motives is
developed converting the classical Morel--Voevodsky stable motivic homotopy theory
into a local theory of framed bispectra and yielding a new model for
$SH(k)$ in~\cite{GP5}. A key computation of~\cite{GPMain} is to give
explicit fibrant resolutions of the suspension spectra/bispectra of
smooth algebraic varieties.

The main results of this paper are concentrated around explicit computations
of motivic Thom spectra, decribed below, which play a central role in stable motivic homotopy theory.
We use computational miracles of Voevodsky's framed correspondences
to extend the machinery of framed motives that are of crucial importance in~\cite{GPMain}
to ``$E$-framed motives", where $E$ is a motivic Thom $T$-spectrum.

By definition, $E$ is called a {\it Thom spectrum}
if every space $E_n$ has the form
\[E_n=\colim_i E_{n,i}, \text{  }E_{n,i}=V_{n,i}/(V_{n,i}-Z_{n,i}),\]
where $V_{n,i}\to V_{n,i+1}$ is a directed sequence of smooth
varieties, $Z_{n,i}\to Z_{n,i+1}$ is a directed system of smooth
closed subschemes in $V_{n,i}.$ We say that a Thom spectrum $E$ {\it
has the bounding constant $d$\/} if $d$ is the minimal integer such
that codimension of $Z_{n,i}$ in $V_{n,i}$ is strictly greater than
$n-d$ for all $i,n$. If $E$ is also symmetric then it is said to be
a {\it spectrum with contractible alternating group action}, if for
any $n$ and any even permutation $\tau\in \Sigma_n$ there is an
$\A^1$-homotopy $E_n\to \uhom(\A^1,E_n)$ between the action of
$\tau$ and the identity map. In other words, $E$ neglects the action
of even permutations up to $\A^1$-homotopy. The most interesting
examples of such symmetric Thom spectra, all of which have the
bounding constant $d=1$, are given by the spectra $MGL$, $MSL$ or
$MSp$ (the latter two are regarded as $T^2$-symmetric spectra for
which the above definitions remain the same). These Thom spectra are of fundamental importance.
If we regard $E$ as a
$\PPo$-spectrum, denote by $\Theta^\infty(E)$ the standard
stabilization $\colim_n\uhom(\mathbb P^{\w n},E[n])$ of $E$. Taking
the Suslin complex at each level, we get a $\PPo$-spectrum
$C_*\Theta^\infty(E)$.

Our first computation (see Theorem~\ref{thmthom}) is as follows.

\begin{thmm}
Let $E$ be a Thom spectrum with the bounding constant $d$. Let
$C_*\Theta^{\infty}(E)^f$ be a spectrum obtained from
$C_*\Theta^{\infty}(E)$ by taking a level Nisnevich local fibrant replacement.
Then the spectrum $C_*\Theta^{\infty}(E)^f$ is motivically fibrant
starting from level $\max(0,d)$ and is stably equivalent to $E$.
\end{thmm}

If $E$ is a symmetric $T$-spectrum, then there is another natural
stabilization functor $\Theta^{\infty}_{sym}(E)$ (see
Definition~\ref{thetasymdef}). It is different from
$\Theta^{\infty}(E)$ and involves actions of certain permutations on
$E$. As above, we can take the Suslin complex at each level and form
a $\PPo$-spectrum $C_*\Theta^\infty_{sym}(E)$.

Given a Thom $T$-spectrum $E$, denote by $\Fr_n^E(X)$ the space
$\Fr_n^E(X)=\uhom(\PP^{\w n},X_+\w E_{n})$ and
$\Fr^E(X):=\colim_n\Fr^E_n(X)=\Theta^{\infty}(X_+\w E)_0$. By the
Voevodsky lemma~\ref{Voevlemma} $\Fr_n^E(X)$ and $\Fr^E(X)$ have an
explicit geometric description. We can similarly define the sheaves
$\Fr^E(T^i)$, $i\geqslant 0$. Altogether they form a $\PPo$-spectrum
$\Fr^E(S_T):=(\Fr^E(S^0),\Fr^E(T),\Fr^E(T^2),\ldots)$. As usual,
denote by $C_*\Fr^E(S_T)$ the $\PPo$-spectrum obtained from
$\Fr^E(S_T)$ by taking the Suslin complex levelwise.

The next computation (see Theorem~\ref{main}) gives the following
fibrant resolutions of $E$ (starting from some level).

\begin{thmm}
For a symmetric Thom $T$-spectrum $E$ with the bounding constant $d$
and contractible alternating group action the following
$\PPo$-spectra are isomorphic to $E$ in $SH(k)$ and motivically
fibrant starting from level $\max(0,d)$:
\begin{itemize}
\item $C_*\Fr^E(S_T)^f$
\item $C_*\Theta^{\infty}(E)^f$
\item $C_*\Theta^{\infty}_{sym}(E)^f$,
\end{itemize}
where ``$f$'' refers to levelwise Nisnevich local fibrant replacements of the
corresponding spectra.
\end{thmm}

Our next goal is to represent a Thom spectrum $E$ in the category of
$(S^1,\Gm^{\w 1})$-bispectra and construct an explicit fibrant
resolution for it. To this end, we introduce and study in
Section~\ref{sectionefr} $E$-framed motives of smooth algebraic
varieties $M_E(X)$, $X\in\Sm_k$. They are defined similarly to framed
motives introduced in~\cite{GPMain} and are explicit sheaves of
$S^1$-spectra.

The main result here (see Theorem~\ref{mainbisp}) is as follows.

\begin{thmm}\label{intro:main_thm}
Suppose $X\in \Sm_k$ and $E$ is a symmetric Thom $T$-spectrum with
the bounding constant $d$ and contractible alternating group action.

$(1)$ If $d=1$ then the $(S^1,\Gm^{\w 1})$-bispectrum
\[M_E^{\mathbb G}(X)_f:=(M_E(X)_f,M_E(X_+\w\Gm^{\w 1})_f,M_E(X_+\w\Gm^{\w 2})_f,\ldots)\]
is motivically fibrant and represents the $T$-spectrum $X_+\w E$ in
the category of bispectra, where ``$f$'' refers to stable local
fibrant replacements of $S^1$-spectra.

$(2)$ If $d<1$ then the $(S^1,\Gm^{\w 1})$-bispectrum
\[M_E^{\mathbb G}(X)_f:=(M_E(X)_f,M_E(X_+\w\Gm^{\w 1})_f,M_E(X_+\w\Gm^{\w 2})_f,\ldots)\]
is motivically fibrant and represents the $T$-spectrum $X_+\w E$ in
the category of bispectra, where ``$f$'' refers to level local
fibrant replacements of $S^1$-spectra.

$(3)$ If $d>1$ then the $(S^1,\Gm^{\w 1})$-bispectrum
\[\Omega_{S^{1}\w\Gm^{\w 1}}^{d-1}((M_{E[d-1]}(X)_f,M_{E[d-1]}(X_+\w\Gm^{\w 1})_f,
M_{E[d-1]}(X_+\w\Gm^{\w 2})_f,\ldots))\]
is motivically fibrant and represents the $T$-spectrum $X_+\w E$ in
the category of bispectra, where ``$f$'' refers to stable local
fibrant replacements of $S^1$-spectra. Here $E[d-1]$ stands for the
$(d-1)$-th shift of $E$ in the sense of Definition~\ref{shift}.
Another equivalent model for the $T$-spectrum $X_+\w E$ in the
category of bispectra is given by
\[\Omega_{S^{1}\w\Gm^{\w 1}}^{d-1}((M_{T^{d-1}\w E}(X)_f,
M_{T^{d-1}\w E}(X_+\w\Gm^{\w 1})_f,M_{T^{d-1}\w E}(X_+\w\Gm^{\w
2})_f,\ldots)).\] This bispectrum is motivically fibrant and ``$f$''
refers to stable local fibrant replacements of $S^1$-spectra.
\end{thmm}

One of the most impressive applications of the theory of framed
correspondences and the machinery of framed motives is that they
lead to computing explicit fibrant resolutions of classical
topological objects in terms of algebraic varieties. These
computations are far relatives for the celebrated constructions of
Pontrjagin~\cite{Pontr} who interpreted homotopy groups of spheres
in terms of smooth manifolds. For example, the classical topological
sphere spectrum is computed in~\cite{GPMain} as the framed motive
$M_{fr}(pt)(pt)$ of the point $pt=\Spec k$ evaluated at the point
whenever the base field $k$ is algebraically closed of
characteristic zero. We use the preceding theorem to get a similar
topological application in Theorem~\ref{fincoeff}. The main example
here concerns the motivic cobordism spectrum $MGL$ whose realization
is isomorphic to the topological complex cobordism spectrum $MU$.

\begin{thmm}
Let $k$ be an algebraically closed field of characteristic zero with
an embedding $\epsilon:k\hookrightarrow\mathbb C$. Suppose E is a
symmetric Thom $T$-spectrum with the bounding constant $d\leqslant
1$ and contractible alternating group action. Then for all integers
$N > 1$ and $n\in\mathbb Z$, the natural realisation functor
$Re^\epsilon:SH(k)\to SH$ in the sense of~\cite{PPR1} induces an
isomorphism
   $$\pi_n(M_E(pt)(pt);\mathbb Z/N)\cong\pi_n(Re^\epsilon(E);\mathbb Z/N)$$
between stable homotopy groups with $\textrm{mod\,} N$ coefficients.
\end{thmm}

Given a motivic Thom spectrum $E$ and $X\in Sm/k$, we fix any group completion 
$Fr^E(\Delta^\bullet_k,X)^{\gp}$ of the space $Fr^E(\Delta^\bullet_k,X)$, which is functorial in $X$. 
For instance, one can take $Fr^E(\Delta^\bullet_k,X)^{\gp}=\Omega_{S^1}Fr^E(\Delta^\bullet_k,X\otimes S^1)$. Put
   $$\pi^{E}_n(X):= \pi_n(Fr^E(\Delta^\bullet_k,X)^{\gp})$$
and call $\pi^{E}_n(X)$ the $n$-th singular algebraic $E$-homotopy group of $X$. 

The following result on the singular algebraic $E$-homotopy 
is an analogue of the celebrated theorem of
Suslin and Voevodsky~\cite{SV96} on the singular algebraic homology
(see Theorem~\ref{sve}):

\begin{thmm}
Suppose E is a
symmetric Thom $T$-spectrum with the bounding constant $d\leqslant
1$ and contractible alternating group action.
The assignment $X\mapsto\pi^{E}_{\ast}(X)$ is a
generalized homology theory on $Sm/\mathbb C$. Moreover, passing to
homotopy groups with finite coefficients, we get equalities
\[ \pi^{E}_n(X; \mathbb Z/m) = \pi_n(X(\mathbb{C})_+\wedge Re^\epsilon(E); \mathbb Z/m)\]
for all integers $n\geqslant 0$ and $m\neq 0$.

Also, the first part of this theorem is true over any perfect field $k$.
Namely, the assignment $X\mapsto \pi^{E}_{\ast}(X)$
is a generalized homology theory on the category $Sm/k$.
\end{thmm}

We can simplify $E$-framed motives further by removing a bit of
information in the definition of $E$-framed correspondences. In this
way we arrive at ``normally framed motives $\widetilde {M}_E(X)$" (see
Definition~\ref{frecolimit}). They play a pivotal role in our analysis and -- most importantly -- lead
to explicit computations of
the algebraic cobordism spectrum $MGL$ (see below).

We prove the following result (see Theorem~\ref{bisptilde})
computing $E$ in terms of normally framed motives.

\begin{thmm}
Suppose $X\in \Sm_k$ and $E$ is a symmetric Thom $T$-spectrum with
the bounding constant $d=1$ and contractible alternating group action.
Then we have a $(S^1,\Gm^{\w 1})$-bispectrum
\[\widetilde{M}_E^{\mathbb G}(X)_f:=(\widetilde{M}_E(X)_f,\widetilde{M}_E(X_+\w\Gm^{\w 1})_f,\widetilde{M}_E(X_+\w\Gm^{\w 2})_f,\ldots),\]
which is motivically fibrant and represents the $T$-spectrum $X_+\w E$ in
the category of bispectra, where ``$f$'' refers to stable local
fibrant replacements of $S^1$-spectra.
\end{thmm}

The last section is dedicated to further explicit models representing
the algebraic cobordism spectrum $MGL$ in the category of bispectra.
We first introduce Nisnevich sheaves
$\Emb(-,X)=\colim_n\Emb_n(-,X)$, $X\in\Sm_k$, where $\Emb_n(U,X)$ is
the set of couples $(Z,f)$ such that $Z$ is a closed l.c.i.
subscheme in $\A^n_U,$ finite and flat over $U$, and $f$ is a
regular map $f\colon Z\to X$. For $U,X\in\Sm_k$ we also denote by
$Cor_n^{\Omega}(U,X)$ the groupoid with objects given by the set
$\Emb_n(X,Y)$ whose morphisms between $(Z_1,f_1)$ and $(Z_2,f_2)$
are isomorphisms $\alpha\colon Z_1\to Z_2$ such that
$\pi_{Z_2}\alpha=\pi_{Z_1}$ and $f_2\alpha=f_1$, where $\pi_{Z_i}$
denotes the projection $\pi_{Z_i}\colon Z_i\to\A^n_X\to X.$ There
are natural stabilization maps $Cor_n^{\Omega}(-,X)\to
Cor_{n+1}^{\Omega}(-,X)$ induced by the natural inclusions
$\A^n_U\to\A^{n+1}_U.$ Denote by $Cor^{\Omega}(-,X)$ the colimit
$\colim_nCor_n^{\Omega}(-,X).$

In Theorem~\ref{mglembomega} we compute $M_{MGL}^{\G}(X)$ as follows.

\begin{thmm}
For $X\in\Sm_k$ there is a natural levelwise stable local equivalence between
$(S^1,\Gm^{\w 1})$-bispectra $M_{MGL}^{\G}(X)$ and
\[(C_*\Emb(X_+\wedge\Ss),C_*\Emb(X_+\wedge\Gm^{\wedge 1}\wedge\Ss),\ldots)\]
or
\[(C_*NCor^{\Omega}(X_+\wedge\Ss),C_*NCor^{\Omega}(X_+\wedge\Gm^{\wedge 1}\wedge\Ss),\ldots).\]
Here ``$N$" refers to the nerve of isomorphisms.
In particular, the $(S^1,\Gm^{\w 1})$-bispectra
\[(C_*\Emb(X_+\wedge\Ss)_f,C_*\Emb(X_+\wedge\Gm^{\wedge 1}\wedge\Ss)_f,\ldots)\]
and
\[(C_*NCor^{\Omega}(X_+\wedge\Ss)_f,C_*NCor^{\Omega}(X_+\wedge\Gm^{\wedge 1}\wedge\Ss)_f,\ldots)\]
are motivically fibrant and represent the $T$-spectrum $X_+\w MGL$ in
the category of bispectra, where ``$f$'' refers to stable local
fibrant replacements of $S^1$-spectra.
\end{thmm}

We finish the paper by the following  important computation (see
Theorem~\ref{reprmgl}) of  the algebraic cobordism in terms of
smooth quasi-projective varieties. This computation is an
application of the preceding theorem.

\begin{thmm}\
The $(S^1,\Gm^{\w 1})$-bispectrum $M_{MGL}^{\G}(X)$ is isomorphic in
$SH(k)$ to a bispectrum
$$(E^{X_+\wedge\Ss},E^{X_+\wedge\Gm^{\wedge 1}\wedge\Ss},\ldots),$$ each term of which is given by
a sequential colimit of simplicial smooth quasi-projective varieties
$E^{X_+\wedge\Gm^{\wedge i}\w S^j}$, $i,j\geqslant 0$.
\end{thmm}

Throughout the paper we denote by $\Sm_k$ the category of smooth
separated schemes of finite type over the base field $k$. We shall
assume that $k$ is perfect for the reason that the main
result of~\cite{GPPresheaves} (complemented by~\cite{DP} in
characteristic 2 and by~\cite[A.27]{DKO} for finite fields) says that over such fields for any
$\A^1$-invariant quasi-stable radditive framed presheaf of Abelian
groups $\mathcal F$, the associated Nisnevich sheaf $\mathcal
F_{nis}$ is strictly $\A^1$-invariant. By a motivic space we shall
mean a pointed simplicial Nisnevich sheaf on $\Sm_k$. If $\mathcal C$ is a category
cotensored over the category of pointed motivic spaces $\mathcal M$, we shall write
$\uhom(A,C)\in\mathcal C$ for the cotensor object associated with $A\in\mathcal M$ and $C\in\mathcal C$ unless it is specified otherwise.
We choose the flasque local/motivic model structures
on motivic spaces (respectively $S^1$- or $\mathbb P^1$-spectra of
motivic spaces) in the sense of~\cite{Is}.

\subsubsection*{Relations to other works.}
This paper (first appeared in the archive in April 2018) depends on
a series of papers on framed
motives~\cite{AGP,GNP,GPPresheaves,GPMain}. Computations of motivic
Thom spectra like those of Theorem~\ref{mainbisp} in terms of
tangentially framed correspondences as defined in~\cite{EHKSY} were
later obtained in~\cite{EHKSY1}. Our approach is based on
Voevodsky's framed correspondences~\cite{Voe2}. Technique developed
in Sections~\ref{thetasym} and~\ref{thespectrum} is crucial for the
theory of motivic $\Gamma$-spaces~\cite{GPO}, an extension of the
celebrated Segal machine of $\Gamma$-spaces~\cite{S} to the world of
motivic homotopy theory. A systematic study of normally framed
correspondences associated with Thom spectra is given in
Section~\ref{normally}. This type of correspondences associated with
the motivic sphere spectrum is of great utility in~\cite{AN,EHKSY}.
Normally framed correspondences lead to representability of some
important motivic Thom spectra like $MGL$ by schemes (see
Theorem~\ref{reprmgl}). The representability theorem is also proven
in~\cite{EHKSY1}.

\section{Preliminaries}

In this section we collect basic facts about spectra and motivic spaces with framed correspondences.

\subsubsection*{Spectra of Thom type}

\begin{dfn}\label{cstar}
For every space $X$ denote by $C_*X$ its Suslin complex. It is the
diagonal of the bisimplicial sheaf $(n,m)\mapsto
\uhom(\Delta^n_k,X_m)$ where $X_m$ is the sheaf of $m$-simplices of
$X$.
\end{dfn}

\begin{dfn}
Given two spaces $X,Y$ and maps $f,g\colon X\to Y$,

\begin{itemize}
\item a {\it simplicial homotopy\/} between $f$ and $g$ is a map
$H\colon X\w\Delta[1]_+\to Y$ such that the composition $Hi_0=f$ and
$Hi_1=g$, where $i_0,i_1\colon X\to X\w\Delta[1]_+$ are the face
maps;

\item an {\it $\A^1$-homotopy\/} between $f$ and $g$ is a map
$H\colon X\to\uhom(\A^1,Y)$ such that $i_0^*H=f$ and $i_1^*H=g$,
where $i_0,i_1\colon \uhom(\A^1,Y)\to Y$ are maps induced by zero
and unit embeddings of $pt$ into $\A^1$.
\end{itemize}
\end{dfn}

\begin{rem}\label{a1sim}
Every $\A^1$-homotopy $H\colon X\to\uhom(\A^1,Y)$ between $f$ and
$g$ gives rise to a simplicial homotopy $H'\colon
C_*X\w\Delta[1]_+\to C_*Y$ between $f$ and $g$.
\end{rem}

\begin{conv} We shall use the following notation:

\begin{itemize}
\item Given two motivic spaces $B$ and $C$, we denote by $tw$ the twist isomorphism $C\w B\xrightarrow{\cong} B\w C$.
\item For brevity, we shall sometimes write $(A,B)$ to denote $\uhom(A,B)$, where $A$ and $B$ are motivic spaces.
We shall use the canonical map
\[(A,B)\w C\to (A,C\w B)\]
which is adjoint to
\[(A,B)\xrightarrow{C\w id}(C\w A, C\w B) = (C,(A,C\w B).\]
When $C$ and $B$ are distinct spaces we shall often compose the
previous map with the twist isomorphism $tw\colon C\w B\to B\w C$ to get the map
\[(A,B)\w C\to (A,B\w C).\]

\item If there is no likelihood of confusion, we shall use the
equality sign $\PP^{\w m}\w\PP^{\w n}=\PP^{\w n}\w\PP^{\w m}$
for the associativity isomorphism
   \[\PP^{\w m}\w\PP^{\w n}\cong\PP^{\w m+n}=\PP^{\w n+m}\cong\PP^{\w n}\w\PP^{\w m}.\]

\item For any $m,n$ we shall identify the spaces (via associativity isomorphisms)
   \[(\PP^{\w m},(\PP^{\w n},X))=(\PP^{\w m}\w\PP^{\w n},X)=(\PP^{\w n}\w\PP^{\w m},X)=(\PP^{\w n},(\PP^{\w m},X)).\]
\end{itemize}
\end{conv}

Let $T$ be the pointed Nisnevich sheaf $\A^1/(\A^1-0)$. A {\it
$T$-spectrum\/} is a sequence of spaces $E_n$ together with bonding
maps, denoted by $u$. In what follows we work with right spectra,
and so each bonding map is a map $u\colon E_n\w T\to E_{n+1}$.
Denote by $\Sigma_n$ the $n$th symmetric group. A {\it symmetric
$T$-spectrum\/} is a spectrum $E$ together with a left action of
$\Sigma_n$ on $E_n$ such that the bonding maps satisfy the relevant
equivariance properties.

\begin{dfn}
Given $\tau\in \Sigma_n$ we shall write $\tau=(\tau(1),\ldots,
\tau(n))$. The reader should not confuse this notation with cyclic
permutations. For any $n,m$, denote by $\chi_{n,m}\in \Sigma_{n+m}$
the obvious shuffle permutation
$\chi_{n,m}=(n+1,\ldots,n+m,1,\ldots, n)$.
\end{dfn}

If we denote by $S_T$ the symmetric motivic sphere $T$-spectrum
$(S^0, T, T^2,\ldots)$, then
any symmetric $T$-spectrum is a right module over the monoid $S_T$
in the category of symmetric sequences~\cite[7.2]{H}.

\begin{dfn}\label{shift}
Given a symmetric $T$-spectrum $E$ and $n\geqslant 0$, denote by
$u_l\colon T\w E_n\to E_{n+1}$ the composition
\[T\w E_n\xrightarrow{tw} E_n\w T\xrightarrow{u}E_{n+1}\xrightarrow{\chi_{n,1}} E_{1+n}.\]
Observe that the maps give a map of symmetric $T$-spectra
$u_l:T\wedge E\to E[1]$. Here $E[1]$ is the {\it shift symmetric
spectrum\/} whose spaces are given by $E[1]_n=E_{1+n}$ with action
of $\Sigma_n$ by restriction of the $\Sigma_{1+n}$-action on
$E_{1+n}$ along the obvious embedding $\Sigma_n\hookrightarrow
\Sigma_{1+n}$ taking $\tau\in\Sigma_n$ to
$1\oplus\tau\in\Sigma_{1+n}$. The structure maps of $E[1]$ are the
reindexed structure maps for $E$. In turn, $T\wedge E$ is the {\it
suspension spectrum\/} of $E$ whose spaces are defined as $(T\wedge
E)_n=T\wedge E_n$. The symmetric group $\Sigma_n$ acts on $T\wedge
E_n$ through the given action on $E_n$ and trivially on $T$. Each
structure map is the composite
   $$(T\wedge E)_n\wedge T\cong T\wedge(E_n\wedge T)\xrightarrow{id\wedge u_n}(T\wedge E)_{n+1}.$$
\end{dfn}

\begin{dfn}
A symmetric spectrum $E$ is said to be a {\it spectrum with
contractible alternating group action}, if for any $n$ and any even
permutation $\tau\in \Sigma_n$ there is an $\A^1$-homotopy $E_n\to
\uhom(\A^1,E_n)$ between the action of $\tau$ and the identity map.
In other words, $E$ neglects the action of even permutations up to
$\A^1$-homotopy.
\end{dfn}

\begin{dfn}\label{thomspectra} A $T$-spectrum $E$ is called a {\it Thom spectrum}
if every space $E_n$ has the form
\[E_n=\colim_i E_{n,i}, \text{  }E_{n,i}=V_{n,i}/(V_{n,i}-Z_{n,i}),\]
where $V_{n,i}\to V_{n,i+1}$ is a directed sequence of smooth varieties, $Z_{n,i}\to Z_{n,i+1}$
is a directed system of smooth closed subschemes in $V_{n,i}.$

We shall say that a Thom spectrum $E$ {\it has the bounding constant
$d$\/} if $d$ is the minimal integer such that codimension
of $Z_{n,i}$ in $V_{n,i}$ is strictly greater than $n-d$ for all $i,n$.

\end{dfn}

\begin{ex}\label{ThomSpectraex}
The suspension spectrum $\Sigma^{\infty}X_+=X_+\w S_T$ and the
algebraic cobordism spectrum ${MGL}$ of~\cite{VoevodskyICM} (see
also~\cite{PPR,PW}) are examples of symmetric Thom spectra with the
bounding constant $1$ and contractible alternating group action.

If $E$ is a Thom $T$-spectrum with the bounding constant $d\geqslant
0$, then its $n$th shift $E[n]$ (see Definition~\ref{shift}) as well
as the spectrum $T^n\wedge E$, $n\geqslant 0$, is a Thom spectrum
with the bounding constant $d-n$. In turn, its negative shift
$E[-n]=(*,\ldots,*,E_0,E_1,\ldots)$ having $E_0$ in the $n$th entry
is a Thom $T$-spectrum with the bounding constant $n+d$. By
definition, the trivial Thom spectrum $*=(*,*,\ldots)$ has the
bounding constant $+\infty$.

In practice we also deal with symmetric Thom $T^2$-spectra like
$MSL$ or $MSp$ (see~\cite{PW} for definitions). 
We also say that a Thom $T^2$-spectrum $E$
{\it has the bounding constant $d$\/} if $d$ is the minimal integer
such that codimension of $Z_{n,i}$ in $V_{n,i}$ is strictly greater
than $2n-d$ for all $i,n$. The Thom $T^2$-spectra $MSL$ and $MSp$
have the bounding constant $d=1$.

By construction (see~\cite{PW}), the  action of the symmetric group
$\Sigma_n$ on the spaces $MSL_{2n}$ and $MSp_{2n}$ factors through
the action of $SL_{2n}$ and $Sp_{2n}$ respectively. Since $SL_{2n}$
and $Sp_{2n}$ are semisimple simply-connected groups, the sets of
$k$-points $SL_{2n}(k)$ and $Sp_{2n}(k)$ are generated by the root
subgroups $U_{\alpha}(k)$~(see \cite{PraRa}). Since every root
subgroup $U_{\alpha}$ is isomorphic to the affine line $\A^1_k$, we
have that for every element $A$ of $G=SL_{2n}$ or $G=Sp_{2n}$ there
exists a map $h\colon\A^1_k\to G$ such that $h(0)=I,h(1)=A$. It
follows that $MSL$ and $MSp$ are $T^2$-spectra with contractible
alternating group action as well.
\end{ex}

\begin{lem}\label{barbashev}
Let $\Sigma_n\to\GL_n(k)$ be the standard inclusion and let $\tau$ be an even
permutation. Then there is an $\A^1$-curve $L\colon\A^1_k\to\GL_n$
such that $L(0)=I$ is the identity matrix and $L(1)=\tau.$
\end{lem}

\begin{proof}
Since $\tau$ is even, its image belongs to $SL_n(k)$. Thus it 
can be written as a product of elementary matrices:
\[\tau = \prod_{l=1}^m e_{i_l, j_l}(\lambda_l),
\quad\text{ where } 1\leqslant i_l,j_l\leqslant n,\, \lambda_l\in k,\, i_l\neq j_l.\] 
Here an elementary matrix $e_{i,j}(\lambda)$ is a matrix with all its diagonal elements equal to 1, 
$\lambda$ being placed in the $(i,j)$-th entry and zero elsewhere. Then 
$L(t)=\prod_{l=1}^m e_{i_l, j_l}(t\lambda_l)$ defines a regular 
map $\A^1_k\to\GL_n(k)$ with $L(0)=I$ and $L(1)=\tau.$
\end{proof}

\emph{In order to avoid a heavy presentation, from now on we shall
deal with Thom $T$-spectra only. The interested reader will be able
to prove the relevant results for Thom $T^2$-spectra as well.
}

\begin{dfn}
There is a functorial fibrant replacement of motivic spaces $X\to
X^f$ in the flasque Nisnevich local model structure (e.g. given by
controlled fibrant models in the sense
of~\cite[Section~1.2]{JardineMSS}) such that for any $\PPo$-, $T$-
or $S^1$-spectrum $E=(E_0,E_1,\ldots)$ the sequence
$E^f=(E_0^f,E_1^f,\ldots)$
can be canonically equipped with a structure of a spectrum and $E\to
E^f$ is a map of spectra.
\end{dfn}

\subsubsection*{The Voevodsky Lemma}

One of the key facts in the theory of framed correspondences is the
following lemma of Voevodsky that computes Hom-sets between certain
Nisnevich sheaves. Its proof can be found
in~\cite[Section~3]{GPMain}.

\begin{lem}[Voevodsky's Lemma]\label{Voevlemma}
For $X,Y\in\Sm_k$ and a closed subset $X'$ of $X$ and open subset $V$ of $Y$ the set
   \[\Hom_{\Shv_{\bullet}}(X/X',Y/V)\]
is in a natural bijection with the set of equivalence classes of
triples $(U,Z,\phi)$, where $Z$ is a closed subset of $X$ disjoint
with $X'$, $U$ is an \'etale neighborhood of $Z$ in $X$ and
$\phi\colon U\to Y$ is a regular map such that $\phi^{-1}(Y-V)=Z$.
By definition, two triples $(U,Z,\phi)$ and $(U',Z',\phi')$ are {\it
equivalent\/} if $Z=Z'$ and $\phi,\phi'$ coincide on some common
etale neighbourhood of $Z$ in $X$.
\end{lem}

\begin{cor}\label{Pmpermut}
For any Thom spectrum $E$ there is a natural isomorphism of motivic
spaces
   \[\uhom(\PP^{\w m},E_n)\cong\uhom(\PP^m/\PP^{m-1},E_n).\]
As a consequence, the $\Sigma_m$-action on $\uhom(\PP^{\w m},E_n)$
permuting factors of $\PP^{\w m}$ can be extended to an action of
$\GL_m(k)$ (it naturally acts on $\PP^{m}/\PP^{m-1}$), and thus for
any even permutation $\tau\in \Sigma_m$ there is an $\A^1$-homotopy
between the action of $\tau$ and the identity map of $\uhom(\PP^{\w
m},E_n)$ by Lemma~\ref{barbashev}.
\end{cor}

\begin{proof}
By Definition~\ref{thomspectra} $E_n=\colim_i E_{n,i}$, where
$E_{n,i}=V_{n,i}/(V_{n,i}-Z_{n,i})$. For any $X\in\Sm_k$
Lemma~\ref{Voevlemma} implies both sets $\Hom(X_+\w\PP^{\w
m},E_{n,i})$ and $\Hom(X_+\w\PP^m/\PP^{m-1},E_{n,i})$ are naturally
isomorphic. They are described up to isomorphism of sets
(see~\cite[Section~3]{GPMain}) as the equivalences classes of
triples $(U,Z,\phi)$, where $Z$ is a closed subset of $\A^m_X,$
finite over $X$, $U$ is its \'etale neighborhood and $\phi\colon
U\to V_{n,i}$ is such that $\phi^{-1}(Z_{n,i})=Z$.
\end{proof}

\begin{dfn}\label{framesdef}
Following~\cite{GPMain} for $X,Y\in\Sm_k$ and an open subscheme $U$ of $Y$, we set
    \[\Fr_n(X,Y/U):=\Hom_{\Shv_{\bullet}}(X_+\w\PP^{\w n},(Y/U)\w T^n).\]
$\Fr_n(X,Y/U)$ is pointed at the empty correspondence or, equivalently, at the zero map.
By smashing the elements of $\Fr_n(X,Y/U)$ with the canonical motivic equivalence
$\sigma:\PP^{\w 1}\to T$, we get a map of pointed sets $\Fr_n(X,Y/U)\to\Fr_{n+1}(X,Y/U)$. Denote by
   \[\Fr(X,Y/U):=\colim_n(\cdots\to\Fr_n(X,Y/U)\to\Fr_{n+1}(X,Y/U)\to\cdots).\]
We shall also write $C_*\Fr(Y/U)$ to denote the Suslin complex associated to the Nisnevich sheaf
$X\mapsto\Fr(X,Y/U)$ (see Definition~\ref{cstar}).

More generally, we can define the sets $\Fr_n(X,\mathcal
G):=\Hom_{\Shv_{\bullet}}(X_+\w\PP^{\w n},\mathcal G\w T^n)$,
$\Fr(X,\mathcal G)$ for every pointed Nisnevich sheaf $\mathcal G$
as well as the Suslin complex $C_*\Fr(\mathcal G)$ associated to the
Nisnevich sheaf $X\mapsto\Fr(X,\mathcal G)$.

Below we shall often deal with sheaves of the form
   $$\uhom(\PP^{\w i},\Fr_n(Y/U)),\quad i,n\geqslant 0.$$
By Voevodsky's Lemma its value at $X\in \Sm_k$ consists of the
triples $(W,Z,\phi)$, where $Z$ is a closed subset of $\A^{i+n}_X$,
finite over $X$, $W$ is its \'etale neighborhood and $\phi\colon
W\to Y$ is such that $\phi^{-1}(Y-U)=Z$.
\end{dfn}

\begin{prop}[Additivity Theorem]\label{additivity}
Suppose $X,X'\in\Sm_k$
and $Z,Z'$ are closed subsets of $X$ and $X'$ respectively. Denote
$Y=X/(X-Z),Y'=X'/(X'-Z').$ Then for every $i\geqslant 0$ the canonical
map
\[\uhom(\PP^{\w i},C_*\Fr(Y\vee Y'))\to\uhom(\PP^{\w i},C_*\Fr(Y))\times \uhom(\PP^{\w i},C_*\Fr(Y'))\]
is a schemewise weak equivalence.
\end{prop}

\begin{proof}
The proof is like that of the Additivity Theorem of~\cite{GPMain}.
\end{proof}

\begin{cor}\label{specialgamma}
Let $\Gamma^{\op}$ be the category of finite pointed sets and
pointed maps. Under the notation of Proposition~\ref{additivity} the
association
   $$K\in\Gamma^{\op}\mapsto\uhom(\PP^{\w i},C_*\Fr(Y\wedge K))$$
is a special $\Gamma$-space in the sense of Segal~\cite{S}. As a
result, the Segal $S^1$-spectrum
   $$\uhom(\PP^{\w i},M_{fr}(Y))=\uhom(\PP^{\w i},C_*\Fr(Y\wedge\mathbb S))$$
is sectionwise positively fibrant. Here $\mathbb
S=(S^0,S^1,S^2,\ldots)$ is the sphere spectrum and
$M_{fr}(Y):=C_*\Fr(Y\wedge\mathbb S)$ is the framed motive of $Y$ in
the sense of~\cite{GPMain}.
\end{cor}

\section{The functor $\Theta^{\infty}$ and the layer
filtration}

If there is no likelihood of confusion, we shall often regard
$T$-spectra as $\PPo$-spectra by means of the canonical motivic
equivalence $\sigma:\mathbb P^{\wedge 1}\to T$. Given a $T$-spectrum
$E$, denote by $u_*\colon E_i\to\uhom(\mathbb P^{\wedge 1},E_{i+1})$
the adjoint to the bonding map.
Following Jardine~\cite[\S 2]{JardineMSS}, we give the following
definition.

\begin{dfn}
Denote by $E\w T$ the {\it fake $T$-suspension spectrum\/} with
terms $(E\w T)_i=E_i\w T$ and bonding maps given by
\[(E_i\w T)\w T\xrightarrow{u\w T} E_{i+1}\w T.\]
It is important to note that the bonding maps do not permute two
copies of $T$ on the left.
Denote by $\Omega^\ell(E)$ the {\it fake loop $\mathbb
P^1$-spectrum\/} with terms $\Omega^\ell(E)_i=\uhom(\PPo,E_i)$ and
bonding maps adjoint to
\[\uhom(\mathbb P^{\wedge 1},E_i)\xrightarrow{u_*}\uhom(\mathbb P^{\wedge 1},\uhom(\mathbb P^{\wedge 1},E_{i+1})).\]
We notice again that two copies of $\mathbb P^{\wedge 1}$ on the
right are not permuted.
\end{dfn}

\begin{dfn}
We denote by $E[1]$ the {\it shifted $T$-spectrum\/} $E[1]_i=E_{i+1}.$
Its bonding maps $u\colon E_i\w T\to E_{i+1}$ induce a
map of $T$-spectra $u\colon E\w T\to E[1]$.
\end{dfn}

\begin{dfn}
Denote by $\Theta^1 (E)=\Omega^\ell(E[1])$. By adjointness there is a canonical map of $\mathbb P^1$-spectra
\[E\to\Omega^\ell(E\w T)\to\Omega^\ell(E[1])=\Theta^1(E).\]
Denote by $\Theta^n(E)$ the $n$-fold composition
$\Theta^1(\Theta^1\ldots (E))$. There are natural stabilization maps
$\Theta^n(E)\to\Theta^{n+1}(E)$ and $\Theta^{\infty}(E)$ denotes the
colimit
\[\Theta^{\infty}(E)=\colim_n\Theta^n(E).\]
\end{dfn}

\begin{rem}
We shall need the following explicit description of spaces of the
$\mathbb P^1$-spectrum $\Theta^nE$ and its bonding maps. The $j$th
space equals
\[\Theta^n(E)_j=\uhom(\PP^{\w n},E_{j+n}).\]
The bonding maps of $\Theta^n(E)$ are adjoint to
\[(\PP^{\w n},E_{n+j})\xrightarrow{u_*}(\PP^{\w n},(\PP^{\w 1},E_{n+j+1}))
=(\PP^{\w n}\w\PP^{\w 1},E_{n+j+1})=(\PP^{\w 1},(\PP^{\w n},E_{n+j+1})).\]
One should note that we do not permute copies of $\PP^{\w 1}$ here.
The stabilization map $\Theta^n(E)_j\to\Theta^{n+1}(E)_j$ can be described as the composite map
\[\uhom(\PP^{\w n},E_{j+n})\xrightarrow{-\wedge\sigma}\uhom(\PP^{\w n}\w\PP^{\w 1},E_{j+n}\w T)
\stackrel{u}\to\uhom(\PP^{\w n+1},E_{j+n+1}).\] Here the left arrow
smashes the simplices of the left space with $\sigma:\mathbb
P^{\wedge 1}\to T$.
\end{rem}

\begin{lem}\label{thetaomega}
For any spectrum $E$ the adjoint of each bonding map
   $$\Theta^{\infty}(E)_i\to\uhom(\mathbb P^{\wedge 1},\Theta^{\infty}(E)_{i+1})$$
 in the spectrum $\Theta^{\infty}E$ is an isomorphism.
\end{lem}

\begin{proof}
For every $n$ and $i$ the adjoint of the bonding map
   \[\Theta^n(E)_i\to\uhom(\mathbb P^{\wedge 1},\Theta^n(E)_{i+1})=\uhom(\mathbb P^{\wedge n+1},E_{i+n+1})=\Theta^{n+1}(E)_i\]
coincides with the stabilization map $\Theta^n(E)_i\to\Theta^{n+1}(E)_{i}$.
Thus we get an isomorphism of sequences $\Theta^{n+1}(E)_i$ and $\uhom(\mathbb P^{\wedge 1},\Theta^{n}(E)_{i+1})$.
\end{proof}

\begin{dfn}
Given a $T$-spectrum $E$, define its {\it $n$-th layer $L_n(E)$\/}
as the $T$-spectrum
   \[L_nE=(E_0,E_1,\ldots, E_n, E_n\w T, E_n\w T^2,\ldots).\]
The maps of spaces $E_n\w T^i=E_n\w T\w T^{i-1}\xrightarrow{u\wedge
T^{i-1}} E_{n+1}\w T^{i-1}$ induce maps of spectra $L_nE\to
L_{n+1}E$ and an obvious isomorphism of spectra $E\cong\colim_n
L_nE$.
\end{dfn}

We recall the following lemma from~\cite[Section~13]{GPMain}. It
says that $\Theta^\infty$ converts $T$-spectra into framed $\mathbb
P^1$-spectra, i.e. spectra whose spaces are spaces with framed
correspondences.

\begin{lem}\label{thetalayer}
For any $T$-spectrum $E$ there is a canonical isomorphism of
$\mathbb P^1$-spectra
\[\Theta^{\infty}E=\colim_n\Theta^{\infty}(L_nE).\]
Moreover, there is a canonical isomorphism of spaces
$\Theta^{\infty}(L_nE)_i=\uhom(\PP^{\w n},\Fr(E_n\w T^i)).$
\end{lem}

\begin{proof}
Note that $\Theta^m(E)=\colim_n\Theta^m(L_nE)$ for every $m$. Thus
passing to the colimit over $m$, we get the first isomorphism. For
the second statement note that the left hand side is the colimit
over $m$ of the sequence $(\PP^{\w m+n},E_n\w T^{m+i})\to (\PP^{\w
m+n}\w\PP^{\w 1},E_n\w T^{m+i}\w T)$. Thus the identification
$(\PP^{\w m+n},E_n\w T^{m+i})=(\PP^{\w n},(\PP^{\w m},E_n\w T^i\w
T^m))$ (we do not permute copies of $T$ here) provides its
isomorphism with the sequence $(\PP^{\w n},\Fr_m(E_n\w T^i))$.
\end{proof}

\section{The Mayer--Vietoris sequence}

\begin{dfn}
Suppose $X\in\Sm_k$ and $Z$ is a smooth closed subvariety of
codimension $d$. We say that the embedding $Z\to X$ is {\it
trivial\/} if there is an \'etale map $\alpha\colon X\to\A^{n+d}$
such that $Z=X\times_{\A^{n+d}}\A^n$, where $\A^n\to\A^{n+d}$ is the
standard linear embedding.
\end{dfn}

\begin{lem}\label{cover}
For every closed embedding of smooth varieties $Z\to X$ of codimension $d$ there is an
open cover of $X$ by $X_i$ such that the inclusion $X_i\cap Z\to
X_i$ is trivial.
\end{lem}

\begin{proof}
Let $n$ be the dimension of $Z$ and let $x\in Z$ be a closed point. Since the embedding $Z\to X$ is regular, 
there is an open affine neighborhood $X'$ of $x$ in $X$ such that $Z'=X'\cap Z$ is the zero locus of 
$d$ regular functions $f_1,\ldots, f_d$. These regular functions define a map $f\colon X'\to\A^d$ that
is flat by~\cite[Tag 00R4]{stacks-project}. Since $Z'$ is smooth,~\cite[Tag 01V9]{stacks-project} 
implies that $f$ is smooth at $x$, and by~\cite[Tag 054L]{stacks-project} there is an affine open 
neighborhood $U$ of $x$ in $X'$ such that $f\colon U\to\A^d$ can be presented as a composition of 
an \'etale map followed by a projection: $f\colon U\xrightarrow{\alpha}\A^n_{\A^d}\to\A^d$. 
Then $U\cap Z=U\times_{\A^d}\{0\}=U\times_{\A^{n+d}}\A^n$, and so the embedding
$U\cap Z\to U$ is trivial.
\end{proof}

\begin{lem}\label{ZTd}
If $Z\to X$ is a trivial embedding of codimension $d$, then there is an isomorphism of
Nisnevich sheaves $X/(X-Z)\xrightarrow{\cong} Z_+\w T^d$.
\end{lem}

\begin{proof}
Fix an \'etale map $\alpha\colon X\to\A^{n+d}$ such that
$Z=X\times_{\A^{n+d}}\A^n.$ Consider the projection
$pr_n\colon\A^{n+d}\to \A^n$ and the composition
$pr_n\circ\alpha\colon X\to\A^n$. Then $X'=Z\times_{\A^n} X$ is an
etale neighborhood of $Z$ in $X$. Moreover, $X'$ and $X-Z$ is an
elementary Nisnevich cover of $X$, hence $X'/(X'-Z)\to X/(X-Z)$ is
an isomorphism of Nisnevich sheaves. Then the composition
\[X'= Z\times_{\A^n}X\to Z\times_{\A^n}\A^{n+d}=Z\times_k\A^d\]
is an \'etale neighborhood of $Z=Z\times 0$ in $Z\times_k\A^d$, and
thus it induces an isomorphism $X'/(X'-Z)\to
Z\times\A^d/(Z\times\A^d-Z\times 0)=Z_+\w T^d.$
\end{proof}

\begin{lem}\label{1219}
Suppose $G$ is a strictly homotopy invariant Nisnevich sheaf, then
for any bounded chain complex of presheaves $X$ there is an
isomorphism
\[\Hom_{D^-_{Nis}}(Tot(C_*(X))_{Nis},G[n])\cong \Hom_{D^-_{Nis}}(X_{Nis},G[n])\]
\end{lem}

\begin{proof}
Consider the stupid truncation $\sigma_{\geqslant i}X$. Then there
is a short exact sequence of complexes of presheaves
\[0 \to C_*X_i[-i]\to Tot(C_*(\sigma_{\geqslant i}X))\to Tot(C_*(\sigma_{\geqslant i+1}X))\to 0\]
Note that $\Hom_{D^-_{Nis}}((C_*X_i)_{Nis},G[n])\cong
\Hom_{D^-_{Nis}}((X_i)_{Nis},G[n])$ by~\cite[Prop. 12.19]{MVW}, and
$\sigma_{\geqslant N_0}X=0$ and $\sigma_{\geqslant N_1}X=X$ for some
$N_0,N_1$. Then the statement follows by induction.
\end{proof}

\begin{lem}\label{Totzero}
Suppose $F$ is a bounded complex of $\ZF_*$-presheaves such that
$F_{Nis}$ is quasi-isomorphic to zero and the homology presheaves
$H_i(Tot(C_*F))$ are quasi-stable. Then the complex of sheaves
$(Tot(C_*F))_{Nis}$ is locally quasi-isomorphic to zero.
\end{lem}

\begin{proof}
The presheaves $H_i(Tot(C_*F))$ are quasi-stable and homotopy
invariant. By~\cite{GPPresheaves} the associated sheaves
$H_i=H_i(Tot(C_*F))_{Nis}$ are strictly homotopy invariant, and
hence by Lemma~\ref{1219} there is an isomorphism
   \[\Hom_{D_{Nis}^-}((Tot(C_*F))_{Nis},H_i[n])\cong \Hom_{D_{Nis}^-}(F_{Nis},H_i[n])=0.\]
The inductive argument as in the proof of~\cite[13.12]{MVW} gives a map
$(Tot(C_*F))_{Nis}\to H_i[i]$ inducing an isomorphism on homology sheaves. It is zero by the above arguments,
hence $H_i=0.$
\end{proof}






Suppose $Z\subseteq X$ is a closed subset of $X$, and $X=X_1\cup
X_2$ is a Zariski cover of $X$. Denote by $X_{12}=X_1\cap X_2,
Z_1=X_1\cap Z,Z_2=X_2\cap Z,Z_{12}=X_{12}\cap Z$ and $Y=X/(X-Z),
Y_1=X_1/(X_1-Z_1), Y_2=X_2/(X_2-Z_2),Y_{12}=X_{12}/(X_{12}-Z_{12})$.

\begin{lem}\label{pushout}
The maps $Y_{12}\to Y_1$, $Y_{12}\to Y_2$ are injective and the
sheaf $Y$ is the pushout of the diagram $Y_1\hookleftarrow Y_{12}\hookrightarrow Y_2$.
\end{lem}

\begin{proof}
For any Henselian local scheme $U$ the map
   $$Y_{12}(U)=X_{12}(U)/(X_{12}-Z_{12})(U)\to X_1(U)/(X_1-Z_1)(U)=Y_1(U)$$
is injective, because $(X_{12}-Z_{12})(U)=X_{12}(U)\cap (X_1-Z_1)(U)$. Similarly, the map $Y_{12}(U)\to Y_2(U)$ is injective.
Note that $Y(U)=X(U)/(X-Z)(U)$, $X(U)=X_1(U)\cup_{X_{12}(U)}X_2(U)$
and $(X-Z)(U)=(X_1-Z_1)(U)\cup_{(X_{12}-Z_{12})(U)}(X_2-Z_2)(U)$. Hence $Y(U)=Y_1(U)\cup_{Y_{12}(U)}Y_2(U)$.
\end{proof}

\begin{dfn}\label{zfrpi}
Let $F^{\PP^{\w i}}(U,Y)$, $U\in\Sm_k$, be the set of
$x\in\uhom(\PP^{\w i},\Fr(Y))(U)$ such that the support of $x$ is
connected. The free abelian group generated by $F^{\PP^{\w i}}(U,Y)$
is denoted by $\ZF^{\PP^{\w i}}(U,Y)$. Then $\ZF^{\PP^{\w
i}}(U,Y)$ is functorial in $U$. Moreover, $\ZF^{\PP^{\w i}}(-,Y)$ is
a Nisnevich sheaf.
\end{dfn}

The following result gives an explicit computation of homology of
the motivic $S^1$-spectrum $\uhom(\PP^{\w i},M_{fr}(Y))$.

\begin{lem}\label{zfrhomol}
There are isomorphisms of graded presheaves
\[\pi_*(\Z\uhom(\PP^{\w i},M_{fr}(Y)))=H_*(C_*\ZF^{\PP^{\w i}}(Y))\]
for all $i\geqslant 0$.
\end{lem}

\begin{proof}
The proof repeats the proof in~\cite[1.2]{GNP} word for word.
\end{proof}

\begin{lem}\label{ZFtriangleaux}
For any $i\geqslant 0$, the natural maps $Y_{12}\to Y_2, Y_{12}\to Y_1,
Y_1\to Y, Y_2\to Y$ give rise to a short exact sequence of Nisnevich
sheaves
\[0\to\ZF^{\PP^{\w i}}(Y_{12})\to\ZF^{\PP^{\w i}}(Y_{1})\oplus\ZF^{\PP^{\w i}}(Y_2))\to\ZF^{\PP^{\w i}}(Y)\to 0.\]
\end{lem}

\begin{proof}
Let $U$ be a local Henselian scheme. There is a coequalizer diagram
of pointed sets
\[F^{\PP^{\w i}}(U,Y_{12})\rightrightarrows F^{\PP^{\w i}}(U,Y_2)\vee F^{\PP^{\w i}}(U,Y_1)\to F^{\PP^{\w i}}(U,Y).\]
Thus it gives rise to a right exact sequence
\[\ZF^{\PP^{\w i}}(Y_{12})\to\ZF^{\PP^{\w i}}(Y_{1})\oplus\ZF^{\PP^{\w i}}(Y_2)\to\ZF^{\PP^{\w i}}(Y)\to 0.\]
It remains to note that the latter sequence is also exact on the
left.
\end{proof}

\begin{cor}\label{ZFtriangle}
The cone of the morphism of complexes
   \[C_*\ZF^{\PP^{\w i}}(Y_{12})\to C_*\ZF^{\PP^{\w i}}(Y_{1})\oplus C_*\ZF^{\PP^{\w i}}(Y_2)\]
is locally quasi-isomorphic to the complex $C_*\ZF^{\PP^{\w
i}}(Y))$. In particular, we have a triangle in the derived category
of complexes of sheaves
   \[C_*\ZF^{\PP^{\w i}}(Y_{12})\to C_*\ZF^{\PP^{\w i}}(Y_{1})\oplus C_*\ZF^{\PP^{\w i}}(Y_2)
     \to C_*\ZF^{\PP^{\w i}}(Y).\]
\end{cor}

\begin{proof}
Note that homology presheaves of $C_*\ZF^{\PP^{\w i}}(Y)$ are
quasi-stable. Then by Lemmas~\ref{ZFtriangleaux} and~\ref{Totzero}
the totalization of the bicomplex
\[0\to C_*\ZF^{\PP^{\w i}}(Y_{12})\to
C_*\ZF^{\PP^{\w i}}(Y_{1})\oplus C_*\ZF^{\PP^{\w i}}(Y_2)\to
C_*\ZF^{\PP^{\w i}}(Y)\to 0\] is locally quasi-isomorphic to zero.
\end{proof}

\begin{prop}[The Mayer--Vietoris sequence]\label{MV}
For every $i\geqslant 0$ the square of $S^1$-spectra
   $$\begin{xymatrix}
     {\uhom(\PP^{\w i},M_{fr}(Y_{12}))\ar[r]\ar[d]& \uhom(\PP^{\w i},M_{fr}(Y_1))\ar[d]\\
      \uhom(\PP^{\w i},M_{fr}(Y_2))\ar[r]&\uhom(\PP^{\w i},M_{fr}(Y))}
   \end{xymatrix}$$
is a homotopy pushout square in the local stable model structure of $S^1$-spectra.
\end{prop}

\begin{proof}
The natural map from the cone of the morphism
\[\uhom(\PP^{\w i},M_{fr}(Y_{12}))\to \uhom(\PP^{\w i},M_{fr}(Y_1)\vee M_{fr}(Y_2))\]
to $\uhom(\PP^{\w i},M_{fr}(Y))$ induces locally an equivalence on
homology between connective spectra by Corollary~\ref{ZFtriangle}. Then it is a local
stable equivalence.
\end{proof}

For any space $A$, there is an obvious map
$\Fr_n(A)\to\uhom(B,\Fr_n(A\w B))$ defined by $(\PP^{\w n},A\w
T^n)\to(B\w\PP^{\w n},A\w B\w T^n)$. It gives rise to a map of
spectra $M_{fr}(A)\to\uhom(B,M_{fr}(A\w B)).$

\begin{lem}\label{ploopseasy} For any $X\in\Sm_k$ for $j\geqslant 1$ the map
\[C_*\Fr(X_+\w T^j)\to \uhom(\PP^{\w i}, C_*\Fr(X\w T^j\w T^i))\]
is a local weak equivalence for any $i\geqslant 0$.
\end{lem}

\begin{proof}
The map in question is obtained as the colimit of the maps
   \begin{equation}\label{rrr}
    C_*\Fr_n(X_+\w T^j)\to\uhom(\PP^{\w i},C_*\Fr_n(X_+\w T^j\w T^i)).
   \end{equation}
Consider the triangle
\[
\xymatrix{
C_*\uhom(\PP^{\w i}\w\PP^{\w n},X_+\w T^j\w T^{i+n})& C_*\uhom(\PP^{\w i}\w\PP^{\w n-i},X_+\w T^j\w T^{i}\w T^{n-i})\ar[l]\\
C_*\uhom(\PP^{\w n},X_+\w T^j\w T^n)\ar[u]\ar[ru]_{\cong}}
\]
where the vertical map is the map~\eqref{rrr}, the skew map is the
isomorphism given by identification $\PP^{\w n}=\PP^{\w i}\w\PP^{\w
n-i}$ and $T^n=T^i\w T^{ n-i}$, and the horizontal map is induced by
the stabilization map $(\PP^{\w i},C_*\Fr_{n-i}(X_+\w
T^j))\to(\PP^{\w i},C_*\Fr_{n}(X_+\w T^j))$. The composite map of
the triangle differs from the left vertical map by the shuffle
permutation action $\chi_{n,i}$ on $\PP^{\w i+n}$ and on $T^{i+n}$
respectively. Thus if $n$ is even then the triangle is commutative
up to a simplicial homotopy by Corollary~\ref{Pmpermut} and
Remark~\ref{a1sim}. Note that the horizontal map induces an
isomorphism on the colimit over $n$. Thus the vertical map induces a
bijection on the colimits of sheaves $\pi^{Nis}_*$. For $j\geqslant
1$ the space $C_*(\Fr(X_+\w T^j))$ is locally connected
by~\cite[8.1]{GNP}. The space $\uhom(\PP^{\w i}, C_*(\Fr(X_+\w T^j\w
T^i)))$ is isomorphic to $C_*(\Fr(X_+\w T^j))$ by means of the
horizontal map, and hence it is locally connected as well. We see
that the vertical map induces a local weak equivalence.
\end{proof}

\begin{lem}\label{pi0}
Suppose $Z\to X$ is a closed embedding of smooth varieties of codimension $d$. Then for $i<d$
the space $\uhom(\PP^{\w i},(C_*\Fr(X/X-Z)))$ is locally connected.
\end{lem}

\begin{proof}
If $U$ is a local Henselian scheme, then every correspondence $c$ in
$\uhom(\PP^{\w i},\Fr(X/X-Z))=\colim_n\uhom(\PP^{\w i+n},(X/X-Z)\w
T^n)$ can be described by triples $c=(S,U,\phi)$, where the support
$S$ is a closed subset of $\A^{i+n}_U,$ finite over $U$, and
$\phi\colon U\to X\times \A^n$ is a regular map from an \'etale
neighborhood of $S$ such that $S=\phi^{-1}(Z\times 0)$ (see
Voevodsky's Lemma~\ref{Voevlemma} and~\cite[Section~3]{GPMain} for
details). Since $S$ is finite over Henselian $U$, it is a disjoint
union of local schemes $S_j$, finite over $U$, for $j=1\ldots l$.
Each map $S_j\to Z$ factors through $S_j\to Z_j$, where $Z_j=X_j\cap
Z$ for some open $X_j$ in $X$ and such that $Z_j\to X_j$ is a
trivial embedding. Thus the correspondence $c$ lies in the image
of $\uhom(\PP^{\w i},\Fr(\vee_j (X_j/X_j-Z_j)))=\uhom(\PP^{\w
i},\Fr(\vee(Z_j\w T^d)))= \Fr(\vee (Z_j\w T^{d-i}))=\Fr((\sqcup
Z_j)_+\w T^{d-i})$, and $\pi_0^{Nis}(C_*\Fr((\sqcup Z_j)_+\w
T^{d-i})=*$ for $i<d$ by~\cite[A.1]{GNP}. Since the class of
$c\in\pi_0^{Nis}(\uhom(\PP^{\w i},(C_*\Fr(X/X-Z))))$ belongs to the image of
$\pi_0^{Nis}(C_*\Fr((\sqcup Z_j)_+\w T^{d-i}))=*$, then
$c$ equals the class of the basepoint of $\uhom(\PP^{\w i},(C_*\Fr(X/X-Z)))$.
We conclude that $\pi_0^{Nis}(\uhom(\PP^{\w i},(C_*\Fr(X/X-Z))))=*$.
\end{proof}

\begin{lem}\label{omegas}
Suppose $Z\to X$ is a closed embedding of smooth varieties of
codimension $d$. Then for $i<d$ the $S^1$-spectrum $\uhom(\PP^{\w
i},M_{fr}(X/X-Z))$ is locally an $\Omega$-spectrum and the
$S^1$-spectrum $\uhom(\PP^{\w i},M_{fr}(X/X-Z))_f$, obtained from
$\uhom(\PP^{\w i},M_{fr}(X/X-Z))$ by taking a local fibrant
replacement levelwise, is motivically fibrant. In particular, the
motivic space $\uhom(\PP^{\w i}, C_*\Fr(X/X-Z))_f$ is motivically
fibrant.
\end{lem}

\begin{proof}
It follows from Additivity Theorem~\ref{additivity},
Corollary~\ref{specialgamma} and Lemma~\ref{pi0} that the
$\Gamma$-space taking a finite pointed set $K$ to $\uhom(\PP^{\w i},
C_*\Fr((X/X-Z)\wedge K))$ is locally very special. By the Segal
machine~\cite{S} $\uhom(\PP^{\w i},M_{fr}(X/X-Z))= \uhom(\PP^{\w i},
C_*\Fr((X/X-Z)\wedge\mathbb S))$ is locally an $\Omega$-spectrum,
and hence so is $\uhom(\PP^{\w i},M_{fr}(X/X-Z))_f$. Since all
spaces of $\uhom(\PP^{\w i},M_{fr}(X/X-Z))_f$ are locally fibrant,
we see that $\uhom(\PP^{\w i},M_{fr}(X/X-Z))_f$ is sectionwise an
$\Omega$-spectrum. Since the sheaves of homotopy groups of
$\uhom(\PP^{\w i},M_{fr}(X/X-Z))_f$ are strictly homotopy invariant
by~\cite[1.1]{GPPresheaves}, $\uhom(\PP^{\w i},M_{fr}(X/X-Z))_f$ is
motivically fibrant by~\cite[7.1]{GPMain}.
\end{proof}

\begin{cor}\label{omegass}
For a Thom spectrum $E$ with the bounding constant $d$, $\uhom(\PP^{\w
n},C_*\Fr(E_n\w T^i))_f$ is a motivically fibrant space for
$i\geqslant \max(0,d)$.
\end{cor}

\begin{lem}\label{spectraeqprelim}
Given $i,n\geqslant 0$, the natural map of $S^1$-spectra
\[M_{fr}(X_+\w T^n)_f\to\uhom (\PP^{\w i},M_{fr}(X_+\w T^n\w T^i)_f)\]
is a levelwise local weak equivalence in positive degrees, where
``f" refers to a levelwise local fibrant replacement. In particular,
the map is a stable local weak equivalence. If $n>0$ then this map is
a levelwise local weak equivalence of spectra in all degrees.
\end{lem}

\begin{proof}
The statement of the lemma can be reformulated as follows for
$n\geqslant 0$: the map of $S^1$-spectra
\[M_{fr}(X_+\w S^1\w T^n)_f\to\uhom (\PP^{\w i},M_{fr}(X_+\w S^1\w T^n\w T^i)_f)\]
is a levelwise local weak equivalence. The spectra $M_{fr}(X_+\w
S^1\w T^n)_f, M_{fr}(X_+\w S^1\w T^n\w T^i)_f$ are both motivically
fibrant by~\cite[7.5]{GPMain}.

The proof of~\cite[4.1(2)]{GPMain} shows that the map in question is
a levelwise local weak equivalence if so is the map
   $$M_{fr}(X_+\w S^1\w T^n)_f\to\uhom (\Gm^{\w i}\w S^i,M_{fr}(X_+\w S^1\w T^n\w \Gm^{\w i}\w S^i)_f),$$
where $\Gm^{\w 1}$ is the mapping cone of $pt_+\to(\Gm)_+$ sending
$pt$ to $1\in\Gm$ and $\Gm^{\w i}$ is the $i$th smash product of
$\Gm^{\w 1}$. Our assertion now follows from the Cancellation
Theorem for framed motives~\cite{AGP}. The same arguments apply to
show that the map
\[M_{fr}(X_+\w T^n)_f\to\uhom (\PP^{\w i},M_{fr}(X_+\w T^n\w T^i)_f)\]
is a levelwise local weak equivalence in all degrees for $n>0$.
\end{proof}

\begin{prop}\label{spectraeq}
Suppose $Z\to X$ is a closed embedding of smooth varieties of
codimension $d$ and $M_{fr}(X/(X-Z))_f$ is obtained from
$M_{fr}(X/(X-Z))$ by taking a level local fibrant replacement. Then
\[\uhom(\PP^{\w i},M_{fr}(X/(X-Z)))\to \uhom(\PP^{\w i},M_{fr}(X/(X-Z))_f)\]
is a levelwise local weak equivalence of $S^1$-spectra for $i<d$. In
particular, the right spectrum is a fibrant replacement of the left
spectrum in the stable motivic model structure of $S^1$-spectra
whenever $i<d$. If $i=d$ then the map is a levelwise local weak
equivalence in positive degrees. In particular, the map is a stable
local weak equivalence for $i=d$.
\end{prop}

\begin{proof}
Suppose $i<d$. By Lemma~\ref{cover} there is a cover of $X$ by open
subsets $X_j$ such that $X_j\cap Z\to X_j$ is a trivial embedding.
We proceed by induction on $n$, the number of elements in the cover.
For $n=1$ we have $X/X-Z\cong Z_+\w T^d$ by Lemma~\ref{ZTd}. Then
the map in question fits into a commutative square
\[
\xymatrix{
\uhom(\PP^{\w i},M_{fr}(Z_+\w T^d))\ar[r] & \uhom(\PP^{\w i},M_{fr}(Z_+\w T^d)_f)\\
M_{fr}(Z_+\w T^{d-i})\ar[u]^{\simeq}\ar[r]^{\simeq} & M_{fr}(Z_+\w T^{d-i})_f\ar[u]^{\simeq}}
\]
The left arrow is a levelwise local weak equivalence by
Lemma~\ref{ploopseasy}, and the right arrow is a levelwise local
weak equivalence by Lemma~\ref{spectraeqprelim}. Thus the upper map
is a levelwise local weak equivalence.

For the induction step present $X$ as the union of $X_1$ and $X_2$
such that $X_1$ can be covered by $n-1$ trivial open pieces, and
$Z\cap X_2\to X_2$ is a trivial embedding. Then for $X_{12}=X_1\cap
X_2$ the embedding $Z\cap X_{12}\to X_{12}$ is trivial. Denote by
$Y$ the sheaf $X/X-Z$ and by $Y_i$ the sheaf $X_{i}/(X_i-(X_i\cap
Z))$. Consider a commutative diagram of $S^1$-spectra\footnotesize
\[\xymatrix{
\uhom(\PP^{\w i},M_{fr}(Y_{12}))\ar[r]\ar[d] & \uhom(\PP^{\w
i},M_{fr}(Y_1))\vee\uhom(\PP^{\w
i},M_{fr}(Y_2))\ar[r]\ar[d] & \uhom(\PP^{\w i},M_{fr}(Y))\ar[d]\\
\uhom(\PP^{\w i},M_{fr}(Y_{12})_f)\ar[r] & \uhom(\PP^{\w
i},M_{fr}(Y_1)_f)\vee\uhom(\PP^{\w i},M_{fr}(Y_2)_f)\ar[r] &
\uhom(\PP^{\w i},M_{fr}(Y)_f) }
\]
\normalsize The upper row is a homotopy cofiber sequence in the
local stable model structure of $S^1$-spectra by
Proposition~\ref{MV}. By Lemma~\ref{omegas}
$M_{fr}(Y_{12})_f,M_{fr}(Y_1)_f,M_{fr}(Y_2)_f,M_{fr}(Y)_f$ are
motivically fibrant. It follows from Proposition~\ref{MV} that the
sequence
   $$M_{fr}(Y_{12})_f\to M_{fr}(Y_1)_f\vee M_{fr}(Y_2)_f\to M_{fr}(Y)_f$$
is a homotopy cofiber sequence of motivically fibrant spectra in the
local stable model structure, and hence so is the lower sequence of
the commutative diagram above, because $\mathbb P^{\wedge i}$ is a
flasque cofibrant motivic space. Two left vertical arrows are
levelwise local weak equivalences by induction hypothesis. Hence the
right arrow is a stable local weak equivalence factoring as
   $$\uhom(\PP^{\w i},M_{fr}(Y))\to\uhom(\PP^{\w i},M_{fr}(Y))_f\to\uhom(\PP^{\w i},M_{fr}(Y)_f).$$
Since the left arrow is a levelwise local equivalence by definition,
then the right arrow is a stable local weak equivalence. But the
middle spectrum is motivically fibrant by Lemma~\ref{omegas} as well
as so is the right spectrum. It remains to observe that a stable
local equivalence between motivically fibrant spectra must be a
levelwise local weak equivalence.

If $i=d$ then we replace all framed motives and their levelwise
local fibrant replacements by framed motives smashed with the unit
circle $S^1$. Then all spaces of $M_{fr}(Y\w S^1)$ become connected
and $M_{fr}(Y\w S^1)_f$ is a motivically fibrant $S^1$-spectrum. It
is now enough to repeat the above arguments word for word
(Lemma~\ref{ploopseasy} is also satisfied for spaces of the form
$C_*\Fr(X_+\w S^1)$ which are automatically sectionwise connected)
to show that
   $$\uhom(\PP^{\w d},M_{fr}(Y\w S^1))\to\uhom(\PP^{\w d},M_{fr}(Y\w S^1)_f)$$
is a stable local weak equivalence of spectra. By
Corollary~\ref{specialgamma} the left spectrum is sectionwise an
$\Omega$-spectrum. Since a stable equivalence between
$\Omega$-spectra is a levelwise weak equivalence, it follows that
the map of spectra is a levelwise local weak equivalence. Therefore,
the map
   $$\uhom(\PP^{\w d},M_{fr}(Y))\to\uhom(\PP^{\w d},M_{fr}(Y)_f)$$
is a levelwise local weak equivalence in positive degrees.
\end{proof}

\begin{thm}\label{geommain}
Suppose $Z\to X$ is a closed embedding of smooth varieties of
codimension $d$. Then the space $\uhom(\PP^{\w i},C_*\Fr(X/X-Z)_f)$
is motivically fibrant and
\[\uhom(\PP^{\w i},C_*\Fr(X/X-Z))\to \uhom(\PP^{\w i},C_*\Fr(X/X-Z)_f)\]
is a local weak equivalence for $i<d.$
\end{thm}

\begin{proof}
The statement follows from Proposition~\ref{spectraeq}.
\end{proof}

\begin{cor}\label{geomaincor}
If $E$ is a Thom spectrum with the bounding constant $d$, then the motivic space
$\uhom(\PP^{\w m},C_*\Fr(E_n\w T^i)_f)$ is motivically fibrant and
\[\uhom(\PP^{\w m},C_*\Fr(E_n\w T^i))\to \uhom(\PP^{\w m},C_*\Fr(E_n\w T^i)_f)\] is a
local weak equivalence for $m\leqslant n+i-d$.
\end{cor}

\begin{proof}
We have $E_n\w T^i=\colim_j V_{n,j}\times\A^i/(V_{n,j}\times
\A^i-Z_{n,j}\times 0)$, where codimension of $Z_{n,j}$ in $V_{n,j}$
is strictly greater than $n-d$. Then codimension of $Z_{n,j}\times
0$ in $V_{n,j}\times\A^i$ is strictly greater than $n+i-d$. Then for
$E_{n,j}=V_{n,j}/(V_{n,j}-Z_{n,j})$ we get that $\uhom(\PP^{\w
m},C_*\Fr(E_{n,j}\w T^i)_f)$ is motivically fibrant and the map
$\uhom(\PP^{\w m},C_*\Fr(E_{n,j}\w T^i))\to \uhom(\PP^{\w
m},C_*\Fr(E_{n,j}\w T^i)_f)$ is a local weak equivalence for every
$j$ by Theorem~\ref{geommain}. By passing to the colimit and using
the fact that a directed colimit of flasque motivically fibrant
spaces (respectively a directed colimit of local weak equivalences)
is flasque motivically fibrant, we get the statement of the lemma.
\end{proof}

\section{Fibrant replacements of Thom spectra}\label{fibrantsect}

In this section we give a model for a fibrant replacement of a Thom spectrum $E$. First
we need the following

\begin{lem}\label{homf}
Suppose $E$ is a Thom spectrum with the bounding constant $d$. Then
for $i\geqslant\max(0,d)$ and $n\geqslant 0$ the map of spaces
$\uhom(\PP^{\w 1},C_*\Theta^{\infty}(L_nE)_{i+1})\to\uhom(\PP^{\w
1},C_*\Theta^{\infty}(L_nE)_{i+1}^f)$ is a local weak equivalence,
where $L_nE$ is the $n$-th layer of $E$ and $(L_nE)_{i+1}^f$ is a
local fibrant replacement of the space $(L_nE)_{i+1}$.
\end{lem}

\begin{proof}
By Corollary~\ref{geomaincor} the space $\uhom(\PP^{\w
n},C_*\Fr(E_n\w T^{i+1})^f)$ is motivically fibrant. By
Lemma~\ref{thetalayer} the map in question coincides with the
horizontal map of the diagram
\begin{equation}\label{aaa}
\xymatrix{
\uhom(\PP^{\w 1},\uhom(\PP^{\w n},C_*\Fr(E_n\w T^{i+1})))\ar[r]\ar[d] & \uhom(\PP^{\w 1},\uhom(\PP^{\w n},C_*\Fr(E_n\w T^{i+1}))^f)\ar[ld]\\
\uhom(\PP^{\w 1},\uhom(\PP^{\w n},C_*\Fr(E_n\w T^{i+1})^f))
}\end{equation}
The diagram~\eqref{aaa} is obtained by applying $\uhom(\PP^{\w 1},-)$ to the diagram
\begin{equation}\label{ccc}
\xymatrix{
\uhom(\PP^{\w n},C_*\Fr(E_n\w T^{i+1}))\ar[r]\ar[d] & \uhom(\PP^{\w n},C_*\Fr(E_n\w T^{i+1}))^f\ar[ld]\\
\uhom(\PP^{\w n},C_*\Fr(E_n\w T^{i+1})^f) }
\end{equation}
The slanted arrow exists by the right lifting property for fibrant
spaces. The horizontal arrow of~\eqref{ccc} is a local weak
equivalence, and the vertical arrow of~\eqref{ccc} is a local weak
equivalence by Corollary~\ref{geomaincor}. It follows that the
slanted arrow of~\eqref{ccc} is a local weak equivalence between
fibrant spaces, and hence so is the slanted arrow of~\eqref{aaa}
since $\PP^{\w 1}$ is a flasque cofibrant space. The vertical arrow
of~\eqref{aaa} is a local weak equivalence by
Corollary~\ref{geomaincor}. We see that the horizontal map
of~\eqref{aaa} is a local weak equivalence.
\end{proof}

The following theorem says that a fibrant replacement of a Thom
spectrum $E$ can be computed (starting at some level depending on
its bounding constant) by first applying the $\Theta^\infty$-functor
to $E$, then by taking the Suslin complex of each space of
$\Theta^\infty(E)$ and finally by taking local fibrant replacements
for $C_*\Theta^\infty(E)$.

\begin{thm}\label{thmthom}
Let $E$ be a Thom spectrum with the bounding constant $d$. Let
$C_*\Theta^{\infty}(E)^f$ be a spectrum obtained from
$C_*\Theta^{\infty}(E)$ by taking a level local fibrant replacement.
Then the spectrum $C_*\Theta^{\infty}(E)^f$ is motivically fibrant
starting from level $\max(0,d)$ and is stably equivalent to $E$.
\end{thm}

\begin{proof}
Since a directed colimit of flasque locally fibrant spaces is
flasque locally fibrant, it follows that
$C_*\Theta^{\infty}(E)^f=\colim_n C_*\Theta^{\infty}(L_nE)^f$. Hence
it is sufficient to prove that for every $n$ the spectrum
$C_*\Theta^{\infty}(L_nE)^f$ is motivically fibrant starting from
level $d$. For $i\geqslant d$ the space
$C_*\Theta^{\infty}(L_nE)_i^f$ equals $\uhom(\PP^{\w n},C_*\Fr(E_n\w
T^i))^f$ by Lemma~\ref{thetalayer}. Moreover, it is motivically
fibrant by Corollary~\ref{omegass}. Thus it remains to prove that
each bonding map
\[C_*\Theta^{\infty}(L_nE)_i^f\to\uhom(\PP^{\w 1},C_*\Theta^{\infty}(L_nE)_{i+1}^f)\]
is a local weak equivalence. It fits into the following commutative diagram:
\[
\xymatrix{
C_*\Theta^{\infty}(L_nE)_i^f\ar[r] & \uhom(\PP^{\w 1},C_*\Theta^{\infty}(L_nE)_{i+1}^f)\\
C_*\Theta^{\infty}(L_nE)_i\ar[u]\ar[r]^(.35)\cong & \uhom(\PP^{\w 1},C_*\Theta^{\infty}(L_nE)_{i+1})\ar[u]}
\]
where the right vertical arrow is a local weak equivalences by Lemma~\ref{homf}
and the lower arrow is an isomorphism by Lemma~\ref{thetaomega}. Since the left vertical
arrow is a local weak equivalence, then so is the upper arrow, as required.
\end{proof}

\section{The functor $\Theta^{\infty}_{sym}$}\label{thetasym}

Whenever a Thom $T$-spectrum $E$ is symmetric, we can also construct
further fibrant replacements for it. To this end, we introduce
another stabilization functor $\Theta^{\infty}_{sym}$ on the level
of symmetric $T$-spectra, which is slightly different from
$\Theta^{\infty}$. The spaces of $\Theta^{\infty}_{sym}(E)$ and
$\Theta^{\infty}(E)$ are in fact isomorphic, but the bonding maps
are different: the bonding maps of $\Theta^{\infty}_{sym}(E)$
require the structure of a symmetric spectrum on $E$, whereas the
bonding maps of $\Theta^{\infty}(E)$ do not.

Given a $T$-spectrum $E$, let $T\w E$ be the suspension spectrum of
$E$ (see Definition~\ref{shift}). The functor $E\mapsto T\w E$ has a
right adjoint loop functor $E\to\Omega_T E$, where  $\Omega_T E$ has
the spaces $(\Omega_T E)_i=\uhom(T,E_i)$. If there is no likelihood
of confusion, we denote by $\Omega E$ the $\PPo$-spectrum
with $(\Omega E)_i=\uhom(\mathbb P^{\wedge
1},E_i)$ and the bonding maps are given by
\[\uhom(\mathbb P^{\wedge 1},E_i)\w\mathbb P^{\wedge 1}\to
\uhom(\mathbb P^{\wedge 1},E_i\w\mathbb P^{\wedge 1})\xrightarrow{\sigma}
\uhom(\mathbb P^{\wedge 1},E_i\w T)\stackrel{u}\to(\mathbb P^{\wedge 1},E_{i+1}),\]
where $\mathbb P^{\wedge 1}\to T$ is a canonical motivic equivalence.

\begin{dfn}
Define the functor $\Theta^1_{sym}(E)=\Omega(E[1])$, where $E[1]$ is the shift
spectrum (see Definition~\ref{shift}), and
\[\Theta^n_{sym}(E):=\Theta^1_{sym}(\Theta^1_{sym}(\ldots(E)))\quad\text{ ($n$ times)}.\]
\end{dfn}

\begin{dfn}\label{thetasymdef}
If $E$ is a symmetric $T$-spectrum, then there is a canonical map of
$T$-spectra $T\w E\to E[1]$ (see Definition~\ref{shift}). Notice
that this map requires the symmetric spectrum structure of $E$. By
adjointness we have a map $E\to\Omega(T\w
E)\to\Omega(E[1])=\Theta^1_{sym}(E)$. Iterating the latter map, we
get a sequence of maps of spectra
\[E\to\Theta^1_{sym}(E)\to\Theta^2_{sym}(E)\to\cdots\]
Denote by $\Theta^{\infty}_{sym}(E)$ the colimit of this sequence. Then for every symmetric
$T$-spectrum $E$ there is a natural map of $\PPo$-spectra
\[\epsilon:E\to\Theta^{\infty}_{sym}(E).\]
\end{dfn}

\begin{rem}\label{rem:stabilization}
We need to describe bonding maps of $\Theta^n_{sym}(E)$ and
stabilization maps $\Theta^n_{sym}(E)\to\Theta^{n+1}_{sym}(E)$
explicitly. One has,
   \[\Theta^n_{sym}(E)_i=\uhom(\PP^{\w n},E_{n+i}).\]
Each bonding map equals the composition
   \[\uhom(\PP^{\w n},E_{n+i})\w\PP^{\w 1}\to
\uhom(\PP^{\w n},E_{n+i}\w\PP^{\w 1})\xrightarrow{\sigma}
\uhom(\PP^{\w n},E_{n+i}\w T)\stackrel{u}\to \uhom(\PP^{\w
n},E_{n+i+1})\] and the stabilization map
$\Theta^n_{sym}(E)_i\to\Theta^{n+1}_{sym}(E)_i$ equals the
composition
\[\uhom(\PP^{\w n},E_{n+i})\to(\PP^{\w n},\uhom(\PP^{\w 1}, E_{n+i}\w T))
\stackrel{u}\to\uhom(\PP^{\w
n+1},E_{n+i+1})\xrightarrow{\chi_{i,1}}\uhom(\PP^{\w
n+1},E_{n+1+i}),\] where the left arrow is induced by the external
smash product with $\sigma:\PP^{\w 1}\to T$ and $\chi_{i,1}$ is the
shuffle permutation in $\Sigma_{n+i+1}$ permuting the last element
with preceding $i$ elements and preserves the first $n$ elements.
\end{rem}

\begin{lem}\label{spacetrue=fake}
For any symmetric $T$-spectrum $E$ for any $i$ there is an
isomorphism of motivic spaces
$\Theta^{\infty}(E)_i\xrightarrow{\cong}\Theta^{\infty}_{sym}(E)_i.$
\end{lem}

\begin{proof}
Define a map $f_n\colon\Theta^n(E)_i\to\Theta^n_{sym}(E)_i$ by the
formula
\[f_n\colon \uhom(\PP^{\w n},E_{i+n})\xrightarrow{\chi_{i,n}}\uhom(\PP^{\w n},E_{n+i}),\]
where $\chi_{i,n}$ is the shuffle permutation that permutes the last
$n$ elements with the first $i$ elements. Then the following diagram
is commutative:
\[
\xymatrix{
\uhom(\PP^{\w n},E_{i+n})\ar[r]^{\chi_{i,n}}\ar[d]^{u} & \uhom(\PP^{\w n},E_{n+i})\ar[d]^{\chi_{i,1}\circ u}\\
\uhom(\PP^{\w n+1},E_{i+n+1})\ar[r]^{\chi_{i,n+1}} & \uhom(\PP^{\w
n+1},E_{n+1+i}).}
\]
Here the left vertical arrow is the stabilization map
$\Theta^n(E)_i\to\Theta^{n+1}(E)_i$ and the right vertical map is
the stabilization map
$\Theta^n_{sym}(E)_i\to\Theta^{n+1}_{sym}(E)_i$ of
Remark~\ref{rem:stabilization}. So the maps $f_n$ induce a morphism
of sequences. Then the maps $f_n$ induce the desired isomorphism on
colimits
$f\colon\Theta^{\infty}(E)_i\xrightarrow{\cong}\Theta^{\infty}_{sym}(E)_i$.
\end{proof}

\begin{lem}\label{infm=inf}
For any symmetric $T$-spectrum $E$ there are isomorphisms of spaces
\[\Theta^{\infty}(\Theta^m_{sym}(E))_i\cong
\Theta^{\infty}_{sym}(E)_i,\quad
\Theta^{n}(\Theta^{\infty}_{sym}(E))_i\cong
\Theta^{\infty}_{sym}(E)_i.\]
\end{lem}

\begin{proof}
Applying Lemma~\ref{spacetrue=fake} to the symmetric spectrum
$\Theta^m_{sym}(E)$, we have
   $$\Theta^{\infty}(\Theta^m_{sym}(E))_i\cong\Theta^{\infty}_{sym}(\Theta^m_{sym}(E))_i=\Theta^{\infty}_{sym}(E)_i.$$
Also,
$$\Theta^{n}(\Theta^{\infty}_{sym}(E))_i=\uhom(\PP^{\w
n},\Theta^{\infty}_{sym}(E)_{i+n})\cong\uhom(\PP^{\w
n},\Theta^{\infty}(E)_{i+n})=\Theta^{n}(\Theta^{\infty}(E))_i=\Theta^{\infty}(E)_i,$$
as required.
\end{proof}

\begin{lem}\label{fakeconnected}
Suppose $E$ is a Thom $T$-spectrum with the bounding constant $d$. Then
the space $C_*\Theta^{\infty}(E)_i$ is locally connected for
$i\geqslant\max(0,d).$
\end{lem}

\begin{proof}
By Lemma~\ref{pi0} the space $C_*\uhom(\PP^{\w n},\Fr(E_n\w
T^i))=C_*\Theta^{\infty}(L_n(E))_i$ is locally connected for every
$n$. Then
$C_*\Theta^{\infty}(E)_i=\colim_nC_*\Theta^{\infty}(L_n(E))_i$ is a
locally connected space.
\end{proof}

\begin{prop}\label{truefakezigzag}
Suppose $E$ is a symmetric Thom $T$-spectrum with the bounding
constant $d$ and contractible alternating group action. Then the
natural maps of spectra
\[\xi:C_*\Theta^{\infty}_{sym}(E)\to C_*\Theta^{\infty}(\Theta^{\infty}_{sym}(E))\] and
\[C_*\Theta^{\infty}(\epsilon):C_*\Theta^{\infty}(E)\to C_*\Theta^{\infty}(\Theta^{\infty}_{sym}(E)),\]
obtained from the map $\epsilon:E\to\Theta^{\infty}_{sym}(E)$ by
applying $C_*\Theta^{\infty}$ to it, induce local weak equivalences
of spaces starting from level $\max(0,d)$.
In particular, there is a commutative diagram
   $$\begin{xymatrix}{
       E\ar[r]^{\eta}\ar[d]_\epsilon&C_*\Theta^{\infty}(E)\ar[d]^{C_*\Theta^{\infty}(\epsilon)}\\
       C_*\Theta^{\infty}_{sym}(E)\ar[r]^(.43){\xi}&C_*\Theta^{\infty}(\Theta^{\infty}_{sym}(E))
      }\end{xymatrix}$$
of $\mathbb P^1$-spectra, in which all arrows are stable motivic equivalences.
\end{prop}

\begin{proof}
Fix a number $i\geqslant d$. Consider a two-dimensional sequence
\[A_{n,m}=C_*\Theta^n(\Theta^m_{sym}(E))_i\] with horizontal maps $A_{n,m}\to A_{n+1,m}$
induced by $\Theta^n\to\Theta^{n+1}$ and vertical maps $A_{n,m}\to
A_{n,m+1}$ induced by $\Theta^m_{sym}\to\Theta^{m+1}_{sym}$. To
prove the statement, we need to show that the maps
$\colim_{m}A_{0,m}\to\colim_{n,m}A_{n,m}$ and
$\colim_{n}A_{n,0}\to\colim_{n,m}A_{n,m}$ are local weak
equivalences.

Without loss of generality it is sufficient to prove that for every
$m,n$ the maps
\begin{equation}\label{bbb}
\colim_{n}A_{2n,2m}\to\colim_nA_{2n,2m+2}
\end{equation}
and
\begin{equation}\label{ddd}
\colim_{m}A_{2n,2m}\to\colim_{m}A_{2n+2,2m}
\end{equation}
are local weak equivalences.

Note that the spaces $C_*\Theta^{\infty}(\Theta^m_{sym}(E))_i$ and
$C_*\Theta^{n}(\Theta^{\infty}_{sym}(E))_i$ are isomorphic to
$C_*\Theta^{\infty}(E)_i$ by Lemmas~\ref{infm=inf}
and~\ref{spacetrue=fake}. Hence they are locally connected by
Lemma~\ref{fakeconnected}.

To prove that~\eqref{bbb} is a local weak equivalence, we apply
Lemma~\ref{colimequi} below for the case
$A_n=A_{2n,2m},B_n=A_{2n,2m+2}$ and the maps $i_n^A:A_{2n,2m}\to A_{2n+2,2m}$,
$i_n^B:A_{2n,2m+2}\to A_{2n+2,2m+2}$, $f_n:A_{2n,2m}\to A_{2n,2m+2}$ are given by maps
of the two dimensional sequences above.
Define a map $g_n\colon A_{2n,2m+2}\to A_{2n+2,2m}$ as
an identification via associativity isomorphism
\begin{gather*}
A_{2n,2m+2}=C_*\uhom(\PP^{\w 2n},\uhom(\PP^{\w 2m+2},E_{2n+2m+2+i}))=\\
=C_*\uhom(\PP^{\w 2n+2},\uhom(\PP^{\w 2m},E_{2n+2m+2+i}))=A_{2n+2,2m}.
\end{gather*}
Then $g_nf_n$ differs from $i_n^A$ by the action of an even
permutation on $\PP^{\w 2n+2m+2}$ and an even permutation on
$E_{2n+2m+2+i}$. Thus $g_nf_n$ and $i^A_n$ are simplicially
homotopic by Corollary~\ref{Pmpermut} and our assumption that $E$ is
a spectrum with contractible alternating group action as well as the fact that $\mathbb A^1$-homotopies 
become the usual ones after applying Suslin’s complex $C_*$. Also,
$f_{n+1}g_n$ differs from $i_n^B$ by the action of an even permutation
on $\PP^{\w 2n+2m+4}$ and an even permutation on $E_{2n+2m+4}$.
Therefore $f_{n+1}g_n$ is simplicially homotopic to $i_n^B$ for the
same reasons as above. Thus the map on the colimits is a local weak
equivalence by Lemma~\ref{colimequi}. The proof for the
map~\eqref{ddd} is analogous.

Finally, the map $\eta$ of the commutative square of the proposition
is a stable motivic equivalence by~\cite[4.11]{H}, because the
flasque motivic model structure on spaces is almost finitely
generated in the sense of~\cite{H}. By the first part
of the proof $\xi,C_*\Theta^\infty(\epsilon)$ are stable
motivic equivalences, and hence so is $\epsilon$ by the
two-out-of-three property for weak equivalences.
\end{proof}

\begin{lem}\label{colimequi}
Suppose $i_n^A\colon A_n\to A_{n+1}$,
$i_n^B\colon B_n\to B_{n+1}$ are directed systems of spaces,
and $f_n\colon A_n\to B_n$ is a map of directed sequences.
Suppose that there are maps $g_n\colon B_n\to A_{n+1}$  such that
$g_nf_n$ is simplicially homotopic to $i_n^A$ and $f_{n+1}g_n$ is
simplicially homotopic to $i_n^B$. Also, suppose that the spaces $A=\colim
A_n$ and $B=\colim B_n$ are locally connected. Then the map
$f=\colim f_n\colon A\to B$ is a local weak equivalence.
\end{lem}

\begin{proof}
Given a local Henselian scheme $U$, the map $\pi_i(f)(U)\colon
\pi_i(A(U))\to\pi_i(B(U))$ equals the colimit of the system
$\pi_i(f_n)(U)$. Note that the maps $\pi_i(g_n)(U)$ form a map of
sequences $\pi_i(B_n(U))\to\pi_i(A_{n+1}(U))$, which are inverse to
$\pi_i(f_n)(U)$. Therefore the colimit
$\pi_i(f)(U)=\colim\pi_i(f_n)(U)$ is bijective, and hence the map
$f(U)$ induces a weak equivalence of connected simplicial sets
$A(U)\to B(U)$.
\end{proof}

\section{The spectrum $C_*\Fr^E(S_T)$}\label{thespectrum}

The purpose of this section is to introduce another spectrum
$C_*\Fr^E(S_T)$ associated with a symmetric $T$-spectrum $E$. We
show that it is stably equivalent to the spectrum
$C_*\Theta^\infty_{sym}(E)$ whenever $E$ is a Thom spectrum with the
bounding constant $d$ and contractible alternating group action (see
Proposition~\ref{fr=true}).

\begin{dfn}\label{frnex}
Given a $T$-spectrum $E$ and $X\in\Sm_k$, denote by $\Fr_n^E(X)$ the space
$\Theta^n(X_+\w E)_0$:
\[\Fr_n^E(X)=\uhom(\PP^{\w n},X_+\w E_{n})\]
and $\Fr^E(X):=\colim_n\Fr^E_n(X)=\Theta^{\infty}(X_+\w E)_0$.

\end{dfn}

\begin{dfn}
For any symmetric $T$-spectrum $E$ and any $m,n\geqslant 0$, define
a pairing
\[\Fr_n(X,Y)\times\Fr_m^E(Y,Z)\to\Fr_{n+m}^E(X,Z)\]
as follows. Let $a\in\Fr_n(X,Y)$ be given by a map $a\colon
X_+\w\PP^{\w n}\to Y_+\w T^n$ and let $b\in\Fr_m^E(Y,Z)$ be given by
$b\colon Y_+\w\PP^{\w m}\to Z_+\w E_m$. Define $b\circ a$ as the
composition
\begin{gather*}
X_+\w\PP^{\w n}\w\PP^{\w m}\xrightarrow{a\w\PP^{\w m}}
Y_+\w T^n\w\PP^{\w m}\xrightarrow{tw} Y_+\w\PP^{\w m}\w T^n\xrightarrow{b\w T^n} Z_+\w E_{m}\w T^n\xrightarrow{tw}\\
\to Z_+\w T^n\w E_m\xrightarrow{u_l}Z_+\w E_{n+m},
\end{gather*}
where $u_l$ is the map of Definition~\ref{shift}.
\end{dfn}

Note that if $E=S_T$, this definition coincides with the definition
of the composition of framed correspondences defined
in~\cite{Voe2,GPMain}.

\begin{lem}
The pairing above endows $\Fr^E(X)$ with a structure of a presheaf
with framed correspondences.
\end{lem}

\begin{proof}
This is straightforward.
\end{proof}

\begin{dfn}\label{FrnE}
Given a $T$-spectrum $E$ and $X\in\Sm_k$, denote by $\Fr_n^E(X_+\w S_T)$ the
$T$-spectrum with the spaces
\[\Fr_n^E(X_+\w S_T)_i:=\Fr_n^E(X_+\w T^i)=\Fr_n^{T^i\w E}(X).\]
The bonding maps $\Fr_n^E(X_+\w T^i)\w T\to\Fr_n^E(X_+\w T^{i+1})$
are defined as the composite maps
\[\uhom(\PP^{\w n},X_+\w T^i\w E_n)\w T\to\uhom(\PP^{\w n},X_+\w T^i\w E_n\w T)\xrightarrow{tw}\uhom(\PP^{\w n},X_+\w T^i\w T\w E_n).\]
In what follows we normally regard $\Fr_n^E(X_+\w S_T)$ as a
$\PPo$-spectrum. The stabilization maps $\Fr_n^E(X_+\w
T^i)\to\Fr_{n+1}^E(X_+\w T^i)$, given by the compositions
   $$\uhom(\PP^{\w n},X_+\w T^i\w E_n)\xrightarrow{-\wedge\sigma}\uhom(\PP^{\w n}\wedge\PP^{\w 1},X_+\w T^i\w E_n\wedge T)
       \xrightarrow{u_n}\uhom(\PP^{\w n+1},X_+\w T^i\w E_{n+1}),$$
define a map of $\PPo$-spectra $\Fr_n^E(X_+\w
S_T)\to\Fr_{n+1}^E(X_+\w S_T)$. Denote by
   $$\Fr^E(X_+\w S_T):=\colim_n\Fr^E_n(X_+\w S_T).$$
If $E=S_T$ the spectrum $\Fr^E(X_+\w S_T)$ coincides with the
spectrum $\Fr_{\PPo,T}(X)$ defined in \cite{GPMain}.
\end{dfn}

Note that for $X\in\Sm_k$ the spectrum $\Fr^E(X_+\w S_T)$ is
isomorphic to the spectrum $\Fr^{X_+\w E}(S_T)$. If $E$ is a
symmetric Thom spectrum with the bounding constant $d$, then so is
$X_+\w E.$ Thus we shall consider spectra of the form $\Fr^E(S_T)$
in what follows.

For any $n\geqslant 0$ and any symmetric $T$-spectrum $E$, construct
a map of $\PPo$-spectra $f_n\colon\Fr^E_n(S_T)\to\Theta^n_{sym}(E)$
as the composition at each level $i\geqslant 0$
   \[f_{n,i}\colon\uhom(\PP^{\w n},T^i\w E_n)\xrightarrow{tw}
     \uhom(\PP^{\w n},E_n\w T^i)\xrightarrow{u}\uhom(\PP^{\w n},E_{n+i}),\]
where the first map is induced by twist $T^i\w E_n\to E_n\w T^i$.

\begin{lem}\label{FrEsym}
Each map $f_n$, $n\geqslant 0$, is a morphism of spectra commuting
with stabilization maps $\Fr_n^E(S_T)\to\Fr_{n+1}^E(S_T)$ and
$\Theta^n_{sym}(E)\to\Theta^{n+1}_{sym}(E)$. In particular, they
induce a map of spectra
\[f\colon \Fr^E(S_T)\to\Theta^{\infty}_{sym}(E).\]
\end{lem}

\begin{proof}
The following diagram commutes:
\[
\xymatrix{
\uhom(\PP^{\w n},T^i\w E_n)\w T\ar[r]^{tw}\ar[d] &
\uhom(\PP^{\w n},E_n\w T^i)\w T\ar[r]^{u}\ar[d] & \uhom(\PP^{\w n},E_{n+i})\w T\ar[d]\\
\uhom(\PP^{\w n},T^{i+1}\w E_n)\ar[r]^{tw} & \uhom(\PP^{\w n},E_n\w
T^{i+1})\ar[r]^{u} &\uhom(\PP^{\w n},E_{n+i+1}),}
\]
where the left vertical arrow is the $i$th bonding map of the
spectrum $\Fr_n^E(S_T)$, and the right vertical map is the $i$th
bonding map of $\Theta_n^{sym}(E)$. We see that each map $f_n$ is a
morphism of spectra. Consider a commutative diagram
\[
\xymatrix{
\uhom(\PP^{\w n},T^i\w E_n)\ar[r]^{tw}\ar[d] & \uhom(\PP^{\w n},E_n\w T^i)\ar[r]^{u}\ar[d]
& \uhom(\PP^{\w n},E_{n+i})\ar[d]\\
\uhom(\PP^{\w n+1},T^{i}\w E_{n+1})\ar[r]^{tw} & \uhom(\PP^{\w
n+1},E_{n+1}\w T^{i})\ar[r]^{u} &\uhom(\PP^{\w n+1},E_{n+1+i}), }
\]
in which the left vertical map is the stabilization
$\Fr_n^E(S_T)_i\to\Fr_{n+1}^E(S_T)_i$ from Definition~\ref{FrnE},
and the right vertical map is the stabilization map
$\Theta^n_{sym}(E)_i\to\Theta^{n+1}_{sym}(E)_i$ (see
Remark~\ref{rem:stabilization}). The middle vertical arrow equals
the composite map\footnotesize
   $$\uhom(\PP^{\w n},E_n\w T^i)\xrightarrow{-\w\sigma}\uhom(\PP^{\w n+1},E_n\w T^{i+1})
     \xrightarrow{(\chi_{i,1})_*}\uhom(\PP^{\w n+1},E_n\w T^{1+i})\xrightarrow{u_n}\uhom(\PP^{\w n+1},E_{n+1}\w T^i),$$
\normalsize where $(\chi_{i,1})_*$ is induced by the shuffle map
$\chi_{i,1}:T^{i+1}\to T^{1+i}$. For commutativity of the right
square we also use here the fact that the diagram
\[
\xymatrix{E_n\w T^{i+1}\ar[r]^{u}\ar[d]_{id\w\chi_{i,1}}&E_{n+i}\w T\ar[r]^{u}&E_{n+i+1}\ar[d]^{1\oplus\chi_{i,1}}\\
E_n\w T^{1+i}\ar[r]^{u}&E_{n+1}\w T^{i}\ar[r]^{u}&E_{n+1+i},}
\]
is commutative because the compositions of horizontal maps are
$\Sigma_n\times\Sigma_{i+1}$-equivariant maps. Thus the maps
$f_{n,i}$ are compatible with stabilization.
\end{proof}

\begin{cor}
If $E=X_+\w S_T$ then the map $f$ of Lemma~\ref{FrEsym} gives an
isomorphism of spectra $\Fr_{\PPo,T}(X)=\Fr^{X_+\w
E}(S_T)\xrightarrow{\cong}\Theta^{\infty}_{sym}(X_+\w S_T)$.
\end{cor}

\begin{proof}
It suffices to note that the bonding maps of $X_+\w S_T$ are
isomorphisms.
\end{proof}

\begin{prop}\label{fr=true}
For a symmetric Thom $T$-spectrum $E$ with the bounding constant $d$ and
contractible alternating group action, the map $f$ induces a local
weak equivalence for any $i\geqslant\max(0,d)$:
\[f_i\colon C_*\Fr^E(S_T)_i\to C_*\Theta^{\infty}_{sym}(E)_i.\]
\end{prop}

\begin{proof}
The map $f_{n,i}\colon\Fr^E_n(T^i)\to\Theta^{n}_{sym}(E)_i$ fits
into the following commutative diagram
\[
\xymatrix{
\Fr^E_n(T^i)\ar@{=}[d]\ar[rr]^{f_{n,i}} && \Theta^{n}_{sym}(E)_i\\
\Theta^{n}(T^i\w E)_0\ar[r]^{u_l} & \Theta^{n}(E[i])_0\ar@{=}[r] & \Theta^n(E)_i\ar[u]^{\chi_{i,n}}
}
\]
where $\chi_{i,n}$ is the map of Lemma~\ref{spacetrue=fake} and
$u_l$ is the left bonding map from Definition~\ref{shift}. Note that
the maps of the diagram are compatible with stabilization maps, and
hence we can pass to the colimit over $n$. Thus our assertion
follows from Lemma~\ref{trususpshift} below.
\end{proof}

\begin{lem}\label{fakesuspshift}
The natural map of $T$-spectra $u\colon E\w T^i\to E[i]$ induces a
levelwise isomorphism of spaces $\Theta^{\infty}(E\w
T^i)\xrightarrow{\cong}\Theta^{\infty}(E[i])$ for any $T$-spectrum
$E$.
\end{lem}

\begin{proof}
For any $m$ the map $\uhom(\PP^{\w n},E_{n+m}\w
T^i)\stackrel{u}\to\uhom(\PP^{\w n},E_{n+m+i})$ commutes with
stabilization by $n$ and induces an isomorphism on colimits
$\Theta^{\infty}(E\w T^i)_m\to\Theta^{\infty}(E[i])_m.$
\end{proof}

\begin{lem}\label{trususpshift}
For a symmetric Thom $T$-spectrum $E$ with the bounding constant $d$ and
contractible alternating group action, the map of spectra $u_l\colon
T^i\w E\to E[i]$ induces a local weak equivalence of spaces
\[C_*\Theta^{\infty}(T^i\w E)_0\to C_*\Theta^{\infty}(E[i])_0\]
for any $i\geqslant\max(0,d).$
\end{lem}

\begin{proof}
For $i\geqslant\max(0,d)$ the space $C_*\Theta^{\infty}(E)_i$ is locally
connected by Lemma~\ref{fakeconnected}. We apply Lemma~\ref{colimeq}
below to spaces
\[A_n=C_*\Theta^{2n}(T^i\w E)_0=C_*\uhom(\PP^{\w 2n},T^i\w E_{2n}),\]
\[B_n=C_*\Theta^{2n}(E\w T^i)_0=C_*\uhom(\PP^{\w 2n},E_{2n}\w T^i),\]
where the maps $i_n^A$ and $i_n^B$ from that lemma are induced by
stabilization maps $\Theta^{2n}\to\Theta^{2n+2}$ and
$C=C_*\Theta^{\infty}(E[i])_0$. Define $f_n\colon A_n\to B_n$ to be
the map induced by the twist $tw\colon T^i\w E_{2n}\to E_{2n}\w
T^i$. Then the composition $f_{n+1}\circ i_n^A$ coincides with the
composition
\[f_{n+1}\circ i_n^A\colon C_*\uhom(\PP^{\w 2n},T^i\w E_{2n})
\to C_*\uhom(\PP^{\w 2n+2},E_{2n}\w T^2\w T^i)\to C_*\uhom(\PP^{\w
2n+2},E_{2n+2}\w T^i).\]
It differs from the composition $i_n^B\circ
f_n$ by the permutation $E_{2n}\w T^2\w T^i\stackrel{tw}\to E_{2n}\w
T^i\w T^2\stackrel{assoc}= E_{2n}\w T^2\w T^i$. Since it is an even
permutation and $E$ is a spectrum with contractible alternating
group action, we have that $f_{n+1}\circ i_n^A$ and $i_n^B\circ f_n$
are simplicially homotopic. Similarly, in the triangle
\[\xymatrix{
C_*(\PP^{\w 2n},T^i\w E_{2n})\ar[d]^{tw}\ar[r]^{u_l} & C_*(\PP^{\w 2n},E_{2n+i})\ar[r] & C_*\Theta^{\infty}(E[i])_0\\
C_*(\PP^{\w 2n},E_{2n}\w T^i)\ar[ru]^{u}}\]
the composition $u\circ
tw$ differs from $u_l$ by the action of the shuffle permutation
$\chi_{2n,i}$ on $E_{2n+i}$, which is an even permutation. Thus the
triangle commutes up to simplicial homotopy, because $E$ is a
spectrum with contractible alternating group action. Then for any
$i\geqslant d$ the space $C_*\Theta^{\infty}(E[i])_0$ is connected
and our statement follows from Lemmas~\ref{colimeq}
and~\ref{fakesuspshift}.
\end{proof}

\begin{lem}\label{colimeq}
Suppose $i_n^A\colon A_n\to A_{n+1},i_n^B\colon B_n\to B_{n+1}$ are
directed sequences of spaces, $C$ is a locally connected space and
there are maps of sequences $A_n\to C$ and $B_n\to C$. Suppose that
for any $n$ there is a local weak equivalence $f_n\colon A_n\to B_n$ such that the
diagrams
\[
\xymatrix{A_n\ar[r]\ar[d]^{f_n} & A_{n+1}\ar[d]^{f_{n+1}} &&\text{ and }&& A_n\ar[r]\ar[d]^{f_n} & C\\
B_n\ar[r] & B_{n+1} &&&& B_n\ar[ru] }\]
commute up to a simplicial
homotopy. Let $A=\colim A_n, B=\colim B_n$. Then the map $B\to C$ is
a local weak equivalence if and only if so is the map $A\to C$.
\end{lem}

\begin{proof}
Given local Henselian scheme $U$ and $i\geqslant 0$, the maps $\pi_i(f_n(U))$ form a
map of sequences $\pi_i(f_n(U))\colon \pi_i(A_n(U))\to\pi_i(B_n(U))$, and the
map $\colim_n\pi_i(f_n(U))$ fits into the commutative diagram
\[\xymatrix{
\pi_i(A(U))\ar[r]\ar[d]_{\colim\pi_i(f_n)} & \pi_i(C(U))\\
\pi_i(B(U))\ar[ru] }\]
Every $\pi_i(f_n(U))$ is a bijection, and
hence so is $\colim \pi_i(f_n(U))$. If the map $B\to C$ is a local
weak equivalence, then $B(U)$ is connected and all maps
$\pi_i(B(U))\to\pi_i(C(U))$ are bijective. Then $A(U)$ is connected,
and all the maps $\pi_i(A(U))\to\pi_i(C(U))$ are bijective.
Therefore the map $A\to C$ is a local weak equivalence. Similarly,
if $A\to C$ is a local weak equivalence, then so is $B\to C$.
\end{proof}

\section{Fibrant resolutions of symmetric Thom spectra}

We have discussed three types of spectra associated with a symmetric
Thom $T$-spectrum $E$ each of which is obtained from $E$ by a
certain stabilization and taking the Suslin complex at each level:
$C_*\Fr^E(S_T)$, $C_*\Theta^{\infty}(E)$ and
$C_*\Theta^{\infty}_{sym}(E)$. Moreover, by
Propositions~\ref{truefakezigzag} and~\ref{fr=true} they are
isomorphic to each other in $SH(k)$ under certain reasonable
assumptions on $E$. The next theorem says that if we take local
fibrant replacements at each level in these spectra, they become
motivically fibrant starting from some level $d$ onwards. More
precisely, the following result is true:

\begin{thm}\label{main}
For a symmetric Thom $T$-spectrum $E$ with the bounding constant $d$
and contractible alternating group action the following
$\PPo$-spectra are isomorphic to $E$ in $SH(k)$ and motivically
fibrant starting from level $\max(0,d)$:
\begin{itemize}
\item $C_*\Fr^E(S_T)^f$
\item $C_*\Theta^{\infty}(E)^f$
\item $C_*\Theta^{\infty}_{sym}(E)^f$,
\end{itemize}
where ``$f$'' refers to levelwise local fibrant replacements of the
corresponding spectra.
\end{thm}

\begin{proof}
By Propositions~\ref{truefakezigzag} and~\ref{fr=true} we have the
following levelwise local weak equivalences of $\PPo$-spectra
starting from level $d$:
\[C_*\Fr^E(S_T)\to C_*\Theta^{\infty}_{sym}(E)\to
C_*\Theta^{\infty}(\Theta^{\infty}_{sym}(E))\leftarrow
C_*\Theta^{\infty}(E).\]
Since the canonical map $E\to
C_*\Theta^{\infty}(E)$ is a stable motivic equivalence by
Theorem~\ref{thmthom}, we see that $E$ is isomorphic in $SH(k)$ to
each of the spectrum of the theorem.

\begin{sublem}\label{leveleq}
Suppose a map of $\PPo$-spectra $f\colon E\to E'$ is a levelwise
local weak equivalence and all spaces $E_i,E'_i$, $i\geqslant 0$,
are fibrant in the flasque local model structure. Then $E$ is
motivically fibrant if and only if so is $E'$.
\end{sublem}

\begin{proof}
Since each map $f_i\colon E_i\to E'_i$, $i\geqslant 0$, is a local
weak equivalence between locally fibrant spaces, it is a sectionwise
weak equivalence. Therefore if $E_i$ is motivically fibrant, then
$E'_i$ is $\A^1$-invariant, and hence motivically fibrant as well.
Since $E_i,E'_i$ are flasque fibrant by assumption, then the map
$\uhom(\PP^{\w1},E_{i+1})\to\uhom(\PP^{\w1},E'_{i+1})$ is a
sectionwise weak equivalence, because $\PP^{\w1}$ is flasque
cofibrant. Hence the adjoint to the bonding map
$E_i\to\uhom(\PP^{\w1},E_{i+1})$ is a sectionwise weak equivalence
if and only if so is the map $E'_i\to\uhom(\PP^{\w1},E'_{i+1})$.
\end{proof}

Since the spectrum $C_*\Theta^{\infty}(E)^f$ is motivically fibrant
starting from level $d$ by Theorem~\ref{thmthom}, then so are
$C_*\Theta^{\infty}_{sym}(E)^f$ and
$C_*\Theta^{\infty}(\Theta^{\infty}_{sym}(E))^f$ by
Proposition~\ref{truefakezigzag} and the sublemma above. Likewise,
$C_*\Fr^E(S_T)^f$ is motivically fibrant starting from level $d$ by
Proposition~\ref{fr=true} and the sublemma above. This completes the
proof of the theorem.
\end{proof}

The spectrum $C_*\Fr^E(S_T)^f$, which is isomorphic to $E$ in
$SH(k)$ by Theorem~\ref{main}, is of particular interest, because it
will lead to an equivalent model of $E$ in the category of
$(S^1,\mathbb G_m^{\w 1})$-bispectra (see Theorem~\ref{mainbisp}).

\section{$E$-framed motives and bispectra}\label{sectionefr}

Following Definition~\ref{frnex}, for any space $\mathcal X$ and any
$T$-spectrum $E$ denote by $\Fr_n^E(\mathcal X)=\uhom(\PP^{\w
n},\mathcal X\w E_n)$, $n\geqslant 0$. Also, set $\Fr^E(\mathcal
X)=\colim_n(\uhom(\PP^{\w n},\mathcal X\w
E_n))=\Theta^\infty(\mathcal X\w E)_0$. If $\mathcal X=X_+$, $X\in
\Sm_k$, then we shall write $\Fr^E(X)$ dropping $+$ from notation.

\begin{lem}\label{facts}
$\Fr^E(\mathcal X)$ is functorial in $\mathcal X$ and $E$. If $E$ is
a directed colimit of $T$-spectra $\colim_kE_k$, then
$\Fr^E(\mathcal X)=\colim_k\Fr^{E_k}(\mathcal X)$. In particular,
$\Fr^E(\mathcal X)=\colim_k\Fr^{L_kE}(\mathcal X)$, where $L_kE$ is
the $k$-th layer of $E$.
\end{lem}

\begin{dfn}
Given a $T$-spectrum $E$, the assignment $K\mapsto C_*\Fr^E(\mathcal
X\w K)$ is plainly a $\Gamma$-space. The {\it $E$-framed motive
$M_E(\mathcal X)$ of $\mathcal X$\/} is the Segal symmetric
$S^1$-spectrum $C_*\Fr^E(\mathcal X\w\mathbb S)$. If $E=S_T$ then
$M_E(\mathcal X)$ is the framed motive $M_{fr}(\mathcal X)$ of
$\mathcal X$ in the sense of~\cite{GPMain}.
\end{dfn}

Lemma~\ref{facts} implies the following

\begin{cor}\label{factscor}
$M_E(\mathcal X)$ is functorial in $\mathcal X$ and $E$. If $E$ is a
directed colimit of $T$-spectra $\colim_kE_k$, then $M_E(\mathcal
X)=\colim_kM_{E_k}(\mathcal X)$. In particular, $M_E(\mathcal
X)=\colim_k M_{L_kE}(\mathcal X)$, where $L_kE$ is the $k$-th layer
of $E$.
\end{cor}

The next statement is straightforward.

\begin{lem}\label{factmotive}
$M_{L_kE}(\mathcal X)=\uhom(\PP^{\w k},M_{fr}(\mathcal X\w E_k))$
for any $k\geqslant 0$.
\end{lem}

\begin{dfn}
Given a Thom $T$-spectrum $E$, $U\in \Sm_k$ and $Y=X/(X-Z)$, where
$X\in \Sm_k$ and $Z$ is a closed subset in $X$, denote by
$\ZF_n^E(U,Y)$ the free Abelian group generated by the elements of
$\Fr_n^E(U,Y)=\uhom(\PP^{\w n},Y\w E_n)$ with connected support
(recall that the elements of $\Fr_n^E(U,Y)$ have an explicit
geometric description using Voevodsky's Lemma~\ref{Voevlemma}). We
also set $\ZF^E(U,Y):=\colim_n\ZF_n^E(U,Y)$, where the colimit maps
are defined in the same fashion with those of $\Fr^E(U,Y)$.

The assignment $K\mapsto C_*\ZF^E(U,Y\w K)$ is plainly a
$\Gamma$-space. The {\it linear $E$-framed motive $LM_E(Y)$ of
$Y$\/} is the Segal symmetric $S^1$-spectrum $C_*\ZF^E(Y\w\mathbb
S)$. If $E=S_T$ then $LM_E(Y)$ is the linear framed motive
$LM_{fr}(Y)$ of $Y$ in the sense of~\cite{GPMain}. Note that the
presheaves of stable homotopy groups $\pi_*(LM_E(Y))$ are computed
as the presheaves of homology groups of the complex $C_*\ZF^E(Y)$
(we freely use the Dold--Kan correspondence here).
\end{dfn}

As above we have the following

\begin{lem}\label{factlmotive}
$LM_E(Y)=\colim_k LM_{L_kE}(Y)$ and $LM_{L_kE}(Y)=LM_{fr}^{\PP^{\w k}}(Y\w E_k)$,
where $LM_{fr}^{\PP^{\w k}}(Y\w E_k)$ is Segal's spectrum associated with
the $\Gamma$-space $K\mapsto C_*\ZF^{\PP^{\w k}}(U,Y\w K)$ (see Definition~\ref{zfrpi}).
\end{lem}

The following lemma says that $LM_E(Y)$ computes homology of the
$E$-framed motive of $Y$.

\begin{lem}\label{homolmotive}
Given a Thom $T$-spectrum $E$ and $Y$ as above,
there is an isomorphism of graded presheaves
$\pi_*(\Z M_{E}(Y))=\pi_*(LM_E(Y))$.
\end{lem}

\begin{proof}
We have that $\pi_*(\Z M_{E}(Y))=\colim_k\pi_*(\Z
M_{L_kE}(Y))=\colim_k\pi_*(\Z(\uhom(\PP^{\w k},M_{fr}(Y\w E_k))))
=\colim_k\pi_*(LM_{fr}^{\PP^{\w k}}(Y\w E_k))=\pi_*(LM_E(Y))$.
We have used here Corollary~\ref{factscor}, Lemmas~\ref{zfrhomol}, \ref{factmotive} and~\ref{factlmotive}.
\end{proof}

Following notation of~\cite{GNP}, denote by $\A^1//\Gm$ the mapping
cone of the natural embedding $(\Gm)_+\hookrightarrow\A^1_+$. It is
represented by a simplicial scheme from $\Fr_0(k)$.

\begin{lem}\label{a1gm}
For a Thom $T$-spectrum $E$ with the bounding constant $d$, the natural map
\[M_E(T^{\ell}\w(\A^1//\Gm)^{\w i})\to M_E(T^{\ell}\w T^i),\quad \ell:=\max(0,d-1),\]
is a local stable weak equivalence for any $i> 0.$
\end{lem}

\begin{proof}
By Corollary~\ref{factscor} and Lemma~\ref{factmotive}
$M_E(T^{\ell}\w(\A^1//\Gm)^{\w i})=\colim_n\uhom(\PP^{\w
n},M_{fr}(E_n\w T^{\ell}\w(\A^1//\Gm)^{\w i}))$ and $M_E(T^{\ell}\w
T^i)=\colim_n\uhom(\PP^{\w n},M_{fr}(E_n\w T^{\ell}\w T^i))$. Since a
directed colimit of stable local weak equivalences is a stable local
weak equivalence, it is sufficient to check that the natural map
\[\uhom(\PP^{\w n},M_{fr}(E_n\w T^{\ell}\w(\A^1//\Gm)^{\w i}))\to \uhom(\PP^{\w n},M_{fr}(E_n\w T^{\ell}\w T^i))\]
is a local stable weak equivalence of spectra. By definition of the
bounding constant $d$, the space $E_n\w T^{\ell}$ is a colimit of
spaces of the form $X/X-Z$ where $Z$ has codimension greater than or
equal to $n$. Consider a commutative diagram
\[
\xymatrix{
\uhom(\PP^{\w n},M_{fr}(E_n\w T^{\ell}\w(\A^1//\Gm)^{\w i}))\ar[d]\ar[r] &
\uhom(\PP^{\w n},M_{fr}(E_n\w T^{\ell}\w T^i))\ar[d]\\
\uhom(\PP^{\w n},M_{fr}(E_n\w T^{\ell}\w(\A^1//\Gm)^{\w i})_f)\ar[r]
&  \uhom(\PP^{\w n},M_{fr}(E_n\w T^{\ell}\w T^i)_f), }
\]
where ``$f$'' refers to a level local fibrant replacement.
Then by Proposition~\ref{spectraeq} the vertical arrows are local
stable weak equivalences, and the bottom arrow is a stable weak
equivalence between motivically fibrant $S^1$-spectra
by~\cite[1.1; A.1]{GNP}. By the two-out-of-three-property the upper arrow is a
local stable weak equivalence.
\end{proof}

\begin{prop}\label{efrcompar}
Let $E$ be a Thom $T$-spectrum with the bounding constant $d$
and $\ell=\max(0,d-1)$. Then for every $i>0$ the
$S^1$-spectra $M_E(T^{\ell}\w(\A^1//\Gm)^{\w i})_f$, $M_E(T^{\ell}\w
T^i)_f$ and $M_E(T^{\ell}\w(S^1\w\Gm^{\w 1})^{\w i})_f$, where ``$f$'' refers to a level local fibrant
replacement, are motivically fibrant.
\end{prop}

\begin{proof}
Without loss of generality we assume $d\leqslant 1$. Indeed, if
$d>1$ we replace $E$ by $T^{d-1}\w E$, which has the bounding
constant 1, observing that $M_E(T^{d-1}\w-)\cong M_{T^{d-1}\w
E}(-)$. It suffices to prove the statement for the spectrum
$M_E((\A^1//\Gm)^{\w i})_f$, because our arguments will be the same
for the other two spectra.

By Corollary~\ref{factscor} and Lemma~\ref{factmotive} one has
$M_E((\A^1//\Gm)^{\w i})=\colim_n\uhom(\PP^{\w
n},M_{fr}(E_n\w(\A^1//\Gm)^{\w i}))$. Since flasque motivically
fibrant spectra are closed under filtered colimits, it is enough to
show that each spectrum $\uhom(\PP^{\w
n},M_{fr}(E_n\w(\A^1//\Gm)^{\w i}))_f$ is motivically fibrant. But
this can shown similarly to Lemma~\ref{omegas} if we note that the
space $C_*\Fr(X_+\w(\A^1//\Gm)^{\w i})$, $X\in \Sm_k$, is
locally connected by~\cite[A.1]{GNP} and if we apply the proof of
Lemma~\ref{pi0} to show that the space $\uhom(\PP^{\w
n},C_*\Fr(E_n\w(\A^1//\Gm)^{\w i}))$ is locally connected.
\end{proof}

\begin{lem}\label{efrcompar2}
Under the assumptions of Proposition~\ref{efrcompar} the map
$M_E(T^{\ell}\w(\A^1//\Gm)^{\w i})_f\to
M_E(T^{\ell}\w(S^1\w\Gm^{\w 1})^{\w i})_f$ is a
sectionwise level equivalence for $\ell=\max(0,d-1)$.
\end{lem}

\begin{proof}
The proof of Proposition~\ref{efrcompar} shows that it suffices to
prove the assertion for $d\leqslant 1$ and that the map of
$\uhom(\PP^{\w n},M_{fr}(E_n\w(\A^1//\Gm)^{\w i}))\to
\uhom(\PP^{\w n},M_{fr}(E_n\w(S^1\w\Gm^{\w 1})^{\w i})$ is a local
stable weak equivalence for any $n\geqslant 0$. Since a map of
locally connected spectra is a stable local equivalence if and only
if so is the map on homology, then using Lemma~\ref{zfrhomol} our assertion
reduces to showing that the map of complexes
$C_*\ZF^{\PP^{\w n}}(E_n\w(\A^1//\Gm)^{\w i})\to
C_*\ZF^{\PP^{\w n}}(E_n\w(S^1\w\Gm^{\w 1})^{\w i})$ is locally a
quasi-isomorphism. The latter fact is proved similar
to~\cite[8.2]{GPMain}.
\end{proof}

We shall need the following fact which was proven in~\cite[Section~12]{GPMain}.

\begin{lem}\label{sublem}
Let $\mathcal X$ be an $\A^1$-local motivic $S^1$-spectrum whose
presheaves of stable homotopy groups are homotopy invariant
quasi-stable radditive presheaves with framed correspondences. Suppose $\mathcal
X^f$ is a local stable fibrant replacement of $\mathcal X$. Then the map
of spectra $\underline{\Hom}(\mathbb G_m^{\w 1},\mathcal
X)\to\underline{\Hom}(\mathbb G_m^{\w 1},\mathcal X^f)$ is a local stable
equivalence.
\end{lem}

\begin{prop}\label{cancel}
The following statements are true:

$(1)$ Suppose $Z\to X$ is a closed embedding of smooth varieties of
codimension $d$. Then the natural map
   $$\alpha:\uhom(\PP^{\w i},M_{fr}(X/X-Z))\to\uhom(\PP^{\w i}\w\Gm^{\w 1},M_{fr}((X/X-Z)\w\Gm^{\w 1}))$$
is a stable local weak equivalence of $S^1$-spectra for all
$i\leqslant d$.

$(2)$ If $E$ is a Thom $T$-spectrum with the bounding constant
$d\leqslant 1$, then the natural map
   $$\beta:M_{E}(X)\to\uhom(\Gm^{\w 1},M_{E}(X_+\w\Gm^{\w 1})),\quad X\in \Sm_k,$$
is a stable local weak equivalence of $S^1$-spectra.

$(3)$ If $E$ is a Thom $T$-spectrum with the bounding constant
$d\leqslant 1$ and $M_E(X)_f$ is a stable local fibrant replacement
of $M_E(X)$, $X\in \Sm_k$, then $M_E(X)_f$ is a motivically fibrant
$S^1$-spectrum.
\end{prop}

\begin{proof}
(1). First suppose $i=0$. Without loss of generality we assume that
$X/X-Z=Z_+\w T^d$, because we can apply the Mayer--Vietoris sequence
of Proposition~\ref{MV} to reduce the general case to this
particular one. Indeed, $M_{fr}(X/X-Z)$ and $M_{fr}((X/X-Z)\w\Gm^{\w
1})$ are homotopy pushouts (=homotopy pullbacks) of framed motives
of the form $M_{fr}(Z_+\w T^d)$ and $M_{fr}(Z_+\w T^d\w\Gm^{\w 1})$
respectively. Using Lemma~\ref{sublem} the functor $\uhom(\Gm^{\w 1},-)$ respects
homotopy pullbacks (=pushouts) for framed motives in question.
Therefore our assertion reduces to showing that the natural map of
$S^1$-spectra
   $$M_{fr}(Z_+\w T^d)\to\uhom(\Gm^{\w 1},M_{fr}(Z_+\w T^d\w\Gm^{\w 1}))$$
is a stable local weak equivalence. It fits into a commutative
diagram
   $$\begin{xymatrix}{
      M_{fr}(Z_+\w T^d)\ar[r]\ar[d]&\uhom(\Gm^{\w 1},M_{fr}(Z_+\w T^d\w\Gm^{\w 1}))\ar[d]\\
      M_{fr}(Z_+\w(\A^1//\Gm)^{\w d})\ar[r]&\uhom(\Gm^{\w 1},M_{fr}(Z_+\w(\A^1//\Gm)^{\w d}\w\Gm^{\w 1}))
     }\end{xymatrix}$$
in which the lower arrow is a stable local weak equivalence by the
Cancellation Theorem for framed motives of~\cite{AGP}. The vertical
arrows are stable local weak equivalences by~\cite[1.1]{GNP}
and Lemma~\ref{sublem}. We see that the
upper arrow is a stable local weak equivalence, as required.

Suppose now $i\leqslant d$. Consider a commutative diagram of
$S^1$-spectra
   $$\begin{xymatrix}{
      \uhom(\PP^{\w i},M_{fr}(X/X-Z))\ar[r]\ar[d]&\uhom(\PP^{\w i}\w\Gm^{\w 1},M_{fr}((X/X-Z)\w\Gm^{\w 1}))\ar[d]\\
      \uhom(\PP^{\w i},M_{fr}(X/X-Z)_f)\ar[r]&\uhom(\PP^{\w i}\w\Gm^{\w 1},M_{fr}((X/X-Z)\w\Gm^{\w 1})_f)
     }\end{xymatrix}$$
where ``$f$" refers to a level local fibrant replacement. By
Proposition~\ref{spectraeq} the left vertical map is a levelwise
local weak equivalence in positive degrees, and hence a stable local
weak equivalence. Since $M_{fr}(X/X-Z)_f,M_{fr}((X/X-Z)\w\Gm^{\w
1})_f$ are motivically fibrant spectra in positive degrees, the
lower arrow is a sectionwise weak equivalence in positive degrees by
the first assertion and Proposition~\ref{spectraeq}, and hence a
sectionwise stable weak equivalence. The right vertical map is a
levelwise local weak equivalence in positive degrees by Lemma~\ref{sublem} and
Proposition~\ref{spectraeq}, hence it is a stable local weak
equivalence. We see that the upper arrow is a stable local weak
equivalence.

(2). If $E$ is a Thom $T$-spectrum with the bounding constant
$d\leqslant 1$, then the natural map
   $$\beta:M_{E}(X)\to\uhom(\Gm^{\w 1},M_{E}(X_+\w\Gm^{\w 1}))$$
is isomorphic to the sequential colimit of maps
   $$\beta_k:M_{L_k E}(X)\to\uhom(\Gm^{\w 1},M_{L_k E}(X_+\w\Gm^{\w 1})).$$
Every such map is isomorphic to
   $$\uhom(\PP^{\w k},M_{fr}(X_+\w E_k))\to\uhom(\PP^{\w k}\w\Gm^{\w 1},M_{fr}(X_+\w E_k\w\Gm^{\w 1})).$$
It follows from assertion~(1) that each $\beta_k$ is a stable local
weak equivalence of $S^1$-spectra, and hence so is $\beta$ as a
sequential colimit of stable local weak equivalences of
$S^1$-spectra.

(3). We can compute a stable local fibrant replacement $M_E(X)_f$ of
$M_E(X)$ as the spectrum $\colim_k M_{L_k
E}(X)_f=\colim_k\uhom(\PP^{\w k},M_{fr}(X_+\w E_k))_f$ (we use Lemma~\ref{factmotive} here), because a
sequential colimit of fibrant spectra is fibrant in the flasque
local stable model model structure of $S^1$-spectra. So it suffices
to show that each spectrum $\uhom(\PP^{\w k},M_{fr}(X_+\w E_k))_f$
is $\A^1$-invariant. By the Mayer--Vietoris sequence of
Proposition~\ref{MV} this reduces to showing that every spectrum of
the form $\uhom(\PP^{\w k},M_{fr}(X_+\w T^k))_f$ is $\A^1$-invariant.
But the latter is obvious because $\uhom(\PP^{\w k},M_{fr}(X_+\w
T^k))_f\cong M_{fr}(X)_f$ and $M_{fr}(X)_f$ is $\A^1$-invariant
by~\cite[7.1]{GPMain}.
\end{proof}

In what follows by bispectra we shall mean $(S^1,\Gm^{\w
1})$-bispectra in the category of motivic spaces. We are now in a
position to prove the main result of the section. It gives an
explicit fibrant resolution of a symmetric Thom spectrum in the
category of bispectra.

\begin{thm}\label{mainbisp}
Suppose $X\in \Sm_k$ and $E$ is a symmetric Thom $T$-spectrum with
the bounding constant $d$ and contractible alternating group action.

$(1)$ If $d=1$ then the $(S^1,\Gm^{\w 1})$-bispectrum
\[M_E^{\mathbb G}(X)_f:=(M_E(X)_f,M_E(X_+\w\Gm^{\w 1})_f,M_E(X_+\w\Gm^{\w 2})_f,\ldots)\]
is motivically fibrant and represents the $T$-spectrum $X_+\w E$ in
the category of bispectra, where ``$f$'' refers to stable local
fibrant replacements of $S^1$-spectra.

$(2)$ If $d<1$ then the $(S^1,\Gm^{\w 1})$-bispectrum
\[M_E^{\mathbb G}(X)_f:=(M_E(X)_f,M_E(X_+\w\Gm^{\w 1})_f,M_E(X_+\w\Gm^{\w 2})_f,\ldots)\]
is motivically fibrant and represents the $T$-spectrum $X_+\w E$ in
the category of bispectra, where ``$f$'' refers to level local
fibrant replacements of $S^1$-spectra.

$(3)$ If $d>1$ then the $(S^1,\Gm^{\w 1})$-bispectrum
\[\Omega_{S^{1}\w\Gm^{\w 1}}^{d-1}((M_{E[d-1]}(X)_f,M_{E[d-1]}(X_+\w\Gm^{\w 1})_f,
M_{E[d-1]}(X_+\w\Gm^{\w 2})_f,\ldots))\]
is motivically fibrant and represents the $T$-spectrum $X_+\w E$ in
the category of bispectra, where ``$f$'' refers to stable local
fibrant replacements of $S^1$-spectra. Here $E[d-1]$ stands for the
$(d-1)$-th shift of $E$ in the sense of Definition~\ref{shift}.
Another equivalent model for the $T$-spectrum $X_+\w E$ in the
category of bispectra is given by
\[\Omega_{S^{1}\w\Gm^{\w 1}}^{d-1}((M_{T^{d-1}\w E}(X)_f,
M_{T^{d-1}\w E}(X_+\w\Gm^{\w 1})_f,M_{T^{d-1}\w E}(X_+\w\Gm^{\w
2})_f,\ldots)).\] This bispectrum is motivically fibrant and ``$f$''
refers to stable local fibrant replacements of $S^1$-spectra.
\end{thm}

\begin{proof}
(1). The fact that the bispectrum $M_E^{\mathbb G}(X)_f$ is
motivically fibrant follows from Proposition~\ref{cancel} and
Lemma~\ref{sublem}. It remains to show
that it represent $E$ in the category of bispectra.

Without loss of generality we may assume $X=pt$, because we can
replace $E$ with $X_+\w E$. It follows from Theorem~\ref{main} that
$E$ is isomorphic in $SH_T(k)$ to $C_*\Fr^E(S_T)^f$. The latter is a
$T$-spectrum (see Definition~\ref{FrnE}). It is positively fibrant
by Theorem~\ref{main}. The motivic equivalence $\widetilde
T:=\A^1//\Gm\to T$ induces an equivalence of categories
$SH_T(k)\xrightarrow{\simeq}SH_{\widetilde T}(k)$. It takes $E$ to a
$\widetilde T$-spectrum isomorphic to $C_*\Fr^E(S_T)^f$. By
Lemma~\ref{a1gm} and Proposition~\ref{efrcompar} the natural map
$C_*\Fr^E(S_{\widetilde T})^f\to C_*\Fr^E(S_T)^f$ is a sectionwise
weak equivalence in positive degrees, where $S_{\widetilde
T}=(S^0,\widetilde T,\widetilde T^{\w 2},\ldots)$. By the sublemma
on p.~\pageref{leveleq} $C_*\Fr^E(S_{\widetilde T})^f$ is a
motivically fibrant $\widetilde T$-spectrum in positive degrees
(notice that each space $C_*\Fr^E(S_{\widetilde T})^f_{i>0}$ is
motivically fibrant by Proposition~\ref{efrcompar}). We see that
$C_*\Fr^E(S_{\widetilde T})^f$ is a positively fibrant $\widetilde
T$-spectrum representing $E$ in $SH_{\widetilde T}(k)$. Consider now
the canonical motivic weak equivalence $\widetilde T\to S^1\w\Gm^{\w
1}$. For the same reasons $C_*\Fr^E(S_{S^1\w\Gm^{\w 1}})^f$ is a
positively fibrant $\widetilde T$-spectrum which is sectionwise
weakly equivalent to $C_*\Fr^E(S_{\widetilde T})^f$. We use here
Lemma~\ref{efrcompar2} as well. It follows that
$C_*\Fr^E(S_{S^1\w\Gm^{\w 1}})^f$ is a positively fibrant
$S^1\w\Gm^{\w 1}$-spectrum representing $E$ in $SH_{S^1\w\Gm^{\w
1}}(k)$. It remains to observe that this spectrum is equivalent to
the diagonal spectrum for the bispectrum
   $$M_E^{\mathbb G}(X):=(M_E(X),M_E(X_+\w\Gm^{\w 1}),M_E(X_+\w\Gm^{\w 2}),\ldots).$$

(2). This immediately follows from (1) if we observe that a
levelwise local fibrant replacement of each weighted $E$-framed
motive $M_E(X_+\w\Gm^{\w n})_f$, $n\geqslant 0$, is a motivically
fibrant $S^1$-spectrum. To see the latter, we repeat the proof of
Proposition~\ref{cancel}(3) and apply Lemma~\ref{omegas}.

(3). This follows from~(1) if we observe that $E[d-1]$ and
$T^{d-1}\w E$ are symmetric Thom $T$-spectra with the bounding
constant $d=1$.
\end{proof}

We finish the section by proving the following useful result.

\begin{thm}\label{mainbispapp}
Suppose $E$ is a symmetric Thom $T$-spectrum with
the bounding constant $d\leqslant 1$ and contractible alternating group action.

$(1)$ For every elementary Nisnevich square
   $$\xymatrix{U'\ar[r]\ar[d]&X'\ar[d]\\
                       U\ar[r]&X}$$
the square of $S^1$-spectra
   $$\xymatrix{M_{E}(U')\ar[r]\ar[d]&M_{E}(X')\ar[d]\\
                       M_{E}(U)\ar[r]&M_{E}(X)}$$
is homotopy cartesian locally in the Nisnevich topology. 

$(2)$ The natural map $M_{E}(X\times\A^1)\to
M_{E}(X)$ is a stable local weak equivalence of
$S^1$-spectra.

The same is also true for linear $E$-framed motives.
\end{thm}

\begin{proof}
(1). The square of motivic $T$-spectra
   $$\xymatrix{U'_+\wedge E\ar[r]\ar[d]&X'_+\wedge E\ar[d]\\
                       U_+\wedge E\ar[r]&X_+\wedge E}$$
is homotopy cartesian in the stable motivic model structure. By Theorem~\ref{mainbisp}
it induces a homotopy cartesian square of motivically fibrant bispectra
   $$\xymatrix{M_E^{\mathbb G}(U')_f\ar[r]\ar[d]&M_E^{\mathbb G}(X')_f\ar[d]\\
                       M_E^{\mathbb G}(U)_f\ar[r]&M_E^{\mathbb G}(X)_f}.$$
Here `$f$' refers to local replacements in each weight (see Theorem~\ref{mainbisp}). Passing to weight zero
motivic $S^1$-spectra, one gets a homotopy cartesian square of motivically fibrant $S^1$-spectra
   $$\xymatrix{M_E(U')_f\ar[r]\ar[d]&M_E(X')_f\ar[d]\\
                       M_E(U)_f\ar[r]&M_E(X)_f}$$
Our statement now follows.

(2). It is proven similarly to (1) if we start with the stable motivic equivalence of $T$-spectra
$(X\times\mathbb A^1)_+\wedge E\to X_+\wedge E$.

The same statements for linear $E$-framed motives follow from (1), (2)
and Lemma~\ref{homolmotive}.
\end{proof}

\section{Topological Thom spectra with finite coefficients}

In this section we give a topological application of Theorem~\ref{mainbisp}. Namely,
many important topological Thom spectra like $MU$ can be obtained as the the realization
of their motivic counterparts if the base field is $\mathbb C$. We shall prove below that
the stable homotopy groups of such topological Thom spectra with finite coefficients can
be computed by means of the stable homotopy groups with finite coefficients of weight zero
of the associated $E$-framed motive, which is an {\it explicit\/} positively fibrant
$S^1$-spectrum by the very construction.
We first need a couple of useful lemmas.

By $f_0(SH(k))$ we shall mean the full triangulated subcategory of effective $T$-spectra, i.e.,
the subcategory which is compactly generated by the suspension $T$-spectra of the smooth algebraic varieties.
We shall also write $f_\ell(SH(k))$, $\ell\in\mathbb Z$, to denote $f_0(SH(k))\wedge T^{\ell}$.

\begin{lem}\label{fdsh}
Suppose $Z\to X$ is a closed embedding of smooth varieties of codimension $\ell$. Then the suspension
$T$-spectrum $\Sigma^\infty_T(X/X-Z)$ of the sheaf $X/X-Z$ belongs to $f_\ell(SH(k))$.
\end{lem}

\begin{proof}
If $X/X-Z=Z_+\w T^\ell$ then our assertion is trivial. By using
induction, we can cover $X$ by open subsets $X_1,X_2$ such that
$Z_2=Z\cap X_2\to X_2$ is a trivial embedding and $X_1$ is covered
by $n-1$ open trivial pieces. Then for $X_{12}=X_1\cap X_2$ the
embedding $Z_{12}:=Z\cap X_{12}\to X_{12}$ is trivial. Denote by
$Y:=X/X-Z$ and by $Y_i:=X_i/X_i-Z_i$. Then $\Sigma^\infty_T
Y_{12},\Sigma^\infty_T Y_2\in f_\ell(SH(k))$ and $\Sigma^\infty_T
Y_1\in f_\ell(SH(k))$ by induction hypothesis. By
Lemma~\ref{pushout} $Y$ is a pushout of sheaves embeddings
$Y_1\hookleftarrow Y_{12}\hookrightarrow Y_2$. Therefore we have a
triangle in $SH(k)$
   $$\Sigma^\infty_T Y_{12}\to\Sigma^\infty_T Y_{1}\oplus\Sigma^\infty_T Y_{2}\to\Sigma^\infty_T Y\xrightarrow{+}$$
in which the left two entries belong to $f_\ell(SH(k))$. It follows that $\Sigma^\infty_T Y\in f_\ell(SH(k))$.
\end{proof}

\begin{lem}\label{effsp}
Let $E$ be a Thom $T$-spectrum with the bounding constant $d$. Then $E$
belongs to $f_{1-d}(SH(k))$. In particular, $E$ is an effective $T$-spectrum if $d\leqslant 1$.
\end{lem}

\begin{proof}
Since $f_d(SH(k))=f_0(SH(k))\wedge T^d$ for any integer $d$, we may assume $d=1$ and show that
$E$ is an effective $T$-spectrum in this case.

We have $E=\colim_k L_kE$, where each layer has stable homotopy type of $\Omega^k_T((\Sigma^\infty_T E_k)^f)$.
Here $(\Sigma^\infty_T E_k)^f$ stands for a stable motivic fibrant replacement of $\Sigma^\infty_T E_k$. Then
$E$ is isomorphic in $SH(k)$ to $\textrm{hocolim}_k\Omega^k_T((\Sigma^\infty_T E_k)^f)$.

Each $E_k=\colim_j (V_{k,j}/V_{k,j}-Z_{k,j})$, where codimension of
$Z_{k,j}$ in $V_{k,j}$ is larger than or equal to $k$. By
Lemma~\ref{fdsh} the flasque cofibrant $T$-spectrum
$\Sigma_T^{\infty}(V_{k,j}/V_{k,j}-Z_{k,j})$ is in $f_k(SH(k))$.
Since $\Sigma^\infty_T E_k$ is isomorphic in $SH(k)$ to
$\textrm{hocolim}_j (\Sigma_T^{\infty}(V_{k,j}/V_{k,j}-Z_{k,j}))$
and $f_k(SH(k))$ is closed under homotopy colimits, it follows that
$\Sigma^\infty_T E_k\in f_k(SH(k))$, and hence
$\Omega^k_T((\Sigma^\infty_T E_k)^f)\in f_0(SH(k))$. The isomorphism
$E\cong \textrm{hocolim}_k\Omega^k_T((\Sigma^\infty_T E_k)^f)$ in
$SH(k)$ now implies $E\in f_0(SH(k))$.
\end{proof}

Suppose that the base field $k$ has an embedding $\epsilon:k\hookrightarrow\mathbb C$. Following
Panin--Pimenov--R\"ondigs~\cite[\S A4]{PPR1} there is a natural realization functor
   $$Re^\epsilon:SH(k)\to SH,$$
where $SH$ is the homotopy category of the stable model category of classical $S^2$-spectra
of topological spaces (it is canonically equivalent to the homotopy category of the stable model
category of classical $S^1$-spectra as well). $Re^\epsilon$ is an extension of the functor
   $$An:\Sm_k\to\textbf{Top}$$
sending a $k$-smooth variety $X$ to $X^{an}:=X(\mathbb C)$ with the classical topology.

Following Levine's indexing~\cite{Lev}, denote by $\pi_{a,b}^{\A^1}(E)$, where $E\in SH(k)$, the
Nisnevich sheaf on $\Sm_k$ associated to the presheaf $U\mapsto\Hom_{SH(k)}(\Sigma_{S^1}^a\Sigma_{\Gm^{\w 1}}^b U_+,E)$.
For $E\in SH(k)$ (respectively $E\in SH$) and a positive integer $N$, we let $E/N$ denote an object of $SH(k)$
(respectively $E/N\in SH$) that fits into a triangle $E\xrightarrow{N\cdot id}E\to E/N\to E[1]$.
By definition, $\pi_{a,b}^{\A^1}(E;\mathbb Z/N):=\pi_{a,b}^{\A^1}(E/N)$ (respectively $\pi_n(E;\mathbb Z/N):=\pi_n(E/N)$).

We are now in a position to prove the main result of the section.

\begin{thm}\label{fincoeff}
Let $k$ be an algebraically closed field of characteristic zero with
an embedding $\epsilon:k\hookrightarrow\mathbb C$. Suppose E is a
symmetric Thom $T$-spectrum with the bounding constant $d\leqslant
1$ and contractible alternating group action. Then for all integers
$N > 1$ and $n\in\mathbb Z$, the realization functor $Re^\epsilon$
induces an isomorphism
   $$\pi_n(M_E(pt)(pt);\mathbb Z/N)\cong\pi_n(Re^\epsilon(E);\mathbb Z/N)$$
between stable homotopy groups with $\textrm{mod\,} N$ coefficients.
\end{thm}

\begin{proof}
By Lemma~\ref{effsp} $E$ is an effective $T$-spectrum. It follows from~\cite[7.1]{Lev} that the map
   $$\pi_{n,0}^{\A^1}(E;\mathbb Z/N)(pt)\to \pi_n(Re^\epsilon(E);\mathbb Z/N)$$
is an isomorphism for all $n\in\mathbb Z$. Theorem~\ref{mainbisp} implies
that $\pi_{n,0}^{\A^1}(E;\mathbb Z/N)$ is computed as the sheaf $\pi_n^{Nis}(M_E(pt)_f;\mathbb Z/N)$. It
remains to observe that
   $$\pi_n^{Nis}(M_E(pt)_f;\mathbb Z/N)(pt)=\pi_n(M_E(pt);\mathbb Z/N)(pt)=\pi_n(M_E(pt)(pt);\mathbb Z/N),$$
what completes the proof.
\end{proof}

As the realization of $MGL$ is isomorphic to $MU$ in $SH$, the complex
cobordism $S^2$-spectrum, and, by Quillen's Theorem~\cite{Qui},
$\pi_*(MU)$ is isomorphic to the Lazard ring $Laz=\mathbb
Z[x_1,x_2,\ldots]$, $\deg(x_i)=2i$, the preceding theorem implies
the following

\begin{cor}
Let $k$ be an algebraically closed field of characteristic zero with an embedding $\epsilon:k\hookrightarrow\mathbb C$.
For all $n>1$ and $i\in\mathbb Z$, there is an isomorphism $\pi_i(M_{MGL}(pt)(pt);\mathbb Z/n)\cong Laz_i/nLaz_i$,
where $M_{MGL}(pt)$ is the $MGL$-motive of the point $pt=\Spec(k)$.
\end{cor}

We finish the section by the following result about the singular algebraic $E$-homotopy
defined in the introduction.
It is an analogue of the celebrated theorem of
Suslin and Voevodsky~\cite{SV96} on singular algebraic homology.

\begin{thm}\label{sve}
Let $k$ be an algebraically closed field of characteristic zero with
an embedding $\epsilon:k\hookrightarrow\mathbb C$. Suppose E is a
symmetric Thom $T$-spectrum with the bounding constant $d\leqslant
1$ and contractible alternating group action and $X\in Sm/k$.
There are canonical isomorphisms of Abelian groups
\[ \pi^{E}_n(X; \mathbb Z/m) = \pi_n(X(\mathbb C)_+\wedge Re^\epsilon(E); \mathbb Z/m)\]
for all integers $n\geqslant 0$ and $m\neq 0$.

Moreover, if $k$ is any perfect field, then the assignment  
   $$X\mapsto\pi^{E}_{\ast}(X)=\pi_*(Fr^E(\Delta^\bullet_k,X)^{\gp})$$ 
is a generalized homology theory on $Sm/k$.
\end{thm}

\begin{proof}
Since $M_E(X)$ can be identified with $M_{X_+\wedge E}(pt)$ and
$X(\mathbb C)_+\wedge Re^\epsilon(E)\cong Re^\epsilon(X_+\wedge E)$ by~\cite[A.23]{PPR1}, Theorem~\ref{fincoeff} implies
   $$\pi_n(M_E(X)(pt);\mathbb Z/m)\cong\pi_n(X(\mathbb C)_+\wedge Re^\epsilon(E);\mathbb Z/m),\quad n\geqslant 0.$$
We have that 
   $$\pi_n(M_E(X)(pt);\mathbb Z/m)\cong\pi_n(\Omega_{S^1}Fr^E(\Delta^\bullet_k,X\otimes S^1);\mathbb Z/m)
       \cong\pi_n(Fr^E(\Delta^\bullet_k,X)^{\gp};\mathbb Z/m)=\pi_n^E(X;\mathbb Z/m).$$
       
Now the fact that the assignment
   $$X\mapsto\pi^{E}_{\ast}(X)=\pi_*(Fr^E(\Delta^\bullet_k,X)^{\gp})$$ 
is a generalized homology theory on $Sm/k$ with $k$ perfect immediately follows 
from Theorem~\ref{mainbispapp} (verifying the excision property 
and the homotopy invariance property for homology theories).
\end{proof}

\section{Normally $E$-framed motives}\label{normally}

Suppose $E$ is a symmetric Thom $T$-spectrum with contractible
alternating group action and the bounding constant $d=1$. In Theorem~\ref{mainbisp}
we have constructed an explicit fibrant bispectrum representing $E$
in terms of $E$-framed motives. We can simplify $E$-framed motives further by
forgetting a bit of information and construct, up to a local equivalence of $S^1$-spectra,
an equivalent model for them, called normally $E$-framed motives. Then we construct in Theorem~\ref{bisptilde}
an explicit fibrant bispectrum representing $E$ whose entries are expressed in terms of
weighted normally $E$-framed motives. Another advantage of normally $E$-framed motives
is that they lead to representability of important Thom spectra like $MGL$ by schemes
(this material is treated in the next section in details).


\begin{conv}\label{convention}
From now on we shall assume that a symmetric
Thom spectrum $E$ with the bounding constant $d=1$
and contractible alternating group action is of the form:
\begin{itemize}

\item for any $n\geqslant 0$, $E_n=Th(V_n)$ with
$V_n\to Z_n$ a $\Sigma_n$-equivariant vector bundle of rank $n$ over $Z_n\in\Sm_k$;

\item The bonding maps $E_n\w T^m\to E_{n+m}$ are induced by
closed embeddings $i_{n,m}\colon Z_n\to Z_{n+m}$ such that we have a Cartesian square
\[\xymatrix{
V_n\times\A^m\ar[r]^{I_{n,m}}\ar[d] & V_{n+m}\ar[d]\\
Z_n\ar[r]^{i_{n,m}} & Z_{n+m}}\]
and $i_{n+m,r}\circ i_{n,m}=i_{n,m+r}$. Applying the shuffle permutation
$\chi_{n,m}\in\Sigma_{n+m}$, define the {\it left inclusion maps\/} $i_{n,m}^l:=\chi_{n,m}\circ i_{n,m}$.
We require the left inclusion maps $i_{n,m}^l$ to fit into Cartesian squares
\[\xymatrix{
\A^m\times V_n\ar[r]^{I_{n,m}^l}\ar[d] & V_{m+n}\ar[d]\\
Z_n\ar[r]^{i_{n,m}^l} & Z_{m+n}}\]
where $I_{n,m}^l$ is the composition
$\A^m\times V_n\xrightarrow{tw} V_n\times\A^m\xrightarrow{I_{n,m}} V_{n+m}\xrightarrow{\chi_{n,m}} V_{m+n}$.
Observe that the maps $I_{n,m}^l$ induce the left bonding maps $u_l\colon T^m\w E_n\to E_{m+n}$
in the sense of Definition~\ref{shift}.
\end{itemize}
\end{conv}

\begin{rem}
The spectrum $\Sigma^{\infty}_TX_+$ satisfies conditions
of~\ref{convention} with $Z_n=X, V_n=X\times\A^n$. The spectrum
$MGL$ is a directed colimit of spectra of the form~\ref{convention}.
Indeed, for any $i\geqslant 0$ there is a spectrum $E^{(i)}$, where
$E^{(i)}_n=Th(V_n^{(i)})$ and  $V_{n}^{(i)}=\mathcal{T}GL_{n,ni}$ is
the tautological vector bundle over the Grassmannian
$Z_{n}^{(i)}=Gr(n,ni)$. Then the spectra $E^{(i)}$ satisfy
conditions of~\ref{convention} and $MGL=\colim_i E^{(i)}$~\cite[\S
2.1]{PPR}.
\end{rem}

When $E$ is a symmetric Thom $T^2$-spectrum with the bounding constant $d=1$
and contractible alternating group action, we impose analogous conditions:
\begin{itemize}
\item $E_n=Th(V_n)$, where
$V_n\to Z_n$ is a $\Sigma_n$-equivariant vector bundle of rank $2n$
over $Z_n\in\Sm_k$ for any $n\geqslant 0$;

\item The bonding maps $E_n\w T^{2m}\to E_{n+m}$
are induced by closed embeddings $i_{n,m}\colon Z_n\to Z_{n+m}$ such that we have a Cartesian square
\[
\xymatrix{
V_n\times\A^{2m}\ar[r]^{I_{n,m}}\ar[d] & V_{n+m}\ar[d]\\
Z_n\ar[r]^{i_{n,m}} & Z_{n+m}}
\]
and $i_{n+m,r}\circ i_{n,m}=i_{n,m+r}$. Applying the shuffle permutation $\chi_{n,m}\in\Sigma_{n+m}$,
one sets $i_{n,m}^l=\chi_{n,m}\circ i_{n,m}$. The maps $i_{n,m}^l$ are required to fit into Cartesian squares
\[
\xymatrix{
\A^{2m}\times V_n\ar[r]^{I_{n,m}^l}\ar[d] & V_{m+n}\ar[d]\\
Z_n\ar[r]^{i_{n,m}^l} & Z_{m+n}
}
\]
where $I_{n,m}^l$ is the composition
$\A^m\times V_n\xrightarrow{tw} V_n\times\A^m\xrightarrow{I_{n,m}} V_{n+m}\xrightarrow{\chi_{n,m}} V_{m+n}$.
Observe that the maps $I_{n,m}^l$ induce the left bonding maps $u_l\colon T^{2m}\w E_n\to E_{m+n}$
in the sense of Definition~\ref{shift}.
\end{itemize}

\begin{rem}
The $T^2$-spectra $MSL$ and $MSp$ are directed colimits of spectra
that satisfy the above assumptions. Namely, $MSL=\colim E^{(i)}$
where $E^{(i)}_n=Th(V^{(i)}_n)$, $V_{n}^{(i)}=\mathcal{T}SL_{n,ni}$
is the tautological special bundle over the special Grassmannian
$Z_n^{(i)}=SGr(n,ni)$~\cite[\S 4]{PW}. The spectrum $MSp=\colim
E^{(i)}$ where $E^{(i)}_n=Th(V^{(i)}_n)$,
$V_{n}^{(i)}=\mathcal{T}Sp_{n,ni}$ is the tautological symplectic
bundle over the symplectic Grassmannian
$Z_n^{(i)}=HGr(n,ni)$~\cite[\S 6]{PW}.
\end{rem}

As in the previous sections we shall only consider the case of a $T$-spectrum $E$.
The interested reader will easily do the same constructions for $T^2$-spectra in a similar fashion.

\begin{dfn}(Cf.~\cite[B.7.1]{Fulton})
Suppose $X\to Y$ is a closed embedding. We call $X$ a {\it locally
complete intersection (l.c.i.) subscheme of $Y$\/} if for every
point of $X$ there is an affine neighborhood in $Y$ such that ideal
of definition of $X$ is generated by a regular sequence.
\end{dfn}

\begin{rem}\label{remlci}(\cite[Corollary 4.5]{AK})
If $Y$ is regular and $X$ is a
closed subscheme of codimension $d$, then $X$ is an l.c.i. subscheme if and only if
the ideal of definition of $X$ is locally generated by $d$ elements.
\end{rem}

\begin{lem}\label{Frescheme}
For $X,Y\in\Sm_k$, there is a natural bijection between the set
$\Fr_n^E(X,Y)$ and the set of equivalence classes of quadruples $(U,Z,\phi,f)$, where
\begin{itemize}
\item $Z$ is a closed l.c.i. subscheme of $\A^n_X$, finite and flat over $X$;
\item $U$ is an \'etale neighborhood of $Z$ in $\A^n_X$;
\item $\phi\colon U\to V_n$ is a regular map, called a framing, such that $Z=U\times_{V_n}Z_n$;
\item $f\colon U\to Y$ is a regular map.
\end{itemize}
Two quadruples $(U,Z,\phi,f)$ and $(U',Z',\phi',f')$ are equivalent
if $Z=Z'$ and there is an open neighborhood $U_0$ of $Z$ in
$U\times_{\A^n_X}U'$ such that the framings $\phi,\phi'$ as well as
regular maps $f,f'$ coincide on $U_0$.
\end{lem}

\begin{proof}
By Voevodsky's Lemma~\ref{Voevlemma}, the elements of $\Fr_n^E(X,Y)$ can
be described as the sets of equivalence classes of quadruples
$(U,Z,\phi,f)$, where $Z$ is a closed subset of $\A^n_X,$ finite
over $X$, $U$ is its \'etale neighborhood, and $\phi\colon U\to V_n$
is a regular map such that $Z=\phi^{-1}(Z_n)$, and $f\colon U\to Y$
is a regular map. Two quadruples $(U,Z,\phi,f)$ and $(U,Z',\phi',f')$ are
equivalent if $Z=Z'$ and there is an open neighbourhood of $Z$ in
$U\times_{\A^n_X}U'$, where $\phi$ coincides with $\phi'$, and $f$
coincides with $f'$.

For any such quadruple the framing $\phi\colon U\to E_n$ defines a
closed subscheme $Z'=U\times_{V_n}Z_n$. Then $(Z')_{red}=Z$, hence $Z$ has codimension $n$
in $U$, and is locally defined by $n$ equations. Then it is an l.c.i. subscheme of $U$
by Remark~\ref{remlci}. Since $(Z')_{red}=Z$ is finite over $X$, $Z$ is finite over $X$ as well, and
the composition $Z'\to U\to\A^n_X$ is a closed embedding by~\cite[Tag 04XV]{stacks-project}.
Since $U\to\A^n_X$ is \'etale, it induces an isomorphism between conormal sheaves of $Z'$ in
$U$ and $Z'$ in $\A^n_X$ by~\cite[Tag 0635]{stacks-project}. Then by Nakayama's lemma
$Z'$ locally is defined in $\A^n_X$ by $n$ equations. Thus $Z'$ is an l.c.i. subscheme in $\A^n_X$,
finite over $X$. It is flat over $X$ by~\cite[Tag 00R3]{stacks-project}. Then the assignment
$(U,Z,\phi,f)\mapsto (U,Z',\phi,f)$ defines the desired bijection between $\Fr_n^E(X,Y)$
and the set of the statement of the lemma.
\end{proof}

\begin{dfn}
For $X,Y\in\Sm_k$ the {\it set of normally framed correspondences\/}
$\widetilde{\Fr}_n^E(X,Y)$ is the set of equivalence classes of quintuples
$(U,Z,\phi,\psi,f)$, where
\begin{itemize}
\item $Z$ is an l.c.i. subscheme of $\A^n_X$, finite and flat over $X$;
\item $U$ is an \'etale neighborhood of $Z$ in $\A^n_X$;
\item $\psi\colon U\to Z_n$ is a regular map and $\phi\colon N_{Z/\A^n_X}\cong (\psi i)^*V_n$
is an isomorphism of vector bundles, where $i$ is the inclusion $i\colon Z\to U$;
\item $f\colon Z\to Y$ is a regular map.
\end{itemize}
Two quintuples $(U,Z,\phi,\psi,f)$ and $(U',Z',\phi',\psi',f')$ are
equivalent if $Z=Z'$ as subschemes of $\A^n_X$ and there is an
open neighborhood $U''$ of $Z$ in $U\times_{\A^n_X}U'$ such that
$\psi=\psi'$ on $U''$, $\phi=\phi'$, and $f=f'.$
\end{dfn}

\begin{dfn}
For an affine scheme $X=\Spec A$ and its closed subscheme $Z=\Spec
A/I$ of $X$ denote by $X^h:=\Spec A^h$, where $(A^h,I^h)$ is the
Henselian pair associated to $(A,I)$\cite[Tag 09XD]{stacks-project}.
We call $X^h$ the {\it Henselization of $X$ in $Z$}. If $(A,I)$ is a
Henselian pair, we will call $(X,Z)=(\Spec A, \Spec A/I)$ a {\it
Henselian pair of schemes}.
\end{dfn}

\begin{rem}\label{Nphi} Suppose $i\colon Z\to U$ is a closed l.c.i. subscheme
and $\phi\colon U\to V_n$ is a regular map such that $\phi^*J\subseteq I$,
where $I$ is the sheaf of ideals defining $Z$ in $U$ and $J$ is the sheaf of ideals
defining $Z_n$ in $V_n$. Then it defines a morphism of vector bundles
\begin{equation}\label{eee}
N(\phi)\colon N_{Z/U}\to (\pi\phi i)^*V_n,
\end{equation}
where $\pi$ is the projection $\pi\colon V_n\to Z_n,$ which is dual to the morphism of sheaves
\[J/J^2\otimes_{\Os_{Z_n}}\Os_Z\to I/I^2.\]
\end{rem}

\begin{lem}\label{Nphi12} In the notation of Remark~\ref{Nphi} one has:
\begin{itemize}
\item[$(1)$] if $\phi\colon U\to V_n$ is a framing of $Z$, then $N(\phi)$ is an isomorphism;
\item[$(2)$] if $U$ is affine, $(Z,U)$ is a Henselian pair, and the morphism
$N(\phi)$ is an isomorphism, then $\phi$ is a framing of $Z$ in $U$.
\end{itemize}
\end{lem}

\begin{proof}
$(1)$. Note that when $\phi$ is a framing, $I$ is generated by the image of $J$. Hence the
map $J/J^2\otimes_{\Os_{Z_n}}\Os_Z\to I/I^2$ induced by $\phi$ is is a surjection of locally free
sheaves of rank $n$, and so it is an isomorphism. Thus $N(\phi)$ is an isomorphism.

$(2)$. Let $U=\Spec R$ and let $I'\subseteq I$ denote the ideal generated by the image of $J$.
Since $N(\phi)$ is an isomorphism, the dual map $J/J^2\otimes_{\Os_{Z_n}}\Os_Z\to I/I^2$ is
surjective, hence $I=I'+I^2$. Since $I\subseteq Jac(R)$, then $I=I'$ by Nakayama's lemma. We see that $I$
is generated by the image of $J$, and hence $\phi$ is a framing.
\end{proof}

There is a forgetful map
\[fog\colon\Fr_n^E(X,Y)\to\widetilde{\Fr}_n^E(X,Y),\ \ (U,Z,\phi,f)\mapsto (U,Z,N(\phi),\pi\phi,f).\]
There is also a stabilization map
\[\widetilde{\Fr}_n^E(X,Y)\to\widetilde{\Fr}_{n+1}^E(X,Y), \ (U,Z,\psi,\phi,f)\mapsto (U\times\A^1,Z\times 0,\psi',\phi',f),\]
where $\psi'$ is the composition $\psi'\colon U\times\A^1\to
U\stackrel{\psi}\to Z_n\to Z_{n+1}$ and $\phi'$ is the composition
\[\phi'\colon N_{Z\times 0/\A^{n+1}_X}=N_{Z/\A^n_X}
\oplus 1\xrightarrow{\phi\oplus 1} (\psi i)^*(V_n\oplus 1)=(\psi'i)^*(V_{n+1}).\]
Denote by $\widetilde{\Fr}^E(-,Y)$ the colimit of the presheaves
$\widetilde{\Fr}_n^E(-,Y)$ with respect to the stabilization maps
\[\widetilde{\Fr}^E(-,Y):=\colim_n\widetilde{\Fr}_n^E(-,Y).\]

\begin{lem}\label{frtilde}
The presheaf $\widetilde{\Fr}^E(-,Y)$ admits framed transfers and the
forgetful map induces a map $fog\colon \Fr^E(-,Y)\to\widetilde{\Fr}^E(-,Y)$
of presheaves with framed transfers.
\end{lem}

\begin{proof}
We shall construct a pairing
\[\Fr_n(X,Y)\times\widetilde{\Fr}_m^E(Y,W)\to\widetilde{\Fr}_{n+m}^E(X,W), (b,a)\mapsto b^*(a)\]
as follows. If $a=(U,Z,\phi,\psi,f)\in\widetilde{\Fr}_m^E(Y,W)$ and
$b=(U',Z',\phi',f')\in\Fr_n(X,Y)$, define
$b^*(a)=(U'',Z'',\phi'',\psi'',f'')$, where $U''=U'\times_YU,
Z''=Z'\times_Y Z$, $\psi''$ is the composition
\[\psi''\colon U'\times_Y U\to U\to Z_m\xrightarrow{i_{m,n}^l} Z_{n+m}.\]
Since the canonical map $\tau:(N_{Z'/U'})|_{Z''}\oplus
(N_{Z/U})|_{Z''}\to N_{Z''/U''}$ is a surjection of vector bundles
of the same rank, it is an isomorphism. Define the isomorphism
$\phi''$ as the composition
\[\phi''\colon N_{Z''/U''}\xrightarrow{\tau^{-1}}(N_{Z'/U'})|_{Z''}\oplus (N_{Z/U})|_{Z''}\xrightarrow{N(\phi')\oplus\phi}
1^n\oplus (i\psi)^*V_m\to (i\psi'')^*V_{n+m},\]
where $N(\phi)$ stands for the isomorphism of formula~\eqref{eee} for
$E=S_T$. The function  $f''$ is, by definition, the composition $U''\to U\to W.$ This
pairing is plainly compatible with stabilization by $m$ and endows
$\widetilde{\Fr}^E(-,Y)$ with the structure of a framed presheaf such
that the forgetful map $\Fr^E(-,Y)\to\widetilde{\Fr}^E(-,Y)$ is a morphism of framed presheaves.
\end{proof}

\begin{rem}\label{henselremark}
If $X$ is an affine smooth variety, the set $\Fr_n^E(X,Y)$
(respectively $\widetilde{\Fr} _n^E(X,Y)$) is in bijective correspondence
with the set of triples $(Z,\phi,f)$, where $Z$ is an l.c.i. closed subscheme in
$\A^n_X$, finite and flat over $X$, $\phi\colon (\A^n_X)^h\to
V_n$ such that $Z=(\A^n_X)^h\times_{V_n}Z_n$, and $f\colon (\A^n_X)^h\to
Y$ (respectively with the set of quadruples $ (Z,\phi,\psi,f)$, where $Z$ is an l.c.i.
closed subscheme in $\A^n_X$, finite and flat over $X $,
$\psi\colon (\A^n_X)^h\to Z_n$, and $\phi\colon
N_{Z/\A^n_X}\stackrel{\cong}\to (\psi i)^*V_n$, $f\colon(\A^n_X)^h\to Y$).
\end{rem}

\begin{lem}\label{locsurj}
Suppose there is an \'etale map $Y\to\A^d$. Then the forgetful map
of presheaves $fog\colon \Fr^E_n(-,Y)\to\widetilde{\Fr}^E_n(-,Y)$ is
locally surjective in the Nisnevich topology.
\end{lem}

\begin{proof}
Suppose $X$ is a local Henselian scheme and
$(Z,\phi,\psi,f)\in\widetilde{\Fr}^E_n(X,Y)$. Then $Z$ is semi-local
Henselian and the map $\psi\colon U\to Z_n$, where $U$ is the
Henselization of $Z$ in $\A^n_X$, factors as $U\xrightarrow{\widetilde\psi}
Z_n^0\subseteq Z_n$, where $Z_n^0$ is an open subset of $Z_n$ such that
the fiber $V_n$ over $Z_n^0$ is a trivial vector bundle. Let $i$ denote the inclusion $Z\hookrightarrow U$. Fix a
trivialization $V_n|_{Z_n^0}\cong Z_n^0\times\A^n$. It gives a
trivialization of $(\psi i)^*V_n$. Composing the latter trivialization with $\phi$, one gets a
trivialization $\widetilde\phi$ of the bundle $N_{Z/\A^n_X}=N_{Z/U}$. The
trivialization $\widetilde\phi$ provides a basis of the $k[Z]$-module $I/I^2$,
where $I$ is the ideal of definition of $Z$ in $U$. The basis of
$I/I^2$ lifts to a set of generators
$\gamma=(\gamma_1,\ldots,\gamma_n)$ of $I$. They define a map
$\phi'=(\widetilde\psi,\gamma)\colon U\to Z_n^0\times\A^n\to V_n$ such that
$\phi=N(\phi')$ in the sense of formula~\eqref{eee}.

Now let us extend the regular map $f\colon Z\to Y$ to
$f'\colon U\to Y$. By assumption, there is an \'etale map $g\colon Y\to\A^d$. There
is also a map $h\colon U\to \A^d$ that extends the composition $gf\colon
Z\to\A^d$. Then $W=U\times_{\A^d}Y$ will give an \'etale
neighbourhood of $Z$ in $U$, hence there is a section $U\to W$. Then
the composition $f'\colon U\to W\to Y$ extends $f$.
We see that the triple $(Z,\phi',f')\in\Fr_n^E(X,Y)$ is a preimage of
$(Z,\phi,\psi,f)\in\widetilde{\Fr}^E_n(X,Y)$.
\end{proof}

\begin{dfn}
For $Y\in\Sm_k$ define a presheaf of
$S^1$-spectra $\widetilde{\Fr}^E(Y\otimes \Ss)$ associated to the
presheaf of $\Gamma$-spaces $K\mapsto \widetilde{\Fr}^E(Y\otimes K)$ (cf.~\cite[Section~5]{GPMain})
\[\widetilde{\Fr}^E(Y\otimes \Ss)=(\widetilde{\Fr}^E(Y),\widetilde{\Fr}^E(Y\otimes S^1),\widetilde{\Fr}^E(Y\otimes S^2),\ldots).\]
The {\it normally $E$-framed motive of $Y$\/} is the presheaf of $S^1$-spectra
\[\widetilde{M}_E(Y)=C_*\widetilde{\Fr}^E(Y\otimes\Ss)=
(C_*\widetilde{\Fr}^E(Y),C_*\widetilde{\Fr}^E(Y\otimes S^1),C_*\widetilde{\Fr}^E(Y\otimes S^2),\ldots).\]
It follows from Lemma~\ref{frtilde} that both $\widetilde{\Fr}^E(Y\otimes \Ss)$ and $\widetilde{M}_E(Y)$
are presheaves of $S^1$-spectra with framed transfers.
\end{dfn}

\begin{lem}\label{tildazf}
The presheaves of stable homotopy groups $\pi_i(\widetilde{\Fr}^E(Y\otimes
\Ss))$ have $\ZF_*$-transfers and the presheaves of stable homotopy groups
$\pi_i(\widetilde{M}_E(Y))$ are $\A^1$-invariant stable $\ZF_*$-pre\-sheaves.
\end{lem}

\begin{proof}
For $X_1,X_2$ there is a natural bijection
$\widetilde{\Fr}^E_n(X_1\sqcup
X_2,Y)\to\widetilde{\Fr}^E_n(X_1,Y)\times\widetilde{\Fr}^E_n(X_2,Y)$,
hence there is an isomorphism $\widetilde{\Fr}^E(X_1\sqcup
X_2,Y\otimes\Ss)\to\widetilde{\Fr}^E_n(X_1,Y\otimes\Ss)\times
\widetilde{\Fr}^E_n(X_2,Y\otimes\Ss)$ of
$S^1$-spectra. Then the presheaves
$\pi_i(\widetilde{\Fr}^E(Y\otimes \Ss))$ are radditive with framed
transfers, and hence these are $\ZF_*$-presheaves.

Recall that $\sigma_X\in \Fr_1(X,X)$ uniquely corresponds to the canonical motivic equivalence
$X_+\w\PP^{\w1}\to X_+\w T$ and is given by the quadruple $(X\times 0,X\times\A^1,pr_{\A^1},pr_X)$.
Then $\sigma_X^*\colon\widetilde{\Fr}_n^E(X,Y)\to
\widetilde{\Fr}_{n+1}^E(X,Y)$ differs from the stabilization map
$\widetilde{\Fr}_n^E(X,Y)\to \widetilde{\Fr}_{n+1}^E(X,Y)$ by the
action of the shuffle permutation $\chi_{1,n}$ on $\A^{n+1}_X$ and on
the vector bundle $V_{n+1}$. As usual, when $n$ is even, they differ by
an $\A^1$-homotopy, hence induce homotopic maps
$C_*\widetilde{\Fr}_n^E(X,Y)\to C_*\widetilde{\Fr}_{n+1}^E(X,Y)$.
Thus $\sigma_X^*$ induces the identity map on presheaves of homotopy groups.
\end{proof}

\begin{dfn}\label{Cechdef}
For a map of simplicial presheaves $f\colon X\to Y$ denote by
$\check{C}(f)$ the diagonal of the \v{C}ech bisimplicial presheaf
with $n$-simplices given by simplicial presheaf
\[\check{C}(f)_n=X\times_Y\times\ldots\times_Y X\quad\text{ ($n+1$ times)}\]
with the usual face and degeneracy maps.
Then $f$ factors as a composition
\[X\xrightarrow{d(f)}\check{C}(f)\xrightarrow{p(f)} Y\]
where $d(f)$ is the diagonal map
\[d(f)_n\colon X_n\to X_n\times_{Y_n}\times\ldots\times_{Y_n} X_n\]
and $p(f)$ is the projection
\[p(f)_n\colon X_n\times_{Y_n}\times\ldots\times_{Y_n} X_n\to Y_n.\]
Note that if $X(U)\to Y(U)$ is surjective for $U\in\Sm_k$, then
$p(f)\colon\check{C}(f)(U)\to Y(U)$ is a weak equivalence of simplicial sets.
\end{dfn}

\begin{dfn}
For every simplicial presheaf $X$ denote by $C_*X$ the diagonal of
the bisimplicial presheaf $n\mapsto X(\Delta^n_k)$. Then there is a
canonical inclusion map $c_0\colon X\to C_*(X)$, and for every $U\in\Sm_k$ the map
\[C_*(c_0)\colon C_*X(U)\to C_*C_*X(U)\]
is a weak equivalence of simplicial sets.
\end{dfn}

\begin{lem}\label{homotopy}
Suppose there is an \'etale map $g\colon Y\to\A^d$ and
$fog_Y\colon\Fr_n^E(Y)\to\widetilde{\Fr}_n^E(Y)$ is the
forgetful map. Then
there exists a map of simplicial presheaves
$H_Y\colon\check{C}(fog_Y)\to C_*\Fr_n^E(Y)$ on the category of smooth
affine varieties, compatible with stabilization by $n$, and that
fits into the commutative diagram
\begin{equation}\label{ooo}
\xymatrix{
\Fr^E_n(Y)\ar[d]_{d(fog)}\ar[rr]^{c_0} && C_*\Fr^E_n(Y)\ar[d]^{C^*(fog_Y)}\\
\check{C}(fog_Y)\ar[rru]^{H_Y}\ar[rr]^{c_0\circ p(fog)} && C_*(\widetilde{\Fr}_n^E(Y)). }
\end{equation}
Moreover, the map $H_Y$ is functorial in $Y$ in the following sense: if $g\colon Y\to\A^d$ is \'etale,
$g'\colon Y'\to\A^d$ is \'etale, and $q\colon Y\to Y'$ is a map such that $g'q=g$, then the diagram
\begin{equation}\label{ooo2}
\xymatrix{
\check{C}(fog_Y)\ar[d]\ar[r]^{H_Y} & C_*\Fr^E_n(Y)\ar[d]\\
\check{C}(fog_{Y'})\ar[r]^{H_{Y'}} & C_*\Fr^E_n(Y')}
\end{equation}
is commutative. Here the vertical arrows are induced by $q$.
\end{lem}

\begin{proof}
For brevity we sometimes write $fog$ instead of $fog_Y$.
For an affine smooth $X$ the set of $m$ simplices
$\check{C}(fog)_m(X)$ consists of $(m+1)$-triples of correspondences
$(Z,\phi_0,f_0),\ldots,(Z,\phi_m,f_m)$ in $\Fr^E_n(X,Y)$ such that the maps
$\pi\phi_0,\ldots, \pi\phi_m\colon U\to Z_n$ are equal, isomorphisms
on normal bundles $N(\phi_i)\colon N_{Z/A^n_X}\to (\pi\phi_i
i)^*V_n$ are equal for $i=0,\ldots, m$, and the regular maps $f_i\colon
U\to Y$ coincide on $Z$. Here $U$ denotes the Henselization of $Z$ in $\A^n_X$ and
we use Remark~\ref{henselremark} here.

The addition map $V_n\times_{Z_n}V_n\to V_n$ and scalar
multiplication map $\A^1\times V_n\to V_n$ give rise to the linear
combination map
\[V_n\times_{Z_n}V_n\times\ldots\times_{Z_n} V_n\times\A^{m+1}\to V_n,\]
\begin{equation*}\label{ggg}
((v_0,\ldots, v_m),(t_0,\ldots, t_m))\mapsto t_0v_0+\ldots +t_mv_m.
\end{equation*}

For a $(m+1)$-tuple $(Z,\phi_0,f_0),\ldots (Z,\phi_m,f_m)$ in
$\check{C}(fog)_m(X)$ the maps $\phi_0,\ldots,\phi_m$ coincide after
composing them with $\pi\colon V_n\to Z_n$, hence they define a map
$\phi\colon U\to V_n\times_{Z_n}\times\ldots\times_{Z_n}V_n$. Taking
composition with the linear combination map, we get a map
\[\Phi=t_0\phi_0+\ldots+t_m\phi_m\colon U\times\Delta^m_k\to V_n,\]
where $t_0,\ldots, t_m$ denote the barycentric coordinates on
$\Delta^m_k$. Let $J$ denote the sheaf of ideals defining $Z_n$ in $V_n$. For every
$\phi_i$ we have that $\phi^*_i(J)$ lies inside the ideal $I$ defining $Z$ in $U$.
Then $\Phi^*(J)$ lies inside the ideal $I\otimes_k k[\Delta^m]\subseteq k[U]\otimes_{k}k[\Delta^m]$
which defines $Z\times\Delta^m_k$ inside $U\times\Delta^m_k$.

Let $\Phi^h:(U\times\Delta^m_k)^h\to V_n$ denote the map on Henselization induced by $\Phi$.
Since $\Phi^*(J)\subseteq I\otimes k[\Delta^m]$, then $(\Phi^h)^*(J)\subseteq (I\otimes k[\Delta^m])^h$,
where the latter denotes the corresponding ideal in the Henselization ring $k(U\times\Delta^m)^h$.
The normal bundles $N_{Z\times\Delta^m_k/U\times\Delta^m_k}$ and $N_{Z\times\Delta^m_k/(U\times\Delta^m_k)^h}$
are canonically isomorphic. We denote them by $N_{Z\times\Delta^m_k}$ for brevity.
By~\ref{Nphi} $\Phi^h$ defines a morphism of vector bundles
\[N(\Phi^h)\colon N_{Z\times\Delta^m_k}\to (\pi\Phi^h i_{\Delta})^*V_n,\]
where $i_{\Delta}\colon Z\times\Delta^m_k\to (U\times\Delta^m_k)^h$ denotes the inclusion.
Let $i\colon Z\to U$ denote the inclusion, and $p\colon Z\times\Delta^m_k\to Z$ denote the projection.
Then $\pi\Phi^hi_{\Delta}=\pi\phi_j i p$ for every $j=0,\ldots,m$. In particular, $\pi\Phi^hi_{\Delta}=\pi\phi_0 i p$.
The normal bundle $N_{Z\times\Delta^m_k}$ is canonically
isomorphic to the pullback $p^*N_{Z/U}$. By construction, the morphism of bundles $N(\Phi^h)$ equals the sum
\[N(\Phi^h)=t_0p^*N(\phi_0)+\ldots +t_mp^*N(\phi_m)\colon p^*N_{Z/U}\to p^*(\pi\phi_0 i)^*V_n.\]
Since $N(\phi_i)=N(\phi_0)$ for all $i=0,\ldots, m$, then $N(\Phi^h)=p^*N(\phi_0)$
is an isomorphism, because so is $N(\phi_0)$.
Then $\Phi^h$ is a framing of $Z\times\Delta^m_k$ in
$(U\times\Delta^m_k)^h$ by Lemma~\ref{Nphi12}(2).

The maps $f_0,\ldots,f_m\colon U\to Y$ coincide on $Z$. Consider the
map \[ t_0gf_0+\ldots +t_mgf_m\colon
U\times\Delta^m_k\to\A^d.\] Then the fiber product
$U'=(U\times\Delta^m_k)\times_{\A^d}Y$ is an \'etale neighborhood of
$Z\times\Delta^m_k$ in $U\times\Delta^m_k$, hence there is a unique
section $s\colon (U\times\Delta^m_k)^h\to U'$, where
$(U\times\Delta^m_k)^h$ is the Henselization of $Z\times\Delta^m_k$
in $U\times\Delta^m_k.$ Denote by $t_0f_0+\ldots+t_mf_m$ the
composition
\[(U\times\Delta^m_k)^h\stackrel{s}\to U'\to Y.\]

Then for every $m\geqslant 0$ one gets a map $H_m\colon\check{C}(fog)_m\to C_m\Fr_n^E(Y)$ defined as
\[H_m\colon (Z,\phi_0,f_0),\ldots (Z,\phi_m,f_m)\mapsto
(Z\times\Delta^m_k,\Phi^h,t_0f_0+\ldots+t_mf_m)\]
in the notation of Remark~\ref{henselremark}.
Clearly, the maps $H_m$ are compatible with the face and degeneracy maps and yield
the desired morphism of simplicial presheaves $H_Y\colon\check{C}(fog_Y)\to
C_*\Fr_n^E(Y)$ on the category of smooth affine varieties.

If a $m$-tuple $((Z,\phi_0,f_0),\ldots (Z,\phi_m,f_m))$ in $\check{C}(fog)_m$
is in the image of $d(fog)$, then $(Z,\phi_i,f_i)=(Z,\phi_0,f_0)$ for $i=0,\ldots m$.
Thus $(Z\times\Delta^m_k,\Phi^h,t_0f_0+\ldots+t_mf_m)=(Z\times\Delta^m_k,\phi_0\circ pr,f_0\circ pr)$,
where $pr\colon U\times\Delta^m_k\to U$ is the projection. Then the left triangle in the
diagram~\eqref{ooo} is commutative.

As we have already proved, if $\Phi=t_0\phi_0+\ldots+t_m\phi_m$ then $N(\Phi^h)=p^*N(\phi_0)$.
It follows that the right triangle in the diagram~\eqref{ooo} is commutative as well.

To see that the diagram~\eqref{ooo2} commutes when $q\colon Y\to Y'$ is a map over $\A^d$,
we note that the following maps coincide:
\[t_0gf_0+\ldots+t_mgf_m=t_0g'qf_0+\ldots+t_mg'qf_m\colon U\times\Delta^m_k\to\A^d.\]
Then the diagram
\[\xymatrix{
(U\times\Delta^m_k)^h\ar[r]\ar[rd] & (U\times\Delta^m_k)\times_{\A^d}Y\ar[r]\ar[d] & Y\ar[d]^{q}\\
                  & (U\times\Delta^m_k)\times_{\A^d}Y'\ar[r] & Y'}\]
is commutative. Thus we get that
\[q(t_0f_0+\ldots+t_mf_m)=t_0qf_0+\ldots +t_mqf_m\colon(U\times\Delta^m_k)^h\to Y',\]
and hence~\eqref{ooo2} commutes.
\end{proof}



\begin{lem}\label{normfrcoret}
Suppose there is an \'etale map $g\colon Y\to\A^d.$ Then the natural map
$fog\colon\Fr^E(Y)\to\widetilde{\Fr}^E(Y)$ induces a local stable
weak equivalence of $S^1$-spectra
$M_E(Y)\stackrel{\simeq}\to \widetilde{M}_E(Y)$.
\end{lem}

\begin{proof}
For every $n\geqslant 0$ the map $fog\colon\Fr^E_n(Y)\to\widetilde{\Fr}^E_n(Y)$
is locally surjective by Lemma~\ref{locsurj}. It follows that the induced
map $\check{C}(fog)\to\widetilde{\Fr}^E_n(Y)$ is a local weak
equivalence. Let $\check{C}(fog\otimes\Ss)$ denote the presheaf of Segal $S^1$-spectra
associated to the presheaf of $\Gamma$-spaces $K\mapsto\check{C}(fog\otimes K)$,
where $fog\otimes K$ is the forgetful map $fog\colon\Fr^E_n(Y\otimes K)\to\widetilde{\Fr}^E_n(Y\otimes K).$
Then the induced map
   $$\check{C}(fog\otimes\Ss)\to{\widetilde{\Fr}}^E_n(Y\otimes\Ss)$$
is a levelwise local weak equivalence of $S^1$-spectra.


For any finite pointed set $K$  we have that $Y\otimes K$ is \'etale
over $\A^d$ via the natural composition $Y\otimes K\to
Y\xrightarrow{g}\A^d$ and for any map $K\to K'$ of pointed sets the
induced map $Y\otimes K\to Y\otimes K'$ is a map of varieties over
$\A^d$. Then the maps $H_{Y\otimes K}$ of Lemma~\ref{homotopy}
induces a map of presheaves of $S^1$-spectra $H\colon
\check{C}(fog\otimes\Ss)\to C_*\Fr^E_n(Y\otimes\Ss)$. Applying $C_*$
we get a commutative diagram
\[
\xymatrix{
C_*\Fr^E_n(Y\otimes\Ss)\ar[d]_{C_*d(fog)}\ar[r] & C_*C_*\Fr^E_n(Y\otimes\Ss)\ar[d]^{C_*C_*(fog)}\\
C_*\check{C}(fog\otimes\Ss)\ar[r]\ar[ru]^{C_*H} &
C_*C_*\widetilde{\Fr}^E_n(Y\otimes\Ss) }
\]
The horizontal arrows in the diagram are motivic stable weak equivalences.
Then $C_*H$ has both a left and a right inverse in $SH_{S^1}(k)$.
So $C_*H$ is a motivic stable weak equivalence as well, and hence so are the vertical arrows.
Since $C_*$ is an idempotent operation up to motivic equivalence and sequential colimits preserve
stable motivic equivalences, it follows that
$C_*(fog):M_E(Y)\to \widetilde{M}_E(Y)$ is a motivic stable weak equivalence. It follows
from Lemmas~\ref{cancel}(3),~\ref{tildazf}, \cite[7.1]{GPMain} and~\cite[1.1]{GPPresheaves} that local stable fibrant replacements
$M_E(Y)_f,{M}_E(Y)_f$ of $M_E(Y),{M}_E(Y)$ are motivically fibrant $S^1$-spectra.
Therefore the induced map $C_*(fog)_f:M_E(Y)_f\to \widetilde{M}_E(Y)_f$ is a sectionwise
level weak equivalence of spectra, and hence $C_*(fog):M_E(Y)\to \widetilde{M}_E(Y)$ is a
stable local weak equivalence, as required.
\end{proof}

\begin{lem}\label{normMV}
Suppose $Y\in\Sm_k$ equals the union of two open subschemes
$Y_1$ and $Y_2$. Let $Y_{12}=Y_1\cap Y_2$. Then
   \[\xymatrix{
     \widetilde{M}_E(Y_{12})\ar[r]\ar[d] & \widetilde{M}_E(Y_{1})\ar[d] &&
      M_E(Y_{12})\ar[r]\ar[d] & M_E(Y_{1})\ar[d]    \\
      \widetilde{M}_E(Y_{2})\ar[r] & \widetilde{M}_E(Y)&&
      M_E(Y_{2})\ar[r] & M_E(Y)}\]
are homotopy pushout squares in the local stable model structure of $S^1$-spectra.
\end{lem}

\begin{proof}
Similarly to~\cite[Definition 8.3]{GPMain} one can introduce the
presheaves of abelian groups $\widetilde{\ZF}^E(Y)$ imposing the
additivity relation on supports in $\Z\widetilde{\Fr}^E(Y)$. The
same reasons as in~\cite[Theorem 1.2]{GNP} show that homology of the
complex $C_*\widetilde{\ZF}(Y)(X)$ computes homology of the
$S^1$-spectrum $\widetilde{M}_E(Y)(X)$ for any $X\in\Sm_k.$
Repeating Lemma~\ref{ZFtriangleaux}, Corollary~\ref{ZFtriangle} and
Proposition~\ref{MV} literally, one gets that the sequence
\[0\to \widetilde{\ZF}(Y_{12})\to \widetilde{\ZF}(Y_1)\oplus \widetilde{\ZF}(Y_2)\to \widetilde{\ZF}(Y)\to 0\]
is locally exact, hence there is a triangle in the derived category
\[C_*\widetilde{\ZF}(Y_{12})\to C_*\widetilde{\ZF}(Y_1)\oplus C_*\widetilde{\ZF}(Y_2)\to C_*\widetilde{\ZF}(Y)\]
of complexes of sheaves. So the first square in the statement of the lemma
is homotopy pushout. The same proof applies to showing that
the second square is homotopy pushout.
\end{proof}

\begin{dfn}\label{frecolimit}
Suppose $E$ is a directed colimit of spectra $E^{(i)}$ satisfying
Condition~\ref{convention}. Define the presheaf $\widetilde{\Fr}^E(Y)$ as the directed colimit
\[\widetilde{\Fr}^E(Y)=\colim_i\widetilde{\Fr}^{E^{(i)}}(Y),\]
and the {\it normally $E$-framed motive $\widetilde{M}_E(Y)$\/} as the
directed colimit $\widetilde{M}_E(Y)=\colim_i
\widetilde{M}_{E^{(i)}}(Y)$.
\end{dfn}

\begin{prop}\label{normfrcor}
Suppose $Y\in\Sm_k$ and $E$ is a directed colimit of spectra $E^{(i)}$ satisfying Condition~\ref{convention}.
Then the natural forgetful map $fog\colon\Fr^E(Y)\to\widetilde{\Fr}^E(Y)$ induces a local stable
weak equivalence of $S^1$-spectra
$M_E(Y)\stackrel{\simeq}\to \widetilde{M}_E(Y)$.
\end{prop}

\begin{proof}
Every smooth variety $Y$ of dimension $d$ has a Zariski cover by
varieties $Y_i$ that admit \'etale maps $Y_i\to\A^d$. Then the
statement follows by induction on the number of varieties in the
cover of $Y$ if we apply Lemmas~\ref{normfrcoret} and~\ref{normMV}
as well as the fact that $M_E(Y)=\colim_i M_{E^{(i)}}(Y).$
\end{proof}

Similarly to framed correspondences there is a natural action of
the category $\Fr_0(k)$ on Nisnevich sheaves $\widetilde{\Fr}^E(-,X)$
that takes $U\in\Fr_0(k)$ to the sheaf $\widetilde{\Fr}^E(-\times U,X\times U)$.
The action gives rise to maps of $S^1$-spectra
   $$a_n:\widetilde{M}_E(X_+\w\Gm^{\w n})\to\uhom(\Gm^{\w 1},\widetilde{M}_E(X_+\w\Gm^{\w n+1})),\quad n\geqslant 0,$$
literally repeating the construction of the same maps for weighted $K$-motives in~\cite[Section~3]{GPmss}.

We finish the section by the following computation.

\begin{thm}\label{bisptilde}
Suppose $X\in \Sm_k$ and $E$ is a symmetric Thom $T$-spectrum with
the bounding constant $d=1$ and contractible alternating group action.
Then the $(S^1,\Gm^{\w 1})$-bispectrum
\[\widetilde{M}_E^{\mathbb G}(X)_f:=(\widetilde{M}_E(X)_f,\widetilde{M}_E(X_+\w\Gm^{\w 1})_f,\widetilde{M}_E(X_+\w\Gm^{\w 2})_f,\ldots)\]
with bonding maps induced by $a_n$-s above
is motivically fibrant and represents the $T$-spectrum $X_+\w E$ in
the category of bispectra, where ``$f$'' refers to stable local
fibrant replacements of $S^1$-spectra.
\end{thm}

\begin{proof}
By Lemma~\ref{tildazf} the sheaves of stable homotopy groups of each $S^1$-spectrum in $\widetilde{M}_E^{\mathbb G}(X)_f$
are $\A^1$-invariant, stable with framed transfers. It follows from~\cite{GPPresheaves} that
they are strictly $\A^1$-invariant. By~\cite[7.1]{GPMain} all $S^1$-spectra of the bispectrum are motivically
fibrant. Observe that the natural map of bispectra
   $${M}_E^{\mathbb G}(X)_f\to\widetilde{M}_E^{\mathbb G}(X)_f,$$
induced by the forgetful map, is a level equivalence by Proposition~\ref{normfrcor}.
It follows from Theorem~\ref{mainbisp} that $\widetilde{M}_E^{\mathbb G}(X)_f$ is motivically
fibrant and represents the $T$-spectrum $X_+\w E$ in
the category of bispectra.
\end{proof}

\section{Computing the algebraic cobordism spectrum $MGL$}

In this section we give another description of the bispectrum
$M_E^{\G}(X)$ for the case $E=MGL$ in terms of Hilbert schemes and
$\Omega$-correspondences.

\begin{dfn}
Given a ring $R$, we call a submodule $M$ of $R^N$ {\it admissible\/} if the quotient
$R^N/M$ is projective. If $M$ is admissible, then it is also projective. We say that a
map $f\colon M\to R^N$ is an {\it admissible embedding} if $f$ is injective
and $f(M)$ is an admissible submodule of $R^N.$
\end{dfn}

\begin{dfn}\label{partialdef}
Given a ring $R$, denote by $R[\Delta^n]=R[t_0,\ldots t_n]/(t_0+\ldots+t_n-1)$
the coordinate ring on $\Delta^n_R$. Also, $R[\partial\Delta^n]:=R[\Delta^n]/(t_0t_1\ldots t_n)$
and for every $0\leqslant i\leqslant n,$ $R[\partial_i\Delta^n]:=R[\Delta^n]/t_i$ denotes
the ring of functions on the $i$-th face. We also set $R[\partial_{ij}\Delta^n]:=R[\Delta^n]/(t_i,t_j)$.

For every $R[\Delta^n]$-module $M$ denote by $\partial M=M
\otimes_{R[\Delta^n]}R[\partial\Delta^n]$, $\partial_iM=M
\otimes_{R[\Delta^n]}R[\partial_i\Delta^n],$ $\partial_{ij}M=M\otimes
_{R[\Delta^n]}R[\partial_{ij}\Delta^n].$
\end{dfn}

\begin{lem}\label{tilde=A}
For any affine $X$ there is a bijection between
$\widetilde{\Fr}_n^{MGL}(X,Y)$ and the set of quadruples
$(Z,\phi,\psi,f)$, where $Z$ is a closed l.c.i. subscheme of
$\A^n_X$, finite and flat over $X$, $R$ is the Henselization ring of
$Z$ in $\A^n_X$, $i\colon Z\to\Spec R$ is the embedding,
$\psi\colon\Spec R\to Gr(n)$, $\phi\colon N_{Z/\A^n_X}\to (\psi
i)^*\tau_n$ is an isomorphism of $k[Z]$-modules, and $f\colon Z\to
Y$ is a regular map.
\end{lem}

\begin{proof}
This follows from Remark~\ref{henselremark}, Definition~\ref{frecolimit} and the fact
that $Gr(n)(R)$ equals $\colim Gr(n,N)(R)$ for any $k$-algebra $R$.
\end{proof}

\begin{dfn} For $X,Y\in\Sm_k$ denote by $\Emb_n(X,Y)$ the set of couples $(Z,f)$,
where $Z$ is a closed l.c.i. subscheme in $\A^n_X,$ finite and flat over $X$,
and $f$ is a regular map $f\colon Z\to Y$. Note that $\Emb_n(X,Y)$ is pointed
at the couple $(\emptyset,\emptyset\to Y)$.
\end{dfn}

We need the following intermediate object:

\begin{dfn}\label{defBn}
For $X,Y\in\Sm_k$ denote by $B_n(X,Y)$ the set of quadruples
$(Z,\phi,\psi,f)$, where $Z$ is a closed l.c.i. subscheme of $\A^n_X$, finite and flat over
$X$, $\psi\colon Z\to Gr(n)$, $\phi\colon N_{Z/\A^n_X}\to \psi^*\tau_n$ is an isomorphism
of vector bundles over $Z$, and $f\colon Z\to Y$ is a regular map.
\end{dfn}

\begin{rem}\label{grind}
The motivic space $Gr(n)=\colim_NGr(n,N)$ is a directed colimit of closed
embeddings of smooth varieties. For a $k$-scheme $Z$ by a regular map $\psi\colon Z\to Gr(n)$
we mean an element of $\colim_N \Hom(Z, Gr(n,N)).$ Then every
regular map $\psi\colon Z\to Gr(n)$ induces a vector bundle $\psi^*\tau(n)$ over $Z$.
Note that for a $k$-algebra $R$ the set $Gr(n,N)(R)$ is in bijective correspondence with
the set of rank $n$ admissible submodules of $R^N$ (see~\cite[Tag 089R]{stacks-project}).
\end{rem}

Note that $B_n(-,Y), \Emb_n(-,Y)$ are presheaves on $\Sm_k$. There are natural
forgetful maps $\widetilde{\Fr}_n^{MGL}(-,Y)\to B_n(-,Y)\to\Emb_n(-,Y)$. We shall
prove that for any smooth affine $X$ these maps induce weak equivalences of simplicial sets
\[C_*\widetilde{\Fr}_n^{MGL}(X,Y)\to C_*B_n(X,Y)\to C_*\Emb_n(X,Y).\]

\begin{lem}\label{A=B}
For every affine smooth $X$ the map $C_*\widetilde{\Fr}_n^{MGL}(X,Y)\to C_*B_n(X,Y)$
is a trivial Kan fibration of simplicial sets.
\end{lem}

\begin{proof}
The map on zero simplices $\widetilde{\Fr}_n^{MGL}(X,Y)\to B_n(X,Y)$ is
surjective by Lemma~\ref{projpatching}.
Suppose $\sigma\colon\Delta[m]\to C_*(B_n(X,Y))$ is a $m$-simplex and there is a
lift of the boundary $\gamma\colon\partial\Delta[m]\to C_*\widetilde{\Fr}_n^{MGL}(X,Y)$.
Then $\gamma$ is represented by a collection
$\gamma_i\colon\partial_i\Delta[m]\to C_* \widetilde{\Fr}_n^{MGL}(X,Y),$ such that
$\gamma_i$ and $\gamma_j$ agree on the intersection $\partial_i\Delta[m]\cap\partial_j\Delta[m].$

Suppose $\sigma$ is represented by a quadruple $(Z,\phi,\psi,f)\in B_n(\Delta^m_X,Y)$.
Let $\partial\Delta^m_X$ be the variety $\Spec k[\Delta^m_X]/(t_0t_1\ldots t_m)$,
where $t_0,\ldots, t_m$ are the barycentric coordinates of the algebraic  simplex
$\Delta^m_X.$  Let $\partial Z=Z\times_{\Delta^m_X}\partial\Delta^m_X$ denote
the fiber of $Z$ over $\partial\Delta^m_X$.

Let $R$ denote the Henselization ring of $Z$ inside $\A^n\times{\Delta^m_X}$. Note that
by~\cite[Tag09XK]{stacks-project} the ring $R'=R\otimes_{k[\Delta^m_X]}k[\partial
\Delta^m_X]$, is the Henselization ring of $\partial Z$, and
$R'_i=R\otimes_{k[\Delta^m_X]}k[\partial_i\Delta^m_X]$ is the Henselization of $\partial_iZ$, and
$R'_i\otimes_R R'_j$ is the Henselization of $\partial_iZ\cap\partial_jZ.$

Then each $\gamma_i$ is represented by a quadruple $(\partial_iZ,\phi|_{\partial_iZ},\psi'_i,f|_{\partial_iZ})$
as in Lemma~\ref{tilde=A}, where $\psi'_i\colon\Spec R'_i\to Gr(n)$ extends the map
$\psi|_{\partial_iZ}\colon\partial_iZ\to Gr(n)$. For any $i,j$ the maps $\psi'_i$ and $\psi'_j$
agree on $\Spec R'_i\otimes_R R'_j$. Then they descend to a map $\psi'\colon\Spec R'\to Gr(n)$.

Then by Lemma~\ref{projpatching} there exists a map $\psi''\colon\Spec R\to Gr(n)$
that extends $\psi'$ and $\psi.$ Clearly, the quadruple $(Z,\psi'',\phi,f)$ in
$\widetilde{\Fr}_n^{MGL}(\Delta^m_X,Y)$ is the desired lift of $\sigma$ that extends $\gamma.$
\end{proof}

\begin{lem}\label{B=Emb}
For any $n$ and any affine smooth $X$ the forgetful map $f\colon
B_n(-,Y)\to\Emb_n(-,Y)$ induces a trivial Kan fibration of simplicial
sets $C_*f\colon C_*B_n(X,Y)\to C_*\Emb_n(X,Y).$
\end{lem}

\begin{proof}
Suppose $m\geqslant 0$, $\sigma\colon\Delta[m]\to C_*\Emb_n(X,Y)$ is a $m$-simplex
and $\gamma\colon
\partial\Delta[m]\to C_*B_n(X,Y)$ is a lift of its boundary. Let us prove that there 
is a $m$-simplex $\sigma'\colon\Delta[m]\to C_*B_n(X, Y)$ making the diagram
\[
\xymatrix{
\partial\Delta[m]\ar[r]^-{\gamma}\ar[d] & C_*B_n(X,Y)\ar[d]^{C_*f}\\
\Delta[m]\ar[r]_(.35){\sigma}\ar@{-->}[ur]^{\sigma'} & C_*\Emb_n(X,Y)
}
\]
commutative.

Suppose $\sigma$ is given by
a couple $(Z,f)\in \Emb_n(\Delta^m_X,Y)$. The map $\gamma$ is given
by a collection of quadruples
$\gamma_i=(\partial_iZ,\phi_i,\psi_i,f|_{\partial_iZ})\in
B_n(\partial_i \Delta^m_X,Y)$ as in Definition~\ref{defBn}, where
$\partial_i\Delta^m_X$ denotes the $i$-th face of the algebraic
simplex $\Delta^m_X$ and $\partial_iZ$ is the fiber of $Z$ over $
\partial_i\Delta^m_X$.
The elements $\gamma_i$ coincide on the intersections
$\partial_i\Delta^m_X\cap\partial_j\Delta^m_X$, and hence the
regular maps $\psi_i\colon\partial_iZ\to Gr(n)$ coincide on the
intersections $\partial_{ij}Z$. So they descend to a regular map
$\psi\colon \partial Z\to Gr(n,N)$ for some number $N$ by
Remark~\ref{grind}. The map $\psi$ defines an admissible submodule
$j\colon P=\psi^*\tau(n,N)\subseteq k[\partial Z]^N$. The
isomorphisms $\phi_i\colon N_{\partial_iZ}\to\partial_iP$ coincide
on intersections $N_{\partial_{ij}Z}$, and then by
Lemma~\ref{partialglue} there is a unique isomorphism $\phi\colon
N_{\partial Z}\to P$ that extends $\phi_i.$

Then $j\circ\phi\colon N_{\partial Z}\to k[\partial Z]^N$ is an admissible embedding.
By Lemma~\ref{borderextend} it can be extended to an admissible embedding
$\Phi\colon N_Z\to k[Z]^N\oplus k[Z]^d$ such that $\partial\Phi$ equals the composition
of $j\circ\phi$ and the standard embedding $k[\partial Z]^N\to k[\partial Z]^N\oplus k[\partial Z]^d$.
It follows that the image $\Phi(N_Z)\subseteq k[Z]^{N+d}$ is a rank $n$ admissible submodule, and so it
corresponds  to a regular map $\Psi\colon Z\to Gr(n,N+d)\to Gr(n)$ by Remark~\ref{grind} such that the
composition $\partial_iZ\to Z\xrightarrow{\Psi} Gr(n)$ equals $\psi_i$ for any $i=0,\ldots, n.$
Abusing notation, denote by $\Phi$ the isomorphism $\Phi\colon N_{Z/\A^n_X}\to\Phi(N_{Z/\A^n_X}).$
Then the quadruple $(Z,\Phi,\Psi,f)\in B_n(\Delta^m_X,Y)$ provides the desired $m$-simplex $\sigma'.$
\end{proof}


\begin{prop}\label{embn}
For $Y\in\Sm_k$ the sheaf $\Emb_n(-,Y)$ is representable by a
countable disjoint union $E_n^Y:=\bigsqcup_{d\geqslant 0}E_{n,d}^Y$ of smooth quasi-projective varieties.
\end{prop}

\begin{proof}
Denote by $\Emb_n(U,Y)_d$ the set of couples $(Z,f)$ where $Z$ is a
closed l.c.i. subscheme of $\A^n_U,$ finite of degree $d$ and flat
over $U$. Then $\Emb_n(-,Y)_d$ is a subpresheaf of $\Emb_n(-,Y)$,
and $\Emb_n(U,Y)$ is the disjoint union of $\Emb_n(U,Y)_d$,
$d\geqslant 0$, for any connected $U\in\Sm_k.$

By~\cite[Lemma 5.1.3]{EHKSY} the presheaf $\Emb_n(-,k)_d$ is
represented by a smooth quasi-projective scheme
$Hilb^{lci}_d(\A^n)$. There is the universal finite flat map $W_d\to
Hilb^{lci}_d(\A^n)$. Then the Weil restriction functor
$R_{W_d/Hilb^{lci}_d(\A^n)}(W_d\times_k Y)$ coincides with $\Emb_n(-,Y)_d$
and is represented by a quasi-projective smooth scheme $E_{n,d}^Y$
over $k$ by~\cite[7.6.4-7.6.5]{BLR}.
\end{proof}

The natural inclusions of affine spaces $\A^n\to\A^{n+1}$ induce stabilization maps of
pointed sheaves $\Emb_n(-,Y)\to\Emb_{n+1}(-,Y)$. Denote by $\Emb(-,Y)$ the
pointed sheaf $\Emb(-,Y)=\colim_n\Emb_n(-,Y).$ Note that forgetful maps
$\widetilde{\Fr}_n^{MGL}(-,Y)\to\Emb_n(-,Y)$ are consistent with the stabilization maps.

\begin{cor}\label{embscheme}
The sheaf $\Emb(-,Y)$ is isomorphic to a sequential colimit $E^Y$ of smooth quasi-projective varieties.
\end{cor}

\begin{proof}
This follows from Proposition~\ref{embn} and the fact that $\bigsqcup_{d\geqslant 0}E_{n,d}^Y$
is $\colim_{k\geqslant 0}(E_{n,d_1}^Y\sqcup\cdots\sqcup E_{n,d_k}^Y)$. Hence
$\Emb(-,Y)$ is isomorphic to $E^Y:=\colim_{n,k\geqslant 0}(E_{n,d_1}^Y\sqcup\cdots\sqcup E_{n,d_k}^Y)$.
\end{proof}

We shall give an alternative description of the space
$C_*\Emb(-,Y)$ in terms of $\Omega$-correspon\-dences studied
in~\cite{Omega}.

\begin{dfn}
For $X,Y\in\Sm_k$ denote by $Cor_n^{\Omega}(X,Y)$ the groupoid with
objects given by the set $\Emb_n(X,Y)$ whose morphisms between $(Z_1,f_1)$
and $(Z_2,f_2)$ are isomorphisms $\alpha\colon Z_1\to Z_2$ such that
$\pi_{Z_2}\alpha=\pi_{Z_1}$ and $f_2\alpha=f_1$, where $\pi_{Z_i}$
denotes the projection $\pi_{Z_i}\colon Z_i\to\A^n_X\to X.$ The
assignment $X\mapsto Cor_n^{\Omega}(X,Y)$ defines a presheaf of
groupoids on $\Sm_k.$ There are natural stabilization maps $Cor_n^{\Omega}(-,Y)\to Cor_{n+1}^{\Omega}(-,Y)$
induced by the natural inclusions $\A^n_X\to\A^{n+1}_X.$ Denote by $Cor^{\Omega}(X,Y)$ the colimit
$Cor^{\Omega}(X,Y)=\colim_nCor_n^{\Omega}(X,Y).$
\end{dfn}

\begin{lem}\label{lcilem}
Suppose $f\colon X\to Y$ is a l.c.i. embedding, $g\colon X\to W$ is any regular map
and $W$ is regular. Then the map $(f,g)\colon X\to Y\times W$ is an l.c.i. embedding.
\end{lem}

\begin{proof}
The map $(f,g)$ is the composition $X\xrightarrow{\Gamma_g} X\times W
\xrightarrow{f\times id}Y\times W$. The map $f\times id$ is a l.c.i. embedding.
The graph inclusion $\Gamma_g\colon X\to X\times W$ fits into the pullback diagram
\[
\xymatrix{
X\ar[r]\ar[d]^{\Gamma_g} & W\ar[d]\\
X\times W\ar[r] & W\times W
}
\]
where the right arrow is the diagonal embedding. Since $W$ is regular, the diagonal
map $W\to W\times W$ is a l.c.i. embedding, so the ideal defining $X$ in $X\times W$
is locally generated by $n$ elements, where $n=\dim W.$ Then $\Gamma_g\colon X\to X\times W$
is a l.c.i. embedding by Remark~\ref{remlci}. Then the composition of $\Gamma_g$ and $f\times id$
is an l.c.i. embedding by~\cite[B.7.4]{Fulton}.
\end{proof}

\begin{lem}\label{emb=omega}
Let $NCor^{\Omega}(X,Y)$ be the nerve of the groupoid
$Cor^{\Omega}(X,Y)$. Then for a smooth affine $X$ and $Y\in\Sm_k$
the natural map $f\colon \Emb(X,Y)\to NCor^{\Omega}(X,Y)$ induces a
weak equivalence of simplicial sets
\[C_*f\colon C_*\Emb(X,Y)\to C_*NCor^{\Omega}(X,Y).\]
\end{lem}

\begin{proof}
Note that $C_*NCor^{\Omega}(X,Y)$ is a bisimplicial set with
$m$-simplices given by $C_*N_mCor^{\Omega}(X,Y)$. Thus it is
sufficient to prove that for any $m$ the map
\[C_*f\colon C_*\Emb(X,Y)\to C_*N_mCor^{\Omega}(X,Y)\]
is a weak equivalence of simplicial sets. Note that the map of presheaves
$f\colon\Emb_n(-,Y)\to N_mCor^{\Omega}_n(-,Y)$ is an inclusion admitting a retraction
\[p\colon N_mCor_n^{\Omega}(-,Y)\to\Emb_n(-,Y)\] that sends
$((Z_0,f_0)\xrightarrow{\alpha_0}(Z_1,f_1)\to\cdots\xrightarrow{\alpha_{m-1}}(Z_m,f_m))
\in N_mCor_n^{\Omega}(X,Y)$ to $(Z_0,f_0)\in\Emb_n(X,Y)$.

For every smooth affine $X$ and
$((Z_0,f_0)\xrightarrow{\alpha_0}(Z_1,f_1)\xrightarrow{\alpha_1}\cdots\to (Z_m,f_m))\in N_mCor^{\Omega}_n(X,Y)$
consider the map
\[r_i\colon Z_i\times\A^1\to \A^n_X\times_X\A^n_X\times\A^1, \ \ \ (z,t)\mapsto ((1-t)z+t\beta_i(z),t(t-1)z,t).\]
Here $t$ denotes the coordinate on $\A^1$ and $\beta_i\colon Z_i\to Z_0$ is the
isomorphism of $X$-schemes $\beta_i=\alpha_0^{-1}\circ\ldots\circ\alpha_{i-1}^{-1}.$

Note that $r_i$ is a map of schemes over $X\times\A^1$. 
The map $r_i$, restricted to the fiber over $X\times (\A^1-\{0,1\})$, fits into the diagram
\[\xymatrix{
Z_i\times(\A^1-\{0,1\})\ar[r]^(.4){r_i}\ar[rd]_{r'_i} & \A^n_X\times_X\A^n_X\times(\A^1-\{0,1\})\ar[d]^{(x,y,t)\mapsto (x,\frac{y}{t(t-1)},t)}\\
& \A^n_X\times_X\A^n_X\times(\A^1-\{0,1\}),}\]
where $r'_i\colon(z,t)\mapsto((1-t)z+t\beta_i(z),z,t)$ is a l.c.i. embedding by Lemma~\ref{lcilem}.
The fiber of $r_i$ over $X\times 0$ and $X\times 1$ is a l.c.i. embedding.
Then the map $r_i$ is a l.c.i. embedding by Lemma~\ref{lcicrit}. Let us denote by $Z'_i$ the image of $r_i$. Note that the composition $Z_i\times\A^1\xrightarrow{r_i}Z'_i\subset\A^n_X\times_X\A^n_X\times\A^1\to X\times\A^1$ coincides with $\pi\times id_{\A^1}$, where $\pi$ is the projection $\pi: Z_i\subset\A^n_X\to X$. Therefore $Z'_i$ is finite and flat over $X\times\A^1.$

We construct an $\A^1$-homotopy
\[H\colon N_mCor^{\Omega}_n(-,Y)\to N_mCor^{\Omega}_{2n}(\A^1\times-,Y)\]
as follows. We set
\[H\colon ((Z_0,f_0)\xrightarrow{\alpha_0}(Z_1,f_1)\xrightarrow{\alpha_1}\cdots (Z_m,f_m))
\mapsto
((Z_0',f'_0)\xrightarrow{\gamma_0}(Z'_1,f'_1)\xrightarrow{\gamma_1}\cdots
(Z'_m,f'_m)),\] 
where each isomorphism $\gamma_i$ is given by the composition
\[\gamma_i\colon Z'_i\xrightarrow{r_i^{-1}} Z_i\times\A^1\xrightarrow{\alpha_i\times
id}Z_{i+1}\times\A^1\xrightarrow{r_{i+1}}Z'_{i+1},\] 
and each map $f_i'\colon Z'_i\to Y$ is given by the composition
   \[f'_i\colon Z'_i\xrightarrow{r_i^{-1}} Z_i\times\A^1\xrightarrow{\pi_{Z_i}}Z_i\xrightarrow{f_i} Y.\]

Then $H_0\colon N_mCor^{\Omega}_n(-,Y)\to N_mCor^{\Omega}_{2n}(-,Y)$
is the stabilization map and $H_1\colon N_mCor^{\Omega}_n(-,Y)\to
N_mCor^{\Omega}_{2n}(-,Y)$ equals the composition
\[H_1\colon N_mCor^{\Omega}_n(-,Y)\xrightarrow{p} \Emb_n(-,Y)\xrightarrow{f} N_mCor^{\Omega}_n(-,Y)\xrightarrow{stab} N_mCor^{\Omega}_{2n}(-,Y),\]
where the last arrow is the stabilization map. The $\A^1$-homotopy
$H$ gives rise to a simplicial homotopy
\[H\colon C_*N_mCor^{\Omega}_n(-,Y)\times\Delta[1]\to C_*N_mCor^{\Omega}_{2n}(-,Y).\]
By construction, for any $X$ we have
$H(C_*\Emb_n(X,Y)\times\Delta[1])\subseteq C_*\Emb_{2n}(X,Y)$. Then
by Lemma~\ref{obv} the map $f\colon C_*\Emb(X,Y)\to
C_*N_mCor^{\Omega}(X,Y)$ is a weak equivalence.
\end{proof}

\begin{lem}\label{obv}
Suppose $X_n\subseteq X_{n+1}$ is a directed system of inclusions of
simplicial sets, $Y_n\subseteq  Y_{n+1}$ is a directed system of simplicial subsets
$Y_n\subseteq X_n$, and $p_n\colon X_n\to Y_n$ is a sequence of
retractions that agree with inclusions $X_n\subseteq X_{n+1}$ and
$Y_n\subseteq Y_{n+1}.$ Assume that for every $n$ there is a homotopy
$H(n)\colon X_n\times \Delta[1]\to X_{2n}$ such that $H(n)_0\colon X_n\to X_{2n}$
is the inclusion map, $H(n)(Y_n\times\Delta[1])\subseteq Y_{2n}$, and
the map $H(n)_1\colon X_n\to X_{2n}$ equals the composition
\[X_n\xrightarrow{p_n} Y_n\subseteq Y_{2n}\subseteq X_{2n}.\]
 Then the inclusion $Y\to X$ is a weak equivalence, where $Y=\colim_n Y_n, X=\colim_n X_n$.
\end{lem}

\begin{proof}
Consider a point $y\in Y_n$. The inclusion map $j\colon X_n\to X_{2n}$ and the
composition $f\colon X_n\xrightarrow{p_n} Y_n\subseteq X_{2n}$ are homotopic
by means of the free homotopy $H(n)$. Then the two induced maps
\[\pi_i(j),\pi_i(f)\colon\pi_i(X_n,y)\to\pi_i(X_{2n},y)\]
differ by the action $[\gamma]_*$ of the class $[\gamma]\in\pi_1(X_{2n},y)$ on $\pi_i(X_{2n},y)$,
where $\gamma\colon\Delta[1]\to Y_{2n},\gamma(t)=H(n)(y,t)$ is the
loop given by the image of the base point $y$ under the homotopy $H(n)$. Since the
loop $\gamma$ lies inside $Y_n$, the action of $[\gamma]$ on $\pi_i(Y_{2n},y)$
preserves the image of $\pi_i(Y_n,y)$ under the inclusion map $Y_n\to X_{2n}$. Then the image
\[\pi_i(j)(\pi_i(X_n,y))=[\gamma]_*\pi_i(f)(\pi_i(X_n,y))\]
lies inside the image of $\pi_i(Y_n,y)$. Then
$\pi_i(Y,y)\to\pi_i(X,y)$ is surjective for any point $y\in Y$. The
existence of retractions $p_n$ implies that the map
$\pi_i(Y,y)\to\pi_i(X,y)$ is also injective and for every point
$x\in X_n$ the map $t\mapsto H(n)(x,t)$ gives a path between $x$ and
the point of $Y_n$. We see that $\pi_0(Y)\to\pi_0(X)$ is surjective,
and hence $Y\to X$ is a weak equivalence.
\end{proof}

Note that  for any pointed finite set $K$ the assignment
\[K\mapsto \Emb(X_+\w K)\] defines a sheaf of $\Gamma$-spaces. Denote by $\Emb(X_+\w\Ss)$
the corresponding $S^1$-spectrum. For any $W\in\Sm_k$ there is a canonical map
\[\Emb(-,X)\to\Emb(-\times W,X\times W),\]
functorial in $W\in \Fr_0(k)$.
These two constructions give rise to a $(S^1,\Gm^{\w 1})$-bispectrum
\[(C_*\Emb(X_+\w\Ss), C_*\Emb(X_+\w\Ss\w\Gm^{\w 1}),\ldots).\]
Its structure maps literally repeat the construction of the structure
maps for $K$-motives~\cite[Section~3]{GPmss}. We also define a $(S^1,\Gm^{\w 1})$-bispectrum
\[(C_*NCor^{\Omega}(X_+\wedge\Ss),C_*NCor^{\Omega}(X_+\wedge\Gm^{\wedge 1}\wedge\Ss),\ldots)\]
in a similar fashion.

The following theorem computes $M_{MGL}^{\G}(X)$ as the above two bispectra.

\begin{thm}\label{mglembomega}
For $X\in\Sm_k$ there is a natural levelwise stable local equivalence between
$(S^1,\Gm^{\w 1})$-bispectra $M_{MGL}^{\G}(X)$ and
\[(C_*\Emb(X_+\wedge\Ss),C_*\Emb(X_+\wedge\Gm^{\wedge 1}\wedge\Ss),\ldots)\]
or
\[(C_*NCor^{\Omega}(X_+\wedge\Ss),C_*NCor^{\Omega}(X_+\wedge\Gm^{\wedge 1}\wedge\Ss),\ldots).\]
In particular, the $(S^1,\Gm^{\w 1})$-bispectra
\[(C_*\Emb(X_+\wedge\Ss)_f,C_*\Emb(X_+\wedge\Gm^{\wedge 1}\wedge\Ss)_f,\ldots)\]
and
\[(C_*NCor^{\Omega}(X_+\wedge\Ss)_f,C_*NCor^{\Omega}(X_+\wedge\Gm^{\wedge 1}\wedge\Ss)_f,\ldots)\]
are motivically fibrant and represent the $T$-spectrum $X_+\w MGL$ in
the category of bispectra, where ``$f$'' refers to stable local
fibrant replacements of $S^1$-spectra.
\end{thm}

\begin{proof}
The first claim follows from Proposition~\ref{normfrcor} and
Lemmas~\ref{tilde=A},\ref{A=B},\ref{B=Emb},\ref{emb=omega}. The proof
of Theorem~\ref{bisptilde} shows that the bispectra
\[(C_*\Emb(X_+\wedge\Ss)_f,C_*\Emb(X_+\wedge\Gm^{\wedge 1}\wedge\Ss)_f,\ldots)\]
and
\[(C_*NCor^{\Omega}(X_+\wedge\Ss)_f,C_*NCor^{\Omega}(X_+\wedge\Gm^{\wedge 1}\wedge\Ss)_f,\ldots)\]
are motivically fibrant and represent the $T$-spectrum $X_+\w MGL$ in
the category of bispectra.
\end{proof}

We already know from Corollary~\ref{embscheme} that the sheaf $\Emb(-,Y)$ is isomorphic to a sequential colimit
$E^Y$ of smooth quasi-projective varieties. Thus the $(S^1,\Gm^{\w 1})$-bispectrum
\[(\Emb(X_+\wedge\Ss),\Emb(X_+\wedge\Gm^{\wedge 1}\wedge\Ss),\ldots),\quad X\in\Sm_k,\]
can be presented as the $(S^1,\Gm^{\w 1})$-bispectrum
$(E^{X_+\wedge\Ss},E^{X_+\wedge\Gm^{\wedge 1}\wedge\Ss},\ldots)$. By construction, the $(i,j)$-th
term of the latter bispectrum is a sequential colimit
of simplicial smooth quasi-projective varieties $E^{X_+\wedge\Gm^{\wedge i}\w S^j}$.

By using the preceding theorem, we therefore get the following result:

\begin{thm}\label{reprmgl}
The $(S^1,\Gm^{\w 1})$-bispectrum $M_{MGL}^{\G}(X)$ is isomorphic in $SH(k)$ to the bispectrum
$(E^{X_+\wedge\Ss},E^{X_+\wedge\Gm^{\wedge 1}\wedge\Ss},\ldots)$, each term of which is given by
a sequential colimit of simplicial smooth quasi-projective varieties $E^{X_+\wedge\Gm^{\wedge i}\w S^j}$,
$i,j\geqslant 0$.
\end{thm}

\appendix
\section{Technical lemmas}

In this section we recall standard facts about projective modules
over Henselian pairs. Throughout this section $R$ denotes a
Noetherian $k$-algebra. By~\ref{grind} the set $Gr(n,N)(R)$ equals
the set of rank $n$ admissible submodules of $R^N$. If $R\to S$ is a
map of $k$-algebras, by $P\otimes_R S$ we shall mean the image of
$P$ in $S^N$. It gives an element of $Gr(n,N)(S).$ It is important
to recall from~\cite[Tag 089R]{stacks-project} that $Gr(n,N)(R)$ is
{\it functorial\/} in $R$.

\begin{lem}
Suppose $(R,I)$ is a Henselian pair, and $J$ is an ideal in $R$.
Suppose $B$ is an integral $R$-algebra, and $e'\in B/IB$, $e''\in
B/JB$ are two idempotents that coincide in $B/(I+J)B$. Then there is
an idempotent $e\in B$ such that $e+IB=e'$ and $e+JB=e''.$
\end{lem}

\begin{proof}
By \cite[Tag 09XI]{stacks-project} there is a bijections between
idempotents in $B$ and $B/IB$ as well as there is a bijection
between idempotents in $B/J$ and $B/(I+J)B$. If $e$ is an idempotent
in $B$ such that $e+IB=e'$, then $e+JB=e''$.
\end{proof}

Denote by $Idemp_n(R)$ the set of idempotents of the matrix ring $M_n(R)$.

\begin{lem}\label{idemplift}
Suppose $(R,I)$ is a Henselian pair, $J$ is an ideal in $R$.
Consider a diagram of sets:
\[Idemp_n(R)\to Idemp_n(R/I)\times Idemp_n(R/J)\rightrightarrows Idemp_n(R/(I+J)).\]
Suppose $(x',x'')\in Idemp_n(R/I)\times Idemp_n(R/J)$ and the images
of $x',x''$ coincide in $Idemp_n(R/(I+J))$. Then there is $x\in
Idemp_n(R)$ such that the image of $x$ in $Idemp_n(R/I)$ equals
$x'$, and the image of $x$ in $Idemp(R/J)$  equals $x''$.
\end{lem}

\begin{proof}
There is a right exact sequence of $R$-modules
\[M_n(R)\to M_n(R/I)\oplus M_n(R/J)\to M_n(R/(I+J))\to 0.\]
Take a matrix $y\in M_n(R)$ to be a preimage of $(x',x'')$. Let $f(t)\in R[t]$
be the characteristic polynomial of the matrix $y$. Denote by
$B=R[t]/f(t)$. We follow the proof of~\cite[Tag07M5]{stacks-project}.
Note that $B$ is integral over $R$ and there is a ring map $g\colon B\to M_n(R)$ that
sends $t$ to $y$. For any prime ideal $p$ containing $J$ the image of $f(t)$ in $k(p)[t]$
is the characteristic polynomial of an idempotent matrix, hence it
divides $t^n(t-1)^n$. Then $t^n(1-t)^n\in\sqrt{J}B$ and there exists a constant $N_0$
such that for any $N\geqslant N_0$ the element $t^N+(1-t)^N$ is invertible in $B/JB$.
It follows that $e'':=\frac{t^N}{t^N+(1-t)^N}$ in $B/JB$ is an idempotent and a preimage
of $x''$ in $M_n(R/J)$. Likewise there is a constant $N_1$ such that for any $N>N_1$
the element $e'=\frac{t^N}{t^N+(1-t)^N}$ in $B/IB$ is an idempotent and a preimage of $x'$
in $M_n(R/I)$. Then for $N>max(N_0,N_1)$ the images of $e'$ and $e''$ coincide in
$B/(I+J)B$ and by the previous lemma there is an idempotent $e$ in $B$ lifting $e'$
and $e''$. Then $x=g(e)$ is an idempotent matrix in $M_n(R)$
such that the image of $x$ in $M_n(R/I)$ equals $x'$ and the image of $x$ in $M_n(R/J)$ equals $x''$.
\end{proof}

\begin{lem}\label{projectionlift}
Suppose $(R,I)$ is a Henselian pair, $i\colon P\subseteq R^n$ is an admissible submodule.
Then $i'\colon P\otimes_RR/I\subseteq (R/I)^n$ is an admissible submodule of $(R/I)^n$.
Assume that there is projection $\pi'\colon (R/I)^n\to P\otimes_R R/I$ such that $\pi'i'=id$.
Then there is a projection $\pi\colon R^n\to P$ such that $\pi i=id$ and $\pi\otimes R/I=\pi'$.
\end{lem}

\begin{proof}
Let $e_i$, $i=1,\ldots,n,$ denote the standard basis of $R^n$ and let $\bar{e_i}$ be the
standard basis of $(R/I)^n$. Take a map $p\colon R^n\to P$ sending $e_i$ to some
preimage of $\pi'(\bar{e_i})$. Then $p\circ i$ is an endomorphism of $P$ such that
$(p\circ i)\otimes_R R/I$ is the identity endomorphism of $P\otimes_R R/I.$ Then $p\circ i$
is invertible by Nakayama's lemma. It follows that $\pi=(p\circ i)^{-1}p\colon R^n\to P$
is a projection onto $P$ lifting $\pi'$ and $\pi\circ i=id.$
\end{proof}

\begin{lem}
Suppose $(R,I)$ is a Henselian pair, $J$ is an ideal in $R$. Suppose
$P_1\in Gr(n,N)(R/I)$ and $P_2\in Gr(n,N)(R/J)$ are such that $P_1\otimes_{R/I}(R/I
+J)=P_2\otimes_{R/J}(R/I+J)$ in $Gr(n,N)(R/(I+J))$.
Then there is $P\in Gr(n,N)(R)$ such that $P\otimes_R R/I=P_1$ in $Gr(n,N)(R/I)$
and $P\otimes_R R/J=P_2$ in $Gr(n,N)(R/J).$
\end{lem}

\begin{proof}
Choose a projection $p_1\colon (R/I)^N\to P_1$. Then
$p_1\otimes_{R/I} R/(I+J)$ is a projection onto $P_1\otimes_{R/I} R/(I+J)=P_2\otimes_{R/J}
R/(I+J)$. The pair $(R/J,I/I\cap J)$ is Henselian
by~\cite[Tag 09XK]{stacks-project}, then by Lemma~\ref{projectionlift}
there is a projection $p_2\colon (R/J)^N\to P_2$ that lifts
$p_1\otimes_{R/I} R/(I+J)$. Then $A_1=i_1p_1$ and $A_2=i_2p_2$ are
idempotents that coincide in $Idemp_N(R/(I+J))$.
By Lemma~\ref{idemplift} there is an idempotent $A\in Idemp_N(R)$ such
that $A\otimes_R R/I=A_1$ and $A\otimes R/J=A_2.$ Then $P=A(R^N)$ is an
element of $Gr(n,N)(R)$ such that $P\otimes R/I=P_1$ and $P\otimes R/J=P_2.$
\end{proof}

\begin{lem}\label{projpatching}
Suppose $(R,I)$ is a Henselian pair and $J$ is an ideal in
$R$. Suppose $f_1\colon\Spec R/I\to Gr(n)$ and $f_2\colon \Spec R/J\to Gr(n)$ coincide
on $\Spec R/(I+J)$. Then there is $f\colon\Spec R\to Gr(n)$ that extends $f_1$ and $f_2$.
\end{lem}

\begin{proof}
This follows from the previous lemma and the fact that $Gr(n)(R)=\colim_N Gr(n,N)(R).$
\end{proof}

In the ring $R[\Delta^n]$ denote by $t$ the product $t=t_0\ldots t_n$ of barycentric
coordinates in $R[\Delta^n]$. For any $R[\Delta^n]$-module we denote by $M_t$ the
localization $M_t=M\otimes_{R[\Delta^n]}R[\Delta^n][1/t]$.

\begin{lem}\label{partialglue}
Suppose $\partial B$ is a finite flat $R[\partial\Delta^n]$-algebra,
$M$ is a finitely generated projective $\partial B$-module,
$P\subseteq(\partial B)^N$ is an admissible submodule, and for
every $i=0,\ldots, n$ there is an isomorphism $f_i\colon\partial_iM\to\partial_iP$,
and for every $i,j$ the maps $f_i\otimes_{\partial_iB} \partial_{ij}B$ and
$f_j\otimes_{\partial_jB}\partial_{ij}B$ coincide on $\partial_{ij}M$ (see Definition~\ref{partialdef}).
Then there is an isomorphism $f\colon M\to P$ such that $f\otimes \partial_iB=f_i.$
\end{lem}

\begin{proof}
There is a left exact sequence of $R[\partial\Delta^n]$-modules
\[
0\to R[\partial\Delta^n]\to \oplus_{i=0}^n R[\partial_i\Delta^n]\to \oplus_{i<j}R[\partial_{ij}\Delta^n].
\]
Tensoring it with $\partial B$ over $R[\partial\Delta^n]$,
we get a left exact sequence for every projective $\partial B$-module.
The maps $f_i$ induce a commutative square in the following diagram
\[\xymatrix{
0\ar[r] & M\ar[r] & \oplus_i\partial_iM\ar[r]\ar[d]^{f_i} & \oplus_{i<j}\partial_{ij}M\ar[d]^{f_{ij}}\\
0\ar[r] & P\ar[r] & \oplus_i\partial_iP\ar[r] & \oplus_{i<j}\partial_{ij}P.}\]
Since $\partial B$ is flat $R[\partial\Delta^n]$, both $M$ and $P$ are flat
as $R[\partial\Delta^n]$-modules, then the rows in the diagram are exact,
as they are obtained by tensoring with the left exact sequence above. Then there is a
unique isomorphism $f\colon M\to P$ that makes the diagram commutative.
\end{proof}

\begin{lem}\label{Hn}
Suppose $M$ is a finitely generated projective $R[\Delta^n]$-module. A map
\[f\colon M\to R[\Delta^n]^N\]
is an admissible embedding if and only if its restriction to the boundary
\[f\otimes R[\partial\Delta^{n}]\colon M\otimes_{R[\Delta^n]}R[\partial\Delta^{n}]\to R[\partial\Delta^{n}]^N\]
is an admissible embedding and the localized map\[f_t\colon
M_t\to R[\Delta^n]_t^N\] is an admissible embedding.
\end{lem}

\begin{proof}
Let $K$ and $C$ denote the kernel and cokernel of $f$ respectively.
We need to check that $K=0$ and $C$ is projective. Note that
$K$ is a submodule of a free finite rank $R[\Delta^n]$-module. Since $t$ is not a zero
divisor in $R[\Delta^n]$, then the localization map $K\to K_t$ is injective and $K_t=\ker(f_t)=0,$ hence $K=0$.
Let $r$ denote the rank of $M$. For every maximal ideal $m$ of
$R[\Delta^n]$ if $t\notin m$ then $C_m$ is a localization of $C_t$,
hence it is a free module of rank $N-r$. If $t\in m$, then $C/mC$ is a
free module of rank $N-r$. By Nakayama's lemma there is a
surjection $g\colon R[\Delta^n]^{N-r}_m\to C_m$ of modules over the local ring
$R[\Delta^n]_m$. Then localization $(C_m)_t$ is a localization of the projective
module $C_t$ of rank $N-r$. Then $(C_m)_t$ is projective of rank $N-r$.
Since $g_t$ is a surjective map between projective modules of the same rank,
then it is an isomorphism, and so $Ker(g)_t=0$. Then $Ker(g)=0$, since $Ker(g)$ is a
submodule of the free module $R[\Delta^n]^{N-r}_m$, and $t$ is not a zero divisor of $R[\Delta^n].$
\end{proof}

\begin{lem}\label{borderextend}
Suppose $M$ is a finitely generated projective $R[\Delta^n]$-module.
Assume that there is an admissible embedding $f'\colon M\otimes_{R[\Delta^n]}R[\partial\Delta^n]\to R[\partial\Delta^n]^N$.
Then there is a number $d$ and an admissible embedding $f\colon M\to R[\Delta^n]^N\oplus R[\Delta^n]^d$ such that the map
\[f\otimes R[\partial\Delta^n]\colon M\otimes_{R[\Delta^n]}R[\partial\Delta^n]\to
R[\partial\Delta^n]^N\oplus R[\partial\Delta^n]^d\] equals the composition of $f'$ and
the standard embedding $R[\partial\Delta^n]^N\to R[\partial\Delta^n]^N\oplus R[\partial\Delta^n]^d$.
\end{lem}

\begin{proof}
Consider some admissible embedding $j\colon M\to R[\Delta^n]^d$ and some
projection $p\colon R[\Delta^n]^d\to M$ such that $p\circ j=id_M$. Let $e_1,\ldots, e_d$
denote the standard basis of $R[\Delta^n]^d$ and let $\bar{e}_1,\ldots, \bar{e}_d$ be the
standard basis of $R[\partial\Delta^n]^d.$ Consider the composition
\[R[\partial\Delta^n]^d\xrightarrow{p\otimes id} M\otimes_{R[\Delta^n]} R[\partial\Delta^n]\xrightarrow{f'} R[\partial\Delta^n]^N.\]
For $i=1,\ldots, d$ take $x_i\in R[\Delta^n]^N$ to be any preimage
of $f'((p\otimes id)(\bar{e}_i))\in R[\partial\Delta^n]^N.$ Then
there is a homomorphism
\[F\colon R[\Delta^n]^d\to R[\Delta^n]^N\oplus R[\Delta^n]^d,\quad e_i\mapsto (x_i,(t_0t_1\ldots t_n)e_i),\]
where $t_0,\ldots, t_n$ denote the coordinates in the ring $R[\Delta^n]$.
Take $f\colon M\to R[\Delta^n]^N\oplus R[\Delta^n]^d$ to be the composition
$f=F\circ j.$ Let us check that $f$ is an admissible embedding.

Note that $f\otimes_{R[\Delta^n]}R[\partial\Delta^n]$ is the composition of $f'$ and the
standard embedding $R[\partial\Delta^n]^N\to R[\partial\Delta^n]^N\oplus R[\partial\Delta^n]^d$.
In particular, $f\otimes_{R[\Delta^n]}R[\partial\Delta^n]$ is an admissible embedding.

The localization $f_t\colon M_t\to R[\Delta^n]^{N+d}_t$ is the composition of $F_t\circ j_t$
and $F_t$ fits into a commutative triangle
\[
\xymatrix{
R[\Delta^n]_t^d\ar[r]^-{g\oplus id}\ar[rd]_{F_t} & R[\Delta^n]^N_t\oplus R[\Delta^n]_t^d\\
 & R[\Delta^n]_t^N\oplus R[\Delta^n]_t^d\ar[u]_{id\oplus \frac{1}{t}id}
}
\]
where $g\colon R[\Delta^n]_t^d\to R[\Delta^n]^N_t$ is the map that sends $e_i$ to $x_i$.
The right arrow of the triangle is an isomorphism and $g\oplus id$ is an admissible embedding.
Then $F_t$ is an admissible embedding, and hence so is $f_t$.
By Lemma~\ref{Hn} $f$ is an admissible embedding.
\end{proof}

\begin{lem}\label{lcicrit}
Suppose $X$ is an affine variety over $k$, $A$ and $Y$ are
equidimensional flat affine $X$-schemes, $A\to X$ is finite, $Y$ is
Cohen--Macaulay, and $f\colon A\to Y$ is a morphism over $X$.
Suppose $Z$ is a closed subset of $X$ and the map on the fiber
products $f_Z\colon A\times_X Z\to Y\times_X Z$ and
$A\times_X(X-Z)\to Y\times_X(X-Z)$ are l.c.i. embeddings. Then $f$
is an l.c.i. embedding.
\end{lem}

\begin{proof}
Denote by $n=\dim Y-\dim A$ and let $A_Z$ (resp. $Y_Z$, $A_{X-Z}$,
$Y_{X-Z}$) be the fiber product $A\times_X Z$ (respectively $Y\times_XZ$,
$A\times_X(X-Z)$, $Y\times_X(X-Z)$). Let us check that $k[Y]\to
k[A]$ is surjective. For every point $x\in X$ if $x\in X-Z$, then
the localization map $k[Y]_x\to k[A]_x$ is surjective. If $x\in Z$, then
the map $k[Y]\otimes_{k[X]}k(x)\to k[A]\otimes_{k[X]}k(x)$ is
surjective. It follows from Nakayama's lemma that the map on localizations
$k[Y]_x\to k[A]_x$ is surjective. Then $k[Y]\to k[A]$ is surjective,
hence $A\to Y$ is a closed embedding. Let $I$ denote the kernel of
$k[Y]\to k[A]$. For every point $y\in Y$ if $y$ is in
$Y_{X-Z}$, then $I_y$ is generated by a regular sequence of length
$n$. If $y$ is in $Y_Z$, the sequence
\[0\to I_y\otimes_{k[X]}k[Z]\to k[Y_Z]_y\to k[A_Z]_y\to 0\]
is exact, because $k[A]$ is flat over $k[X]$. Then
$I_y\otimes_{k[X]}k[Z]$ is generated by $n$ elements over
$k[Y_Z]_y$, hence $I_y\otimes_{k[Y]_y}k(y)$ is generated by $n$
elements. By Nakayama's lemma $I_y$ is generated by $n$
elements. Since $A$ has codimension $n$ in $Y$, these elements form
a regular sequence~\cite[III.4.5]{AK}. Then $A$ is an l.c.i.
subscheme in $Y$.
\end{proof}

\subsection*{Acknowledgements} The authors thank Marc Hoyois and Ivan Panin for helpful discussions.

\bibliographystyle{plain}

\begin{thebibliography}{10}

\bibitem{AK} A. Altman, S. Kleiman, \emph{Introduction to Grothendieck Duality Theory}, Lecture Notes in
             Mathematics, Vol.~146, Springer-Verlag, Berlin--New~York, 1970.

\bibitem{AN} A. Ananyevskiy, A. Neshitov, Framed and MW-transfers for homotopy
               modules, arXiv:1710.07412, {\em Sel. Math. New Ser.} 25 (2019), article 26.

\bibitem{AGP} A. Ananyevskiy, G.~Garkusha, I.~Panin, Cancellation theorem for
             framed motives of algebraic varieties, arXiv:1601.06642, \emph{Adv. Math.} 383 (2021), article 107681.

\bibitem{BLR} S. Bosch, W. L\"utkebohmert, M. Raynaud, \emph{N\'eron Models},
             Ergebnisse der Mathematik und ihrer Grenzgebiete, Vol.~21, Springer-Verlag, Berlin, 1990.
             
\bibitem{DKO} A. Druzhinin, H. Kolderup, P.~A. \O stv\ae r, Strict $\mathbb A^1$-invariance over the integers,
preprint arXiv:2012.07365.

\bibitem{DP} A.~Druzhinin, I.~Panin, Surjectivity of the etale excision map
             for homotopy invariant framed presheaves, arXiv:1808.07765,
             \emph{Proc. Steklov Inst. Math.} 320 (2023), 91--114.

\bibitem{EHKSY} E. Elmanto, M. Hoyois, A. Khan, V. Sosnilo, M. Yakerson, Motivic infinite loop spaces,
arXiv:1711.05248, \emph{Cambridge J. Math.} 9(2) (2021), 431--549.

\bibitem{EHKSY1} E. Elmanto, M. Hoyois, A. Khan, V. Sosnilo, M. Yakerson,
               Modules over algebraic cobordism, arXiv:1908.02162, \emph{Forum Math.
               Pi} 8:e14 (2020), 1--44.

\bibitem{Fulton} W. Fulton, \emph{Intersection theory}, Springer-Verlag, Berlin-Heidelberg, 1984.

\bibitem{GNP} G. Garkusha, A. Neshitov, I. Panin, Framed motives of relative motivic spheres,
               arXiv:1604.02732, \emph{Trans. Amer. Math. Soc.} 374(7) (2021), 5131--5161.

\bibitem{GPmss} G. Garkusha, I. Panin, On the motivic spectral sequence,
               arXiv:1210.2242, \emph{J. Inst. Math. Jussieu\/} 17(1) (2018), 137--170.

\bibitem{GPPresheaves} G. Garkusha, I. Panin, Homotopy invariant presheaves with framed transfers,
             arXiv:1504.00884, \emph{Cambridge J. Math.} 8(1) (2020), 1--94.

\bibitem{GPMain} G. Garkusha, I. Panin, Framed motives of algebraic varieties
             (after V. Voevosky), arXiv:1409.4372, \emph{J. Amer. Math. Soc.} 34(1) (2021), 261--313.

\bibitem{GP5} G. Garkusha, I. Panin, The triangulated categories of framed bispectra and framed motives,
             arXiv:1809.08006, \emph{Algebra i Analiz\/} 34(6) (2022), 135--169.

\bibitem{GPO} G.~Garkusha, I.~Panin, P.~A. \O stv\ae r, Framed motivic $\Gamma$-spaces, arXiv:1907.00433,
             \emph{Izv. Math.} 87(1) (2023), 3--32.

\bibitem{H} {M.~Hovey}, Spectra and symmetric spectra in general model categories,
             {\em J. Pure Appl. Algebra\/} 165(1) (2001), 63--127.


\bibitem{Is} D. Isaksen, Flasque model structures for simplicial presheaves, {\em K-Theory\/} 36 (2005), 371--395.

\bibitem{JardineMSS} J.F. Jardine, Motivic symmetric spectra, {\em Doc. Math.} 5 (2000), 445--552.

\bibitem{Lev} M. Levine, A comparison of motivic and classical stable homotopy theories, \emph{J. Topology\/} 7 (2014), 327--362.

\bibitem{MVW} C. Mazza, V. Voevodsky, C. Weibel, {\em Lecture notes on motivic cohomology}.
Clay Mathematics Monographs, 2. American Mathematical Society,
Providence, Cambridge, MA, 2006.

\bibitem{Omega} A. Neshitov, Rigidity theorem for presheaves with $\Omega$-transfers,
         \emph{Algebra i Analiz\/} 26(6) (2014), 78--98. English transl. in
         \emph{St. Petersburg Math. J.} 26(6) (2015), 919--932.

\bibitem{PPR} I. Panin, K. Pimenov, O. R\"ondigs, A universality theorem for Voevodsky's algebraic cobordism
spectrum, {\em Homology, Homotopy Appl.} 10(2) (2008), 211--226.

\bibitem{PPR1} I. Panin, K. Pimenov, O. R\"ondigs, On Voevodsky's algebraic K-theory
         spectrum, {\em Abel. Symp. Proc.} 4 (2009), 279--330.

\bibitem{PW} I. Panin, C. Walter, On the algebraic cobordism spectra $MSL$ and $MSp$, arXiv:1011.0651,
\emph{Algebra i Analiz\/} 34(1) (2022), 144--187.

\bibitem{Pontr} L.~S. Pontrjagin, Smooth manifolds and their applications in homotopy theory,
           Tr. Mat. Inst. Steklova 45 (1955), 1--139. (Russian). English transl. in AMS translations, ser. 2, 11,
           1--114, AMS, Providence, RI, 1959.

\bibitem{PraRa} G. Prasad, M. S. Raghunathan, On the Kneser--Tits problem,
         \emph{Comment. Math. Helv.} 60 (1985), 107--12l.

\bibitem{Qui} D. Quillen, On the formal group laws of unoriented and complex cobordism theory, \emph{Bull. Amer. Math. Soc.}
75(6) (1969), 1293--1298.

\bibitem{S} {G. Segal}, Categories and cohomology theories, \emph{Topology} 13 (1974), 293--312.

\bibitem{stacks-project} The Stacks Project Authors, {\em{Stacks Project}}, http://stacks.math.columbia.edu, 2018.

\bibitem{SV96} A. Suslin, V. Voevodsky, Singular homology of abstract algebraic varieties, 
\emph{Invent. Math.} {123} (1996), 61--94.


\bibitem{VoevodskyICM} V. Voevodsky, $\bold A^1$-homotopy theory,
Proceedings of the International Congress of Mathematicians, Vol. I
(Berlin, 1998), Doc. Math. 1998, Extra Vol. I, 579--604.

\bibitem{Voe2} V.~Voevodsky, Notes on framed correspondences, unpublished, 2001.
         Also available at math.ias.edu/vladimir/files/framed.pdf
\end{thebibliography}

\end{document}